\documentclass[a4paper,11pt,reqno]{amsart}
\usepackage{mhequ}
\usepackage{mhsymb}
\usepackage{amssymb}
\usepackage{times}

\usepackage[]{hyperref}

\let\emptyset \undefined
\newsymbol\emptyset    203F

\let\cal\mathcal
\definecolor{darkred}{rgb}{0.9,0.1,0.1}

\def\d{\partial}

\theoremstyle{plain}
\newtheorem{theorem}{Theorem}[section]
\newtheorem{corollary}[theorem]{Corollary}
\newtheorem{lemma}[theorem]{Lemma}
\newtheorem{proposition}[theorem]{Proposition}

\newtheorem{definition}[theorem]{Definition}
\newtheorem{assumption}[theorem]{Assumption}

\theoremstyle{remark}
\newtheorem{remark}[theorem]{Remark}

\numberwithin{equation}{section}

\DeclareMathOperator{\Id}{Id}

\def\E{{\mathbb E}}
\def\P{{\mathbb P}}

\newcommand{\ga}{\gamma}

\newcommand{\cC}{\mathcal{C}}
\newcommand{\bB}{\mathcal{B}}
\newcommand{\dD}{\mathcal{D}}

\renewcommand{\eqdef}{:=}

\newcommand{\bN}{\bB^{0}}
\newcommand{\cN}{\cC^{0}}

\renewcommand{\hat}{\widehat}
\newcommand{\ot}{\otimes}			
\newcommand{\ls}{\lesssim}
\newcommand{\dd}{\,{d}}
\renewcommand{\CE}{\mathcal{D}}

\newcommand{\aX}{\alpha_\star}

\newcommand{\et}{\eta}

\newcommand{\ta}{{\widetilde{\alpha}}}

\newpage

\renewcommand{\CB}{\mathcal{B}}

\def\BV{\mathrm{BV}}

\newcommand{\tF}{\overline{F}}
\renewcommand{\Z}{\mathbb{Z}}
\renewcommand{\R}{\mathbb{R}}

\def\${|\!|\!|}
\newcommand{\XX}{\mathbf{X}}

\newcommand{\ue}{u_\eps}

\newcommand{\dFs}{F_{s,t} }

\newcommand{\dRs}{ R_{s,t}}

\newcommand{\MM}{\mathcal{M}}
\newcommand{\KK}{\mathcal{K}}

\newcommand{\tbp}{\tilde{\alpha}}
\newcommand{\tbm}{\tilde{\alpha}}

\def\${|\!|\!|} 
\def\deltam{\delta_{s,t} m_{\eps}}
\def\deltamg{\delta_{s,t} m_{\eps}^{\gamma}}
\def\deltaa{\delta_\eps m_{t}}
\def\deltaag{\delta_\eps m_{t}^{\gamma}}

\newcommand{\cD}{\mathcal{D}}

\newcommand{\Ru}{\mathcal{R}_u}
\newcommand{\Rue}{\mathcal{R}_{u_\eps}}
\newcommand{\Rub}{\mathcal{R}_{\bar{u}}}

\newcommand{\BB}{\mathcal{F}}
\newcommand{\BBe}{\BB_\eps}

\newcommand{\ve}{v_\eps}

\newcommand{\XE}{\Xi_\eps}

\newcommand{\Pt}{\Psi^\theta}
\newcommand{\Pte}{\Psi^\theta_\eps}
\newcommand{\tPte}{\tilde{\Psi}^\theta_\eps}
\newcommand{\tPt}{\tilde{\Psi}^\theta}
\newcommand{\Rt}{R^\theta}
\newcommand{\tRt}{\tilde{R}^\theta}
\newcommand{\Rte}{R^\theta_\eps}
\newcommand{\tRte}{\tilde{R}^\theta_\eps}
\newcommand{\phie}{\phi_\eps}
\newcommand{\mt}{m_{t}}
\newcommand{\met}{m_{\eps,t}}

\def\iiint{\rlap{\hspace{0.43em}\raisebox{0.2em}{\rule{0.5em}{0.1em}}}}
\def\siint{\rlap{\hspace{0.36em}\raisebox{0.26em}{\rule{0.3em}{0.06em}}}}
\def\iint{\mathchoice{\iiint}{\siint}{\siint}{\siint}\int}

\allowdisplaybreaks

\begin{document}

\title[]
{Approximating rough stochastic PDEs}

\author{Martin Hairer}
\address{Martin Hairer, The University of Warwick, Coventry CV4 7AL, UK} 
\email{M.Hairer@Warwick.ac.uk}

\author{Jan Maas}
\address{Jan Maas,
Institute for Applied Mathematics\\
University of Bonn\\
Endenicher Allee 60\\
53115 Bonn\\
Germany} 
\email{maas@uni-bonn.de}

\author{Hendrik Weber}
\address{Hendrik Weber, The University of Warwick, Coventry CV4 7AL, UK} 
\email{hendrik.weber@warwick.ac.uk}

\thanks{We are very grateful for the referee's careful reading of the original manuscript and 
for pointing out a mistake in our previous proof of Lemma~\ref{le:mistakecorrected}.\\
\rlap{\phantom{A}}\hspace\parindent JM is supported by Rubicon grant 680-50-0901
of the Netherlands Organisation for Scientific Research (NWO).
MH is supported by EPSRC grant EP/D071593/1 and by the 
Royal Society through a Wolfson Research Merit Award. Both MH and HW are supported by the Leverhulme Trust through a Philip Leverhulme Prize. 
}

%\keywords{}

%\subjclass[2000]{Primary 60H07; Secondary: 35J15, 35K90, 35R15, 47A60, 47B44, 47D05, 47F05, 60H30, 60G15}

 \begin{abstract}
We study approximations to a class of vector-valued equations of Burgers type driven by a multiplicative space-time white noise. A solution theory for this class of equations has been developed recently in [Hairer, Weber,  \emph{Probab. Theory Related Fields}, 2013]. The key idea was to use the theory of \emph{controlled rough paths} to give definitions of weak / mild solutions and to set up a Picard iteration argument.

In this article the limiting behaviour of a rather large class of (spatial) approximations to these equations is studied. These approximations are shown to converge and convergence rates are given, but the limit may depend on the particular choice of approximation. This effect is a spatial analogue to the It\^o-Stratonovich correction in the theory of stochastic ordinary differential equations, where it is well known that different approximation schemes may converge to different solutions.   
 \end{abstract}

%\dedicatory{}

\date\today

\maketitle
\thispagestyle{empty}

\section{Introduction}
The aim of the present paper is to study approximations to vector-valued stochastic Burgers-like equations with multiplicative noise. These equations are of the form 
\begin{equ}[e:SPDE]
 \partial_t {u} =   \nu \,  \partial_x^2 u 
 		+ F(u) + G(u) \partial_x u    +\theta(u) \, \xi  \;,
\end{equ}
where the function $u=u(t,x;\omega) \in \R^n$ is  vector-valued. We assume that the  functions $F\colon \R^n \to \R^n$ and $G, \theta \colon \R^n \to \R^{n\times n}$ are smooth and the products in the terms $G(u) \partial_x u$ as well as in $\theta(u) \xi$ are to be interpreted as matrix vector multiplication. The noise term $\xi$ denotes an $\R^n$-valued space-time white noise and the multiplication should be interpreted in the sense of It\^o integration
against an $L^2$-cylindrical Wiener process.

In the case $G=0$, approximations to \eref{e:SPDE} have been very well studied: we refer to \cite{Istvan1,Istvan2,MR1803132} for some of the earlier results in this direction. 
For non-zero $G$, there is a clear distinction between
the \emph{gradient case}, where $G = \nabla \mathcal{G}$ for some sufficiently regular 
function $\mathcal{G}\colon \R^n \to \R^n$ (so that $\d_t u = G(u)\d_x u$ would describe a system 
of conservation laws), and the general case.
In the gradient case, existence and uniqueness for \eref{e:SPDE} has been known at least since the nineties  \cite{dPDT94,Gyo98} and convergence results for numerical schemes have, for example, been obtained in  \cite{IstvanBurgers, DirkArnulf}. 

The emphasis of the present article is on the general, non-gradient, case.
A satisfactory solution theory for the general case is much more 
involved than the gradient case and has been given only very recently 
\cite{Hai10,HW10}. The difficulty in treating \eref{e:SPDE} lies in the 
lack of spatial regularity of its solutions. In fact, it follows from the results in \cite{HW10} that solutions to \eref{e:SPDE} take values in $\cC^\alpha$ for any $\alpha<  \frac{1}{2}$ but not for $\alpha =  \frac{1}{2}$. Unfortunately, it turns out that the pairing 
\begin{equ}[e:bilin]
\cC^\alpha \times \cC^\alpha \ni (u,v) \mapsto u \, \partial_x  v
\end{equ}
is well-defined if and only if $\alpha > \frac{1}{2}$. Even worse: there exists no ``reasonable'' Banach space
$\CB$ containing the solutions to the linearised version of \eref{e:SPDE} and such that \eref{e:bilin} extends to
a continuous bilinear map from $\CB \times \CB$ into the space of Schwartz distributions, see for
example \cite{TerryPath} and \cite{LCL07}. As a consequence, it is not clear at all a priori how to interpret the term $G(u)\partial_xu$ in \eref{e:SPDE} and the classical approach to the construction 
of mild solutions fails. 

In all of the above mentioned references on the gradient case, this issue is resolved by exploiting the conservation law structure of the nonlinearity. This means that the chain rule  is \emph{postulated}  and the nonlinearity is  rewritten as 
\begin{equs}[e:nonlin]
G\big(u(t,x)\big) \, \partial_x u(t,x)  \, = \, \partial_x \mathcal{G} \big( u(t,x) \big)\;,
\end{equs}
which makes sense as a distribution as soon as $u$ is continuous.
The approximation schemes studied e.g. in \cite{IstvanBurgers, DirkArnulf} respect this conservation law structure by considering natural approximations of $\partial_x \mathcal{G}$. For example, it is not difficult to show that
if $\ue$  solves
\begin{equ}[e:approx1]
 \partial_t{u}_\eps =   \nu \,  \partial_x^2 u_\eps 
 		+ F(u_\eps) + \frac{1}{\eps} \Big( \mathcal{G} \big( u_\eps(t,x+ \eps) \big)  - \mathcal{G} \big( u_\eps(t,x) \big) \Big)    +\theta(u_\eps) \, \xi  \; ,
\end{equ}
then $\ue$ converges to $u$ for $\eps \downarrow 0$. Similarly, full finite difference / element approximations
 also converge.

In the \emph{non-gradient} case, i.e.\ when such a function $\mathcal{G}$ does not exist, this approach does not work. The key idea developed in \cite{Hai10,HW10} to overcome this difficulty is  the following: in order to define the product $G\big(u(t,x) \big) \, \partial_x u(t,x)$ as a distribution, it has to be tested against a smooth test function $\varphi$. This expression takes the form
\begin{equs}[e:nonlin2]
\int_{-\pi}^\pi \varphi(x) \,  G\big(u(t,x) \big) \, \partial_x u(t,x) \,  dx = \, \int_{-\pi}^\pi \varphi(x) \,  G\big(u(t,x) \big) \, d_x u(t,x). 
\end{equs}
The fact that we expect $u$ to behave like a Brownian motion as a function of the space variable $x$ suggests that one should interpret this expression as a kind of stochastic integral. In particular, a stochastic integration theory is needed to capture stochastic cancellations. It turns out that the theory of \emph{controlled rough paths}  \cite{Ly98, LQ02, LCL07,Gu04, FV10,GT10}  provides a suitable way to deal with spatial stochastic 
integrals like \eref{e:nonlin2}. Using this idea,  a concept of solutions is given in \cite{Hai10,HW10}. These solutions exist and are unique up to a choice of \emph{iterated integral} which corresponds to the choice of the integral of $u$ against 
itself. This is a situation analogous to the choice between It\^o and Stratonovich integral that is familiar from the
classical theory of SDEs.

Even in the gradient case, effects of this non-uniqueness can be observed. 
A posteriori, this is not surprising: after some reflection, 
it clearly appears that postulating the chain rule \eref{e:nonlin} is a rather bold step to take! Indeed, 
we have just seen that 
the expression \eref{e:nonlin2} is akin to a stochastic integral, and we know very well that the
usual chain rule only holds if such an integral is interpreted in the Stratonovich sense, while it fails if it
is interpreted in the It\^o sense.
In \cite{HM10} approximations to \eref{e:SPDE} are studied in the special case $G = \nabla \mathcal{G}$ 
when the noise is additive, i.e.\ if $\theta(u) = 1$. For a whole class of different natural approximation schemes, convergence to a stochastic process $\bar{u}$ is shown.  The main difference with previous works is that in \cite{HM10}, natural discretisations of $G(u)\,\d_x u$ instead of natural
discretisations of $\d_x \CG(u)$ are considered. A typical example of the type of discretisation for the nonlinearity considered there is
\begin{equs}[e:approx2]
G\big( u(t,x) \big) \, \frac{1}{\eps} \big(  u(t,x+\eps) - u(t,x)\big).
\end{equs}

In general, the limiting process $\bar{u} = \lim_{\eps \to 0} u_\eps$ turns out \emph{not} to be
a solution of \eref{e:SPDE} in the classical sense. Instead, it solves a similar equation with an additional reaction term. This extra term  depends on the specific choice of approximation and it can be calculated explicitly. 
As noted in \cite{HM10}, this additional term is exactly the correction that appears when changing to 
a different stochastic integral.

In the present work, these approximation results are extended to the \emph{non-gradient} case with 
\emph{multiplicative} noise. We study a wide class of approximations (essentially the same as in 
\cite{HM10} but with slightly different technical assumptions)  and extend the convergence result to the general case.  Unsurprisingly, the techniques we use are quite different from \cite{HM10}, since the notion of solution
for the limiting object is completely different. We make full use of the machinery developed in \cite{Hai10,HW10} and we develop a method to 
include approximations to rough integrals.  
In particular, we do obtain an explicit rate of convergence of the order $\eps^{{1\over 6}-\kappa}$ for
$\kappa$ arbitrarily small. 

There are several motivations for this work: Equation \eref{e:SPDE} appears, for example, in the path sampling algorithm introduced in \cite{HSV07}  (see  also \cite{Hai10}). So far, the fact that the limit depends on the specific choice of approximation scheme had been shown only in the gradient case with additive noise. In this work we complete the picture by showing that the same effect can be observed in the general case and we obtain an expression for the correction
term that arises.

Another main motivation is to illustrate how the rough path machinery can be used to obtain concrete
approximation results, including convergence rates. 
This is particularly interesting, as  similar techniques were recently used in \cite{Hai11} to give a solution theory for the KPZ equation \cite{KPZ}
\begin{equs}
\partial_t{h} \, = \, \partial_{x}^2 h + \lambda (\partial_x h)^2 - \infty +\xi\;,
\end{equs}
where $\xi$ denotes space-time white noise and ``$\infty$'' denotes an ``infinite constant'' that needs to be 
subtracted in order to make sense of the diverging term $(\partial_x h)^2$.
This equation is a popular model for surface growth (see e.g.\ \cite{C11} and the references therein). It is conjectured that a large class of microscopic surface growth models (e.g.\ the lattice KPZ equation \cite{SaSp09}
and variations on the weakly asymmetric simple exclusion process \cite{Brazilian,Sigurd}), 
converge to $h$ in suitable scaling limits,
but so far this has only been shown for the weakly asymmetric simple exclusion process \cite{BG97}. 

The present article provides a case study illustrating how one can obtain approximation results for  
a class of equations exhibiting similar features to those of the KPZ equation (see \cite[Section~4]{Hai11}).
In this sense, the present work is really a ``proof of concept'' that lays the foundations for 
further analytical investigations into the universality of the KPZ equation. Notice that although the KPZ equation has 
additive noise, the construction in \cite{Hai11} yields an equation that is very close to the case of multiplicative noise treated here.   

\subsection{Framework and main result}

For $\eps > 0$ we consider a class of approximating stochastic PDEs given by
\begin{equs}[e:SPDEeps]
 d u_\eps &=  \Big( \nu \Delta_\eps u_\eps + F(u_\eps)
 		+ G(u_\eps) D_\eps u_\eps \Big) dt  + \theta(u_\eps) \,  H_\eps dW\\
		u_\eps(0) &= u_\eps^0\;.
\end{equs}
Here, as usual,  we have replaced the formal $\xi$ with the stochastic differential of a cylindrical Brownian motion $W$ on $L^2$. The integral against $dW$ should furthermore be interpreted in the It\^o sense.
For simplicity, we assume that $x$ takes values in $[-\pi,\pi]$ and we endow \eref{e:SPDEeps} with periodic boundary conditions. We do not expect our results to significantly depend on this choice. Throughout the paper we will assume that $F \in \cC^1, G \in \cC^3,$ and $\theta \in \cC^2$.

The operators $\Delta_\eps$, $D_\eps$, and $H_\eps$ appearing in \eref{e:SPDEeps} are Fourier multipliers 
providing approximations to $\partial_x^2$, $\partial_x$ and the identity respectively. In terms of their action
in Fourier space, they are given by 
\minilab{e:defApprox}
\begin{equs}
\widehat{\Delta_\eps u}(k) &= - k^2 f(\eps k)\hat u(k)\;, \label{e:DeltaEps}\\
\widehat{D_\eps u}(k) &= ik g(\eps k)\hat u(k)\;, \label{e:Deps}\\
\widehat{ H_\eps W}(k) &= h(\eps k) \hat W(k)\;.\label{Heps}
\end{equs}
Throughout the paper we will make some standing assumptions on the cut-off functions $f$, $g$ and $h$. 

\begin{assumption}\label{a:f}
The function $f : \R \to (0,+\infty]$ is  even, satisfies $f(0) =1$,  and 
is continuously differentiable on an interval $[-\delta,\delta]$ around $0$.
Furthermore, there exists $c_f \in (0,1)$ such that $f(k) \ge  c_f$ for all $k > 0$. 
\end{assumption}
Besides this weak regularity assumption of $f$ near the origin, we also need a global bound on its oscillations.
In order to state this bound, we introduce the family of functions 
\begin{equ}
b_t(k) = \exp\big(- k^2  f(k)   \, t\big)\;.
\end{equ}
With this notation at hand, we assume that

\begin{assumption}\label{a:BV}
The functions $b_{t}$ are uniformly bounded in the bounded variation norm: 
\begin{equ}
\sup_{t > 0} \big| b_{t} \big|_{\BV} < \infty\;.
\end{equ}
\end{assumption}

We make the following assumption on the approximation of the spatial derivative.

\begin{assumption}\label{a:D}
There exists a signed Borel measure $\mu$ such that
\begin{equ}
\int_\R e^{ikx}\,\mu({d} x) = ik\,g(k)\;,
\end{equ}
and such that 
\begin{align} \label{eq:moments}
 \mu(\R) = 0, \quad 
 |\mu|(\R) < \infty, \quad 
 \int_{\R} x \, \mu({d}x) = 1, \quad
 \int_{\R} |x|^\frac{5}{2} \, |\mu|({d}x) < \infty\;.
\end{align}
In particular, we have $ (D_\eps u)(x) :=  \frac{1}{\eps} \int_\R u(x + \eps y) \, \mu(dy)$, where we identify $u\colon [-\pi, \pi] \to  \R$ with its periodic extension on all of $\R$. 
\end{assumption}
Note that the case $D_\eps u(x)=  \frac{1}{\eps}( u(x +\eps) - u(x))$ mentioned in \eref{e:approx1}  and \eref{e:approx2}  is included as special case $\mu = \delta_1- \delta_0$.
Finally, we make the following assumption on the approximation of the noise.

\begin{assumption}\label{a:h}
The function $h$ is even, bounded and so that $h^2/f$ is of bounded variation.  
Furthermore, $h$ is twice differentiable at the origin with $h(0) = 1$ and $h'(0) = 0$.
\end{assumption}

Note that the assumptions on $g$ and $h$ are identical to those imposed in \cite{HM10}.
Regarding the function $f$, Assumption~\ref{a:f} is actually weaker than the corresponding assumption in
\cite{HM10}. However, we require the additional Assumption~\ref{a:BV}.  This assumption is not too restrictive and in particular all the examples discussed in \cite{HM10} satisfy it. See Remark~\ref{rem:Holder} below for
the main reason why this additional assumption is required. Note that the assumptions on $f$ do not imply that the approximated heat semigroup $S_\eps(t) := e^{t \Delta_\eps}$ is continuous at $0$ in the space of continuous functions. This is natural in the context of numerical approximations, since these would
always involve the projection onto a finite-dimensional subspace.
See Subsections \ref{ss:summary} and \ref{ss:proof} below for a more detailed discussion of this point.

Let $\bar u$ be the solution of the equation
\begin{equs}[e:SPDE2]
 d \bar u &=  \Big( \nu \partial_x^2 \bar u + \tF (\bar u)
 		+ G(\bar u) \partial_x \bar u \Big) \, dt  + \theta(\bar{u}) \, dW\;, \\
		\bar{u}(0) &= u^0  \;.
\end{equs}
In this equation, the vector valued function $\bar{F}$ is given by 
\begin{align}\label{e:Fbar}
 \tF^i := (F^i - \Lambda \,  \theta_k^j  \partial_j G^i_l   \,  \theta_k^l )\;,
\end{align}
where we follow the convention to  sum over repeated indices. The  correction constant   $\Lambda$ can be  calculated explicitly as
\begin{equ}[e:defLambda]
 \Lambda \eqdef \frac{1}{2\pi\nu}  \int_{\R_+} 
    \int_\R \frac{(1- \cos(yt))h^2(t)}{t^2 f(t)}  \, \mu(dy) \, dt\;.
\end{equ}
Note that a straightforward calculation shows that $\Lambda$ is indeed well-defined, as a consequence of the fact that 
$h^2 \ls  f$ by assumption and that $|\mu|$ has a finite second moments.
The constant $\Lambda$ is identical to the constant appearing in \cite{HM10}. There, it 
has been calculated for several natural approximation schemes including the case where 
only the nonlinearity is discretised, as well as a finite difference and a Galerkin 
discretisation. 

Note that in the non-gradient case $G \neq \nabla \mathcal{G}$,  \eref{e:SPDE2} has 
to be interpreted as in \cite{HW10}. Actually, there a slightly different equation 
is considered -- the equation studied in \cite{HW10} does not include the reaction 
term $\tF$ and more importantly, global boundedness of $G,\theta$ as well as 
its derivatives up to order three is assumed to guarantee global existence.  
Treating the additional reaction term $\tF$ is a straightforward modification 
that does not pose any problem for this approach. In the present paper, we also drop the 
assumption on the 
boundedness of $F$, $G$ and $\theta$, so we allow for explosion in finite time. 
We will deal with this by working up to a suitable stopping time. More precisely, for any $K >0$ we define the stopping times 
\begin{equs}
\tau_K^* := \inf \big\{ t \colon| \bar{u}(t) |_{\cN } \geq K   \big\},
\end{equs}
where $|\cdot|_{\cN}$ denotes the supremum norm. The explosion time of $\bar{u}$ is then defined to be $\tau^* = \lim_{K \to \infty} \tau_K^*$.

The main result of this article is the following theorem. 
\begin{theorem}\label{thm:main-one} 
Let $\alpha_\star = {1\over 2} -\kappa$ for some $\kappa > 0$. Then, for every $\kappa$ small enough,
there exists a $\gamma>0$ with $\lim_{\kappa \to 0} \gamma(\kappa) = {1\over 6}$
such that the following is true.

Let $|u^0|_{\cC^{\alpha_\star}} <\infty$ and  $\sup_{\eps \le 1} |u^0_\eps|_{\cC^{\alpha_\star}} < \infty $ and denote by $u_\eps$ and $\bar{u}$ the solutions to \eref{e:SPDEeps} and   \eref{e:SPDE2}. If the initial data $u_\eps^0$ and $u^0$ satisfy additionally  
\begin{equs}
| u_\eps^0 -  u^0|_{\cC^{1/3}} \ls  \eps^\gamma,
\end{equs}
then there exists a sequence of stopping times $\tau_\eps$ satisfying $\lim_{\eps \to 0}\tau_\eps = \tau^*$ 
in probability, and such that for any $\tilde{\gamma} < \gamma$
\begin{equ}
 \lim_{\eps \to 0} \P\Big( \sup_{0 \leq t \leq \tau_\eps} | u_\eps(t) - \bar u(t)|_{\cN } >  \eps^{  \tilde{\gamma}} \Big) = 0\;.
\end{equ}
\end{theorem}

\begin{remark}
As pointed out below in Section~\ref{sec:OFP} and in Appendix~\ref{AppA} the construction of the integral $\int \! \phi \,  G(u) \, du$ involves 
in principle the choice of  \emph{iterated integrals} of a certain Gaussian process, but there turns out to exist a canonical choice  $\XX$. 
The solution theory developed in \cite{HW10} still works if we replace $\XX$ by 
\begin{equs}
\tilde{\XX}(s;x,y) \eqdef \XX( s;x,y) - \Lambda (y-x) \Id\;,
\end{equs}
but it yields a \emph{different} solution.  In \cite{HW10} it was shown that this solution then coincides with $\bar{u}$. One can then also interpret this as stating that the approximations $\ue$ converge to solutions of the \emph{correct} equation \eref{e:SPDE}, but where a \emph{different} stochastic integral is used to interpret the nonlinearity involving $G$. 
\end{remark}
\begin{remark}
In the additive noise case our rate of convergence is not optimal. Actually, at least in the case where the noise is additive and one only discretises the derivative, our argument in Sections \ref{sec:GauFluc} and  \ref{sec:RouPathEst} would give a better rate. We believe that in that case a slight improvement of our calculations would yield a rate of almost  $\eps^{1/2}$. We suspect this to be the true rate of convergence in that case.

In the multiplicative case we do not expect the convergence to be very quick and our rate could be close to optimal. Actually, in \cite{HV11}  approximations to \eref{e:SPDE} were studied numerically. In the case of additive noise the convergence which is the content of Theorem \ref{thm:main-one} could be observed, but not in the case of multiplicative noise. It might however be possible to improve the rate of convergence by considering weak
(in the probabilistic sense) convergence, as was observed in \cite{TalayTubaro} and recently exploited
in the approximation to \eref{e:SPDE} when $G=0$ \cite{Arnaud}. 

Note also that the rate ${1\over 6}$ obtained here seems unrelated to the ``order barrier'' mentioned in
\cite{MR1803132}. 
\end{remark}

\begin{remark}
The condition that the initial conditions are bounded in $\cC^{\alpha_\star}$ and converge in a larger space $\cC^{\frac13}$ may seem slightly bulky. We choose to state the result in this way to obtain the optimal rate of convergence.  Note that if $u^0$ has the regularity of Brownian motion and $u^0_\eps$ is a piecewise linearisation, then these conditions are satisfied. We also refer to Remark \ref{rem:InCond} for a more detailed discussion about the initial condition.  
\end{remark}

\begin{remark}\label{rem:Holder}
A crucial technical difference between the present article and \cite{HM10} comes from the fact that for most of the argument we work in H\"older spaces instead of Sobolev spaces. This is necessary to apply the theory of controlled rough paths. Some arguments become easier in H\"older spaces because Gaussian random fields tend to have the same degree of H\"older regularity as Sobolev regularity. The sample paths of Brownian motion, for example, take values in every Sobolev space $H^s$ for $s< \frac{1}{2}$ but in no $H^s$ for $s \geq \frac{1}{2}$. It also takes values in $\cC^\alpha$ for the same values of $\alpha$ which is a much stronger statement. (Sobolev embedding would not even yield continuous sample paths!)  Using this additional information, we can skip the  messy \emph{high frequency cut-off} needed in the proof in \cite{HM10}.  The price to pay is that it is more difficult to get bounds on the approximated heat semigroup. As the approximations are given in Fourier coordinates, bounds in  $L^2$-based Sobolev spaces are trivial to obtain, but the derivation in H\"older spaces requires some work.  For example, we need the additional Assumption~\ref{a:BV} to ensure that the approximations of the heat semigroup are well-behaved not only in Sobolev but also in H\"older spaces.
\end{remark}

\begin{remark}
It is always possible to reduce ourselves to the case $\nu = 1$ by performing a simple time change.
For the sake of conciseness, we therefore make this choice throughout the remainder of this article.
\end{remark}

\subsection{Structure of the paper}
We start Section \ref{sec:OFP} with a short reminder of the solution theory from \cite{HW10}. Then we introduce  the main quantities needed for the proof of Theorem \ref{thm:main-one} and state the bounds on these. Finally, at the end of this section we give the proof of our main result. In the remaining sections we give the proofs for the bounds stated in Section \ref{sec:OFP}. In Section \ref{sec:PrlmCalc} we provide  a priori bounds on the main quantities involved.  In Section \ref{sec:GauFluc} the convergence of the \emph{extra term} is proved. In  Section \ref{sec:RouPathEst} the convergence of the term involving the spatial rough integrals is shown. The Sections \ref{sec:PrlmCalc} -- \ref{sec:RouPathEst} form the core of our argument.   In Section \ref{sec:HarmAn} we prove some auxiliary regularity results. Finally, in Appendix \ref{AppA} we recall some basic notions of rough path theory used in this
work and in Appendix \ref{AppB} we give a higher-dimensional extension of the classical Garsia-Rodemich-Rumsey Lemma.   

\subsection{Norms and notation}

Throughout the paper we will use a whole zoo of different H\"older type norms and for later reference we provide a list here. For a normed vector space $V$ we denote by $\cN(V)$ the space of continuous functions from $[-\pi,\pi]$ to $V$ and by $\bN(V)$ the space of continuous functions from $[-\pi,\pi]^2$ to $V$  vanishing on the diagonal (i.e.\ for  $R \in \bN (V)$ we have $R(x,x)=0$ for all $x \in[-\pi,\pi]$). We will often omit the reference to the space $V$ 
when it is clear from the context and simply write $\cN$ and $\bN$ instead. 

For a given parameter $\alpha \in (0,1)$  we define H\"older-type semi-norms:
\begin{gather}
| X |_{\alpha} = \sup_{x \neq y } \frac{| X(x) - X(y) |}{|x-y|^\alpha} \quad \text{and} \quad | R |_{\alpha} = \sup_{x \neq y } \frac{|R(x,y)|}{|x-y|^{\alpha}},\label{eq:norms-rp} 
\end{gather}
and denote by $\cC^\alpha$ resp. $\bB^\alpha$ the set of functions for which these semi-norms are finite. The space $\cC^\alpha$ endowed with $| \cdot  |_{\cC^\alpha} = |\cdot |_\cN + | \cdot |_\alpha $ is a Banach space. Here $|\cdot|_{\cN}$ denotes the supremum norm.  The space $\bB^\alpha(V)$ is a Banach space endowed with $| \cdot |_\alpha$ alone. As usual, for $\alpha\geq1$, we will denote by $\cC^\alpha$ the space of $\lfloor \alpha \rfloor$ times continuously differentiable functions whose $\lfloor \alpha \rfloor$th derivative is $\alpha -  \lfloor \alpha \rfloor$ H\"older continuous.

For a function $u: [0,T] \times [-\pi,\pi] \to \R^n$ or $u: [0,T] \times [-\pi,\pi] \to \R^{n\times n}$ and for any $\alpha_1,\alpha_2 \in (0,1)$ and $t_1 < t_2 \leq T$ we denote by
\begin{equ}[e:Norm17]
\| u \|_{\cC^{\alpha_1,\alpha_2}_{[t_1,t_2]}}  := \sup_{\substack{s_1,s_2 \in [t_1,t_2] \\x,y \in [-\pi, \pi]}}  \frac{| u(s_1,x)  - u(s_2,y) | }{ |s_1-s_2|^{\alpha_1} + |x-y|^{\alpha_2} }   + \sup_{\substack{ s\in [t_1,t_2] \\ x \in [-\pi,\pi]}} |u(s,x)|
\end{equ}
the inhomogeneous $\alpha_1,\alpha_2$-H\"older norm of $u$. In most cases we will have $t_1 =0$ and then we simply write
\begin{equs}[e:alphaHol1]
\| u \|_{\cC^{\alpha_1,\alpha_2}_t}  := \| u \|_{\cC^{\alpha_1,\alpha_2}_{[0,t]}} .
\end{equs}
If we are only interested in the spatial regularity, we write for $\gamma \in (0,1)$
\begin{equ}[e:alphaHol2]
\| u \|_{\cC^{\gamma}_{[t_1,t_2]}}  := \sup_{\substack{s \in [t_1,t_2] \\x,y \in [-\pi, \pi]}}  \frac{| u(s,x)  - u(s,y) | }{ |x-y|^{\gamma} }   + \sup_{\substack{ s\in [t_1,t_2] \\ x \in [-\pi,\pi]}} |u(s,x)|,
\end{equ}
and if $t_1 =0$ we use $\| u \|_{\cC^{\gamma}_{t}} := \| u \|_{\cC^{\gamma}_{[0,t]}}$.
We simply write 
\begin{equ}[e:alphaHol3]
\| u \|_{\cN_{[t_1,t_2] }} := \sup_{ s\in [t_1,t_2]} \sup_{  x \in [-\pi,\pi] } |u(s,x)|
\end{equ}
and $\cN_t := \cN_{[0,t]}$ for the supremum norm. We will also need a similar norm, for functions that depend  on two space variables and that vanish on the diagonal. For $R: [0,T] \times [-\pi,\pi]^2 \to \R^n $ or $R: [0,T] \times [-\pi,\pi]^2 \to \R^{n \times n}$   we write \begin{equ}[e:alphaHol4]
\| R \|_{\bB^{\gamma}_{[t_1,t_2]}}  := \sup_{ s\in [t_1,t_2]} \sup_{  x\neq y \in [-\pi,\pi] }  \frac{| R(s;x,y)  | }{ |x-y|^{\gamma} }   .
\end{equ}
Finally, we will sometimes have to allow for blowup of a function near time $t_1 \geq 0$. This can be captured by 
\begin{equ}[e:alphaHolHor]
\| R \|_{\bB^{\gamma}_{[t_1,t_2],\beta}}  :=  \sup_{s \in (t_1,t_2]} (s-t_1)^{\beta}\sup_{x,y \in [-\pi, \pi]}  \frac{| R(s;x,y)  | }{ |x-y|^{\gamma} } ,  
\end{equ}
for some $\beta \in [0,1]$. As above, if $t_1 =0$ we write
\begin{equs}[e:alphaHol4A]
\| R \|_{\bB^{\gamma}_{t,\beta}} := \| R \|_{\bB^{\gamma}_{[0,t],\beta}} . 
\end{equs}

We will write $\cC^{\alpha_1,\alpha_2}_{[t_1,t_2]}, \cC^{\gamma}_{[t_1,t_2]}, \cN_{[t_1,t_2]}, \bB_{[t_1,t_2]}^{\gamma}$  and $\bB_{[t_1,t_2],\beta}^{\gamma}$ for the spaces of functions for which these norms are finite.

We will avoid the use of indices as much as possible and only use them if expressions would get ambiguous otherwise. When we do use indices, we always use the convention of summation over repeated indices. We will write  $A^+ = \frac12(A + A^*)$ and $A^- = \frac12(A - A^*)$ for the symmetric and anti-symmetric part of a matrix $A$. The Hilbert-Schmidt norm of a matrix $A$ will simply be denoted by $|A|$.

Finally, we will use the notation $x \lesssim y$ to indicate that there exists a constant $C$ that does not depend on the relevant quantities so that $x \leq C \, y$.  Similarly, $x \eqsim y$ means that $C^{-1} x \leq y \leq C x$.

\section{Outline and proof of the main result}\label{sec:OFP}

We start this section by presenting an outline of the construction of solutions to \eref{e:SPDE} in Subsection \ref{ss:solutions}. In Subsection \ref{ss:summary} we discuss how the quantities involved behave under approximations. The proofs of the bounds announced in this subsection form the core of this article and will be presented in the subsequent sections. Finally, in Subsection \ref{ss:proof} these bounds will be summarised to give a proof of Theorem \ref{thm:main-one}.

\subsection{Construction of solutions to rough Burgers-like equations}
\label{ss:solutions}

In this section we give  an outline of the construction of local solutions to \eref{e:SPDE}. The construction given here differs slightly from  the construction  presented in \cite{HW10}, as this will hopefully make the proof of the main result in Subsection \ref{ss:proof} more transparent. We comment on the differences below in Remark \ref{re:Differences} and Remark~\ref{re:Differences2}. We refer the reader to Appendix \ref{AppA} for the necessary notions of rough path theory. 
For the moment, we assume that $\alpha$ is an arbitrary exponent in $(\frac{1}{3}, \frac12)$; it will be fixed later in Subsection \ref{ss:proof}. 

Let us start by fixing some notation. 
 Throughout the paper we will write
\begin{equs}
S(t) = e^{t\Delta}
\end{equs}
for  the semigroup generated by $\Delta$. Recall that the operator $S(t)$ acts on functions as convolution (on the torus) with the heat kernel 
\begin{equ}[e:DefptA]
p_t(x) \, = \,  \frac{1}{\sqrt{2 \pi}} \sum_{k \in \Z} e^{- t k^2   } e^{i k x}. 
\end{equ}
For adapted $L^2[-\pi,\pi]$ valued processes $\theta$  and $F$ we will frequently write
\begin{equs}
\Psi^\theta(t) := \int_0^t S(t-s) \,  \theta (s) \,  dW(s) \quad \text{and} \quad \Phi^F(t) := \int_0^t S(t-s) \, F (s) \,  ds.
\end{equs}
Then with this notation the \emph{mild} formulation of  \eref{e:SPDE} reads 
\begin{equ}[e:SPDEA]
u(t) =  S(t) \, u^0 + \Psi^{\theta(u)}(t)  + \Phi^{F(u)}(t) + \iint_0^t S(t-s)  \, G\big( u(s) \big) \, \partial_x  u (s) \, ds\;, 
\end{equ}
where the symbol $\iint$ denotes a \textit{rough integral} which we will define below.

The terms $S(t)u^0$ and the reaction term $\Phi^{F(u)}$ do not cause any major difficulty, so we will concentrate  for the moment on  the two terms
\begin{equ}
 \Psi^{\theta(u)}(t) = \int_0^t S(t-s) \,  \theta (u(s)) \,  dW(s)  \quad \text{and} \quad  \iint_0^t S(t-s)  \, G\big( u(s) \big) \, \partial_x  u (s) \, ds.
\end{equ}
As pointed out in the introduction, we will use the theory of controlled rough paths to interpret the term involving $G$.  We introduce the auxiliary function
\begin{equ}[e:Out1]
Z(t,x) := \iint_{-\pi}^x G\big(u(t,y)\big) \, d_y u(t,y),
\end{equ}
and write
\begin{equ}
  \Xi^u(t)  :=  \int_0^t S(t-s)   \partial_x Z(s) \, ds =  \int_0^t \partial_x \bigl( S(t-s)  Z(s)\bigr) \, ds.
%&:= \int_0^t \bigg[ \iint_{-\pi}^\pi  p_{t-s}(\cdot-y)  \, G\big( u(s,y) \big) \, d_y  u (s,y) \bigg]  ds . 
\end{equ}
We argue below that for every $t$ the rough integral in equation \eqref{e:Out1} defines $x \mapsto Z(t,x)$ as an $\alpha$-H\"older function. Hence the spatial derivative is well defined in the sense of distributions. The last equality follows because $\partial_x$ commutes with $S(t-s)$.  Equivalently, assuming that we know how to
define the rough integral in \eref{e:Out1}, the last identity can be taken as the definition of $\Xi^u$.

In order to define the spatial integral on the right-hand side of \eref{e:Out1} as a rough integral, for every $s \in (0,t)$ we must specify  a reference path $X(s)$. These reference paths must meet the following requirements: 
\begin{itemize}
\item For every $s$ it must be possible to construct the iterated integrals
\begin{equs}[e:Out3]
\XX(s;x,y) =: \iint_x^y \big( X(s,z) - X(s,x) \big)  \otimes \, d_z X(s,z).
\end{equs}
Note here that the right hand side is \textit{defined} by the values of $\XX$.
Regarding $\XX$ itself, we require it to satisfy a number of algebraic and analytic
properties that are natural in view of its interpretation as an ``integral''.
\item For every $s$ the random function $x \mapsto u(s,x)$ must be controlled by $X(s)$, in the sense that we need to be able find a derivative process $u'(s,x)$ such that
\begin{equs}[e:Out2]
u(s,y) -u(s,x) \, = \, u'(s,x) \big( X(s,y) - X(s,x) \big) + \Ru (s;x,y),  
\end{equs}
where the remainder $\Ru$ vanishes sufficiently fast as $|y-x| \to 0$. 
\end{itemize}
Such reference paths are provided by the stationary, zero mean solution to the linear stochastic heat equation 
\begin{equs}
\partial_t{X} =  \partial_{x}^2 X + \Pi \xi.
\end{equs}
Here $\Pi$ denotes the orthogonal projection in $L^2$ onto the space of functions with zero mean. Actually, the construction of $\XX$ is rather straightforward. The point is that $X$ is a Gaussian process with explicitly known covariance structure so that known existence results (see \cite{FV10}) apply.  The process $\XX$ is constructed by evaluating \eref{e:Out3} for a sequence of approximations to $X$ and checking that the sequence of approximate iterated integrals converges in the right sense. The crucial ingredient for this calculation is provided by Nelson's estimate (Lemma~\ref{lem:Nelson}) that yields the equivalence of all moments in a given Wiener chaos.

When checking \eref{e:Out2} it is sufficient to  look at the term $\Psi^{\theta(u)}$. Actually, the terms $S(t) \, u^0$ and $\Phi^{F(u)}$ will be $\cC^1$ in space, so they can be included in the remainder $R^u$ and  we need not worry about them. The same 
turns out to be true for the term $\Xi^u$ discussed in \eref{e:Out1}.

For $\Psi^{\theta(u)}$ we can write 
\begin{equs}
  \Psi^{\theta(u)}(t,y)- \Psi^{\theta(u)} (t,x) =  \theta(t,x) \big(X(t,y)- X(t,x)\big) + R^{\theta(u)}(t;x,y). 
\end{equs}
It is shown in \cite[Proposition 4.8]{HW10} that the  term $R^{\theta(u)}$ does indeed have the necessary $2\alpha$-regularity near the diagonal as soon as 
\begin{equ}[e:alphaHol1A]
 \E  \Bigg( \sup_{\substack{s,t \in [0,\tau] \\x,y \in [-\pi, \pi]}}  \frac{| u(s,x)  - u(t,y) | }{ |s-t|^{\alpha/2} + |x-y|^\alpha }  + \sup_{\substack{t \in [0,\tau] \\x \in [-\pi, \pi]}}  |u(t,x)|  \Bigg)^p 
\end{equ}
is finite for a suitable stopping time $\tau$ and large enough $p$. This is precisely the regularity we expect for $u$. With these observations at hand we are ready to set up a fixed point argument to solve \eref{e:SPDEA}.

Then for some $p \geq 2$  we denote by $\mathcal{A}_p$ the space of triples 
\begin{equs}
(u,u',\Ru) \in L^p\big( \cC^{\alpha/2,\alpha}_T \big) \times L^p\big( \cC_T^\alpha \big) \times L^p\big( \bB^{2\alpha}_{T , \alpha/2} \big)
\end{equs}
that satisfy the following conditions:
\begin{itemize}
\item The processes $t \mapsto u(t), \, t \mapsto u'(t), \, t \mapsto \Ru(t)$ are adapted. 
\item Almost surely, for every $t \in [0,T]$ the triple $u(t, \cdot), u'(t,\cdot), \Ru(t,\cdot)$ is controlled by $X(t,\cdot)$. To be more precise, we assume that \eref{e:Out2} holds almost surely for all $s,x,y$.
\end{itemize}
Here the $L^p$ refers to $p$-th stochastic moments. 

It is easy  to check that $\mathcal{A}_p$ is a closed linear subspace of $ L^p\big( \cC^{\alpha/2,\alpha}_T \big) \times L^p\big( \cC_T^\alpha \big) \times L^p\big( \bB^{2\alpha}_{T,\alpha/2} \big)$. Then for such a triple $(u,u',\Ru)$ and for any stopping time $\tau \leq T$ it makes sense to define 
\begin{equs}
\mathcal{M}: (u,u',\Ru) \mapsto (\tilde{u}, \tilde{u}', \mathcal{R}_{\tilde{u}}),
\end{equs}
where
\begin{equs}[e:MainFP]
\tilde{u}(t) &:=  \Big( S \, u^0 + \Psi^{\theta(u)}  + \Phi^{F(u)} + \Xi^{u}  \Big)(t \wedge \tau), \\
\tilde{u}'(t) &:= \theta\big (u( t  \wedge \tau) \big), \\
\mathcal{R}_{\tilde{u} }(t)  &:= \delta \,  \Big(  S \, u^0  + \Phi^{F(u)}    + \Xi^u \Big) (t \wedge \tau) + \Big(  \delta \Psi^{\theta(u)}   - \theta(u) \delta X \Big) (t  \wedge \tau).
\end{equs}
Here the difference operator $\delta$ is defined as 
\begin{equs}[e:Defdelta]
\delta u(s;x,y) := u(s,y) - u(s,x).
\end{equs}
Using the bounds mentioned above, we can show that, for $u^{0} \in \cC^{\alpha}$ and under suitable assumptions on $\tau, \kappa$ and $p$, the operator $\mathcal{M}$ is a contraction from a ball in $\mathcal{A}_p$ into itself. 

As usual one can obtain solutions on a longer time interval by iterating this procedure. For fixed choice of the iterated integral process $\XX$, these solutions are unique. 

\begin{remark}\label{rem:InCond}
The reason for allowing the remainder $\mathcal{R}_{\tilde{u}}$ to blow up like $t^{-\alpha/2}$ near zero lies in  the initial condition  $u^{0} \in \cC^{\alpha}$. Actually, the regularising property of the heat semigroup implies that we have
\begin{equs}
\sup_{t \le 1} t^{{\alpha}/{2}} \big| S(t) u^{0} \big|_{\cC^{2\alpha}} \ls |u^0|_{\cC^\alpha} .
\end{equs}
We need this bound to control the contribution of the initial condition to the remainder term. 

 This issue would  be avoided completely if we could assume that $u^{0} \in \cC^{2\alpha}$. The problem is that even under this stronger assumption on the initial condition, after positive time the solutions $u(t)$ would only attain values in $\cC^\alpha$ for any $\alpha< \frac{1}{2}$. This  would make it impossible to iterate this construction to get solutions on a longer time interval. 
 \end{remark}

\begin{remark}\label{re:Differences}
The construction in \cite{HW10} is slightly different as it is split up into an \emph{inner}  fixed point argument to  deal with the term 
involving $G$ and an \emph{outer} fixed point argument to conclude. This corresponds to a semi-implicit Picard iteration \begin{equs}
u_{n+1}(t) & =  S(t) \, u^0 + \int_0^t S(t-s) \,  \theta (u_n(s)) \,  dW(s)  \\
& + \int_0^t S(t-s)  \, F \big( u_{n}(s) \big) \,ds + \int_0^t S(t-s)  \, G\big( u_{n+1}(s) \big) \, \partial_x  u_{n+1} (s) \, ds.
\end{equs}
The advantage of this approach is that it separates more clearly the deterministic part from the probabilistic part of the construction. The price to pay is that some stopping arguments get more involved. In terms of the bounds needed, both constructions are essentially equivalent. 
\end{remark}
\begin{remark}\label{re:Differences2}
Another difference between the construction presented here and the one in \cite{HW10} concerns the treatment of the term $\Xi^u$. In \cite{HW10} the term $Z$ defined above in \eqref{e:Out1} is not defined, but $\Xi^u$ is defined directly as 
\begin{equ}
\Xi^u(t):= \int_0^t  \iint_{-\pi}^\pi  p_{t-s}(\cdot-y)  \, G\big( u(s,y) \big) \, d_y  u (s,y)  ds . 
\end{equ}
Then the $\cC^1$-regularity is shown using the scaling behaviour of the heat kernel $p_{t-s}$. In the current context we prefer the new approach (suggested to us by 
the referee) because it allows to separate the rough path bounds from the 
bounds involving the heat semigroup, which makes the argument more transparent.
\end{remark}

\subsection{Outline: Behaviour of the main quantities under approximation}
\label{ss:summary}
In order to prove Theorem \ref{thm:main-one}, we will go through the construction we just described and see how the terms  behave under approximation. 

For $\eps >0$ we denote by $S_\eps(t) = e^{ t \Delta_\eps}$ the semigroup generated by the approximate Laplacian defined in \eref{e:DeltaEps}. Similarly to $S(t)$, it is given by convolution (on the torus) with a heat kernel 
\begin{equ}[e:Defpt]
p_t^\eps(x) \, = \,  \frac{1}{\sqrt{2 \pi}} \sum_{k \in \Z} e^{- t k^2 f(\eps k)  } e^{i k x}. 
\end{equ}
As above, for any adapted $L^2[-\pi,\pi]$ valued processes $\theta$ and $F$  we will write 
\begin{equ}
\Pte (t) = \int_0^t S_{\eps}(t-s) \, \theta(s) \, H_\eps \, dW(s)  \quad \text{and} \quad  \Phi_\eps^F (t) = \int_0^t S_{\eps}(t-s) \,F(s)  \,  ds. 
\end{equ}

Using this notation  the mild version of the approximation  \eref{e:SPDEeps} takes the form
\begin{equs}[e:SPDEB]
u_\eps(t) &=   S_\eps (t) \, u_\eps^0 + \Psi_\eps^{\theta (u_\eps)}(t)  + \Phi_\eps^{F(u_\eps)}(t) \\
&\qquad + \int_0^t S_\eps(t-s)  \, G\big( \ue(s) \big) \, D_\eps \ue (s) \, ds.
\end{equs}

Note that for fixed $\eps > 0$, we do not need rough path theory to solve this fixed point problem. The existence of local solutions can for example be shown through a fixed point argument in $\cN_T$. For positive $\eps$ the approximate derivative operator $D_\eps$ is actually continuous on the space of continuous functions, but the operator norm blows up as $\eps$ goes to zero. Therefore, it will be useful to introduce approximate reference rough paths $(X_\eps,\XX_\eps)$ and interpret the  term involving $G$ on the right-hand side of \eref{e:SPDEB} as an approximation  to a rough integral. 
This allows us to obtain uniform control as $\eps \to 0$.

Similarly to before, we choose as $X_\eps$ the stationary, zero mean  solution of the approximated stochastic heat equation
\begin{equs}
d X_\eps \, =  \, \Delta_\eps X_\eps \, dt + \Pi H_\eps \,dW(t)\;.
\end{equs}
This will be used as a reference rough path for the approximate solution. 
Here, as above, $\Pi$ denotes the orthogonal projection on $L^2$ onto the 
functions with zero mean. If we extend the cylindrical Brownian motion $W$ to negative times we get
\begin{equs}
 X_\eps (t) \, = \, \int_{-\infty}^t \! S_\eps(t-s) \,  \Pi H_\eps \, dW(s).  
\end{equs}
Our first task then consists of checking that for every $s$ the process  $X_\eps(s)$ can indeed be lifted to a rough path $(X_\eps(s),\XX_\eps(s))$, and that for a given adapted process $\theta$ the approximate stochastic convolutions  $\Psi_\eps^{\theta}$ are indeed controlled by $X_\eps$. This is established in Section \ref{sec:PrlmCalc}. To be more precise, we will give bounds on the H\"older regularity of $\Psi^\theta$ in Lemma \ref{le:PTE}. The behaviour of $(X_\eps,\XX_\eps)$ as a rough path is discussed in Lemma \ref{le:GRP}. Finally, in Lemma \ref{le:mistakecorrected} the regularity of the remainder 
\begin{equ}[e:DefRte]
R^\theta_\eps(t;x,y) \, = \,  \big( \Pte(t,y)  - \Pte(t,x) \big) - \theta(t,x) \big( X_\eps(t,y) - X_\eps(t,x) \big)
\end{equ}
is treated. For all of these quantities we can show convergence as $\eps$ goes to zero towards the corresponding terms for the limiting equation. 

Let us point out that, while the derivations of the a priori bounds on $\Psi^\theta_\eps$ and on $(X_\eps, \XX_\eps)$ are rather straightforward, the result for the remainder $R^\theta_\eps$ requires more thought. The necessary spatial $2\alpha$-regularity is shown with a bootstrap argument.

 Once we have established that  $\Psi_\eps^\theta$ is well behaved, it remains to check the behaviour of the term 
\begin{equs}[e:IFPA] 
\int_{-\pi}^x   G\big( u_\eps(s,y)  \big) D_\eps u_\eps(s,y)   \, dy  
\end{equs}
when $\eps \downarrow 0$.  One might hope that for small $\eps$ the integral behaves like an approximation to the rough integral 
\begin{equs}[e:RoIntA]
\iint_{-\pi}^x   G\big( u_\eps(s,y) \big) d_y u_\eps(s,y) .
\end{equs}
Unfortunately, this is not always true. As pointed out in Appendix~\ref{AppA}, rough integrals are limits of second order Riemann sums like \eref{eq:Riem-sum2}. Since the contribution of the second order term may not be negligible in the
limit, one cannot hope to prove that the first order expression in the second line of \eref{e:IFPA} approximates the rough integral \eref{e:RoIntA} in general.
In order to enforce this convergence, we simply add the ``missing'' second order  term to the right-hand side of \eref{e:IFPA}.

 Therefore, we set  
\begin{equs} [e:IFPBA]
Z_\eps(t,x) &  :=  \int_{-\pi}^x \, G\big( u_\eps(s,y)  \big) D_\eps u_\eps(s,y) \, dy  \\
&\qquad+ \int_{-\pi}^x   \, D  G\big(\ue(s,y) \big)  \, u_\eps^\prime  (s,y)    \, D_\eps \XX_\eps(s;y) \, u_\eps^\prime(s,y)   \,  dy  ,
\end{equs}
where
\begin{align*}
D_\eps \XX_\eps(s;y) =  \frac1{\eps}  \int_\R \XX_\eps(s;y, y + \eps z) \, \mu(dz)\;,
\end{align*}
and $u_\eps'$ is a \emph{rough path derivative} of $u_\eps$ with respect to $X_\eps$.
We then define 
\begin{equ}[e:IFPB]
\Xi^{\ue}_\eps(t) := \int_0^t S_\eps(t-s) \partial_x Z_\eps(s) \, ds = \int_0^t \partial_x S_\eps(t-s)  Z_\eps(s) \, ds  .   
\end{equ}
%
%\begin{equs}
%&\Xi^{\ue}_\eps(t,\cdot)   \label{e:IFPB} : = \int_0^t  \bigg[ \int_{-\pi}^\pi p^\eps_{t-s}(\cdot-y) \, G\big( u_\eps(s,y)  \big) D_\eps u_\eps(s,y) \, dy \bigg] \, ds \\
%&+ \int_0^t \!\! \int_{-\pi}^\pi  p^{\eps}_{t-s}(\cdot-y ) \, D  G\big(\ue(s,y) \big)  \, u_\eps^\prime  (s,y)    \, D_\eps \XX_\eps(s;y) \, u_\eps^\prime(s,y)   \,  dy \,  ds ,
%\end{equs}
%
%
We will denote the extra term appearing on the right-hand side of \eref{e:IFPB} by
\begin{equs}[e:DefUps]
\Upsilon_\eps^{\ue} (t, \cdot) :=\int_0^t \!\! \int_{-\pi}^\pi  p^{\eps}_{t-s}(\cdot-y )& \, D  G\big(\ue(s,y) \big) \\
& \,\times  u_\eps^\prime  (s,y)    \, D_\eps \XX_\eps(s;y) \, u_\eps^\prime(s,y)   \,  dy \,  ds .
\end{equs}
Actually, here we have hidden a bit of non-trivial linear algebra in the notation. The expression defining $\Upsilon_\eps^{\ue}$ is trilinear and it is not obvious which term is paired with which. At this level, this does not matter and we will give a precise definition in \eref{e:epsiFP} below.

In Section \ref{sec:RouPathEst} we establish that, under suitable assumptions, $\Xi^{\ue}_\eps$ approximates $\Xi^u$
%\begin{equs}
%\Xi^u(t)    := \int_0^t  \bigg[ \iint_{-\pi}^\pi p_{t-s}(\cdot-y) \, G\big( u(s,y)  \big) d_y u(s,y)  \bigg] \, ds \;,
%\end{equs}
%
provided that the quantity 
\begin{equs}
 \CE_\eps \, = \, &   \|X-X_\eps \|_{\cC^{\alpha}_T} +  \|\XX-\XX_\eps \|_{\bB^{2 \alpha}_T}    +    \| u- u_\eps \|_{\cC^{\alpha}_{T}} +   \|u'-u_\eps' \|_{\cC^{\alpha}_{[\eps^2,T]}} \\
 &\qquad +  \|\Ru - \mathcal{R}_{u_\eps}\|_{\bB^{2\alpha}_{[\eps^2,T],\alpha}}     
\end{equs}
is small.  

Throughout the calculations we will need uniform in $\eps$ smoothing properties of the approximate heat semigroup $S_\eps$. These are established in Section \ref{sec:HarmAn}. A key ingredient is Lemma \ref{lem:Mark}, a version of the  Marcinkiewicz multiplier theorem.  

By adding a rough path derivative and a remainder we can then interpret $\ue$ as solution to the fixed point problem for the operator
\begin{equs}
(\ue, \ue', \Rue) \mapsto  \mathcal{M}_\eps (\ue, \ue', \Rue) = (\tilde{u}_\eps, \tilde{u}_\eps', \mathcal{R}_{\tilde{u}_\eps})
\end{equs}
where 
\begin{equs}
\tilde{u}_\eps(t) &:= \Big( S_\eps \, u^0_\eps +  \Psi_\eps^{\theta(u_\eps)} + \Phi_\eps^{F(u_\eps) }  - \Upsilon_\eps^{\ue}  +  \Xi_\eps^{\ue}  \Big)(t \wedge \tau)  \label{e:OFPB}\\
 \tilde{u}_\eps'(t) &:= \theta\big (u_\eps( t \wedge \tau) \big), \\
\mathcal{R}_{\tilde{u}_\eps }(t)  &:= \delta \,  \Big(  S_\eps \, u_\eps^0  + \Phi_\eps^{F(u_\eps)}  - \Upsilon_\eps^{\ue} + \Xi^{\ue}_\eps  \Big)(t \wedge \tau) \\
&\qquad+ \Big(  \delta \Psi_\eps^{\theta(u_\eps)}   - \theta(u_\eps) \delta X_\eps \Big) (t  \wedge \tau) ,
\end{equs}
for a suitable stopping time $\tau$ (which also depends on $\eps$). 
Here the difference operator $\delta$ is defined as in \eref{e:Defdelta}. 

As expected, the term $\Upsilon_\eps^{\ue}$  will be responsible for the emergence of an extra term in the limit. This term will be treated in Section \ref{sec:GauFluc}, where the convergence of the term $D_\eps \XX_\eps$  will be discussed.

It turns out that  for $\eps$ small enough this term behaves like $\Lambda \Id$, where $\Lambda$ is the constant introduced above in \eref{e:defLambda}. Note that the a 
priori knowledge of the regularity of $\XX_\eps$ would not even imply that the quantity 
$\Upsilon_\eps^{\ue}$ remains bounded, so that the proof of its convergence 
requires to exploit stochastic cancellations. The relevant bound is given in 
Proposition~\ref{prop:rand-fluc}. There, convergence in any stochastic $L^p$ space with respect to the Sobolev norm $H^{-\et}$, for $\eta>0$ is proved. In the next subsection we will be able to establish convergence of the triple $(u_\eps, u_\eps', \Rue)$.

\subsection{Proof of the main result}
\label{ss:proof}
Now we are ready to finish the proof of our main result, Theorem \ref{thm:main-one}, assuming the results from Sections \ref{sec:PrlmCalc} -- \ref{sec:HarmAn}. 

Similar to \eref{e:DefUps} we define 
\begin{equs}[e:DefUps2]
\Upsilon^{\bar{u}} (t) := \Lambda  \int_0^t \!  S(t-s) \big(  D  G\big(\bar{u}(s) \big)  \bar{u}^\prime  (s)    \, \bar{u}^\prime(s)    \big)\,  ds  \quad 
\end{equs}
where the indices are to be interpreted as in \eref{e:Fbar}. Then the mild form of \eref{e:SPDE2} can be written as
\begin{equs}[e:MainFP1]
\bar{u}(t) &:= S(t) \, u^0 + \Psi^{\theta(\bar{u})}(t  )  + \Phi^{F(\bar{u})}(t )  - \Upsilon^{\bar{u}}(t) + \Xi_\eps^{\bar{u}}(t), \\
\bar{u}'(t) &:= \theta\big (\bar{u}( t) \big), \\
\Rub(t)  &:= \delta \Big(  S \, u^0  + \Phi^{F (\bar{u})} - \Upsilon^{\bar{u}}  +  \Xi_\eps^{\bar{u}} \Big)(t) + \Big(  \delta \Psi^{\theta(\bar{u})}  - \theta(\bar{u}) \delta X \Big) (t ).
\end{equs}
Furthermore, for $t < \tau_\eps^{*}$ the process $u_\eps$ solves the fixed point problem  for the operator $\mathcal{M}_\eps$ defined in \eref{e:OFPB}. Recall that  $\tau^*$ denotes the explosion time of $\bar{u}$. Similarly, here $\tau^*_\eps$ denotes the explosion time for $\ue$.  Note that the extra term $\Upsilon^{\bar{u}}$ corresponds to a reaction term and poses no additional problems for the well-posedness of the equation.

In order to optimize the convergence rate we have to work with three different H\"older exponents
\begin{align*}
\alpha_\star> \alpha \geq \tilde{\alpha}\;.
\end{align*}
More precisely for a $0<\kappa < \frac{1}{12}$  we set  $\alpha_\star =\frac{1}{2}- \kappa$. We then fix $\alpha = \alpha_\star - \kappa$ to be a bit smaller and $\tilde{\alpha}\in (\frac{1}{3}, \alpha]$ arbitrary. Note that the regularising property of the heat semigroup implies that 
\begin{equs}
t^{  \frac{\alpha}{2} }\big| S(t) u^0\big|_{2 \alpha} \ls  |u^0 |_{\alpha} \qquad \text{ and } \qquad t^{  \frac{\alpha}{2}  }\big| S_\eps(t) u_\eps^0\big|_{2 \alpha} \ls |u_\eps^0 |_{\aX}.
\end{equs}
In the case of the heat semigroup $S$, this is a standard regularity result. For the approximated semigroup $S_\eps$ the regularisation is shown in  Corollary~\ref{cor:S-Seps}. 

We will measure the regularity of $\bar{u},\ue ,\Rub,$ and $\Rue$ with norms indexed by 
$\alpha$ and $\tilde{\alpha}$. For most terms, the rate of convergence becomes better 
when measured in a norm of lower regularity, see for example Lemma \ref{le:PTE} or Lemma 
\ref{le:GRP} below. In those situations, we use the norms indexed by $\tilde{\alpha}$. 
But in some estimates it is useful to use a priori knowledge on the regularity of 
$u$ that is close to optimal -- this is when we use $\alpha$. We will use $\alpha_\star$ 
to measure the regularity of the initial condition and the regularity of $X$ and 
$X_\eps$. It is useful to have a little bit more regularity available for these quantities.

First we  introduce the following stopping times.  Recall the definitions of the norms $\| \cdot \|_{\cC^{\alpha/2, \alpha}_t}, \| \cdot \|_{\cC^{ \alpha}_t}$  from \eref{e:alphaHol1} and  \eref{e:alphaHol2}. Then for any $K>0$ we define: 
\minilab{e:Stoppzeiten}
\begin{equs}
\sigma^{X}_K &:= \inf \Big\{ t \geq 0 \colon  \|X \|_{\cC^{\aX/2, \aX}_t} \geq K \quad \text{or} \quad \|\XX\|_{\bB^{2\alpha}_t}  \geq K \Big\}, \label{sigmaX}\\
\sigma^{u}_K &:= \inf \Big\{ t \geq 0 \colon  \big\|  \bar{u} \|_{\cC^{\alpha/2, \alpha}_t} \geq K  \Big\},\label{sigmau}\\
\sigma^{R}_K  &:= \inf \Big\{ t  \geq 0\colon |\Rub(t) |_{2\alpha} \geq K t^{-  \frac{\alpha}{2}} \quad  \text{or} \quad |\Rub(t) |_{2\tilde{\alpha}} \geq K t^{-  \frac{\ta}{2}}  \Big\}.   \qquad \label{sigmaR}
\end{equs}
Here we follow the convention to  set the stopping times to be $T$ if the sets are empty. It follows from the bounds in Sections \ref{sec:PrlmCalc} -- \ref{sec:RouPathEst} that for suitable initial conditions $u^0$ these stopping times are almost surely positive. Remark that for $t \leq \sigma^{u}_K$ we have 
\begin{equs}[e:bouthetau]
\| \bar{u}' \|_{\cC^{\alpha}_t} = \| \theta(\bar{u}) \|_{\cC^{\alpha}_t} \leq \sup_{-K \leq |u| \leq K } |\theta(u)| +  \sup_{-K \leq |u| \leq K } | D\theta(u)| K.
\end{equs}
Then we set 
\begin{equs}
\sigma_K =  \sigma^{X}_K  \wedge \sigma^{u}_K  \wedge \sigma^{R}_K .
\end{equs}
In order to have a priori bounds on the corresponding $\eps$-quantities, we fix yet another parameter $\eta=\alpha-\kappa$, and  introduce the stopping times  
\minilab{e:Stoppzeiteneps}
\begin{equs}
\rho^{X}_{\eps} &:= \inf \Big\{ t \geq 0 \colon  \|X - X_\eps \|_{\cC^{\aX/2, \aX}_t} \geq 1, \, \text{or }  \|\XX - \XX_\eps\|_{\bB^{2\alpha}_t}  \geq 1,  \\
& \qquad  \quad  \text{or } \| X-X_\eps\|_{\cN_t} \geq \eps^{\aX},\, \text{or } |X_\eps |_{\dD^{\aX,\eps}} \geq K, \\
&\qquad \quad \text{or }  \Big|  D_\eps \XX_\eps(t) - \Lambda \Id  \Big|_{H^{-\et}}  \geq 1\Big\}, \,\label{sigmaXe}\\
\rho^{u}_{\eps} &:= \inf \Big\{ t \geq 0 \colon \big\|  \bar{u} - \ue \|_{\cC^{ \alpha}_t} \geq 1  \Big\} \wedge \inf \Big\{ t > \eps^2 \colon \big\|  \bar{u} - \ue \|_{\cC^{\alpha/2, \alpha}_{[\eps^2,t]}} \geq 1  \Big\}  , \quad  \qquad \label{sigmaue}\\
\rho^{R}_{\eps}  &:= \inf \Big\{ t > \eps^2 \colon |  \Rub(t) -\Rue(t)  |_{2\alpha} \geq  (t-\eps^2)^{-  \frac{\alpha}{2}}\\
 &\qquad  \quad \text{or} \quad |  \Rub(t) -\Rue(t)  |_{2\tilde{\alpha}} \geq  (t-\eps^2)^{-  \frac{\ta}{2}} \Big\} \label{e:sigmaRe},
\end{equs}
where again, we set the stopping times equal to $T$ if the sets are empty. Here the norm $\dD^{\aX,\eps}$ is defined in \eqref{e:DNorm} below.   In \eref{e:sigmaRe} the quantity $\Rue$ is defined as in \eref{e:OFPB}. Again the bounds proved below in Sections \ref{sec:PrlmCalc} -- \ref{sec:RouPathEst} show that these stopping times are almost surely positive. Note in particular that Lemma \ref{le:GRP} implies that  for any $t>0$ the probability of $\{ \| X-X_\eps\|_{\cN_t} \geq \eps^{\aX}\}$ goes to zero.   

Finally, define 
\begin{equs}
\rho_{K,\eps} = \sigma_K \wedge \rho^{X}_{\eps}  \wedge \rho^{u}_{\eps}  \wedge \rho^{R}_{\eps}\;.
\end{equs}
  It is clear from the definition that if $\rho_{K,\eps}>0$ and for $0 \leq t \leq \rho_{K,\eps}$ we have deterministic bounds on%
\begin{equs}
 &\|X_\eps \|_{\cC^{\aX/2,\aX}_t}, \quad
 \| \XX_\eps \|_{\bB^{2\alpha}_t}, \quad   
 \big|  D_\eps \XX_\eps(t) \big|_{H^{-\et}},  \quad
 \big\|   \ue \|_{\cC^{\alpha}_t}, \quad
 \big\|   \ue \|_{\cC^{\alpha/2, \alpha}_{[\eps^2,t]}}, \\
 &   \| u_\eps^\prime\|_{\cC^\alpha_t } = \|\theta( u_\eps)\|_{\cC^\alpha_t },\quad
 |  \Rue  |_{\bB^{2\alpha}_{[\eps^2,t],\alpha } }, \quad
 |  \Rue  |_{\bB^{2 \tilde{\alpha}}_{[\eps^2,t],\tilde{\alpha}}  }
. 
\end{equs}
From now on, to reduce the number of indices, we will write
\begin{equs}
t_\eps := t \wedge \rho_{K,\eps}.
\end{equs}

Most of the rest of this subsection will be devoted to the proof of the following theorem which will then be shown to imply our main result, Theorem \ref{thm:main-one}. 

\begin{theorem}\label{thm:main-one-mod}
Let the exponents $\alpha_\star, \alpha, \tilde{\alpha}$ and $\kappa$ be as stated at the beginning of this subsection. Suppose that the initial conditions satisfy 
\begin{equs}
\big| u^0 \big|_{\cC^{\alpha_\star}}  < K \quad  \text{and} \quad  \big| u^0_\eps \big|_{\cC^{\alpha_\star}}  < K
\end{equs}
for some large constant $K$.
Then there exists a constant
\begin{equs}
\gamma = \gamma(\tilde{\alpha},\kappa) >0, 
\end{equs}
%
%such that if 
%\begin{equ}
%\big|u^0 - u^0_\eps \big|_{\cC^{  \tilde{\alpha}}} \ls \eps^\gamma,
%\end{equ}
%then 
such that for any terminal time $T > 0$ and for any $p \geq 1$, we have
\begin{equs}[e:MainBound]
 \E \Big[ \| \bar{u} - \ue \|_{\cC^{\tilde{\alpha}}_{T_\eps } }^p + \| \bar{u} - \ue \|_{\cC^{\tilde{\alpha}/2,\tilde{\alpha}}_{[\eps^2, T_\eps] } }^p &+    \| \Rub -\Rue  \|_{\bB^{2 \tilde{\alpha}}_{[\eps^2, T_\eps]  , \ta/2} }^p   \Big]  \\
 &\quad\ls \big|u^0 - u^0_\eps \big|_{\cC^{  \tilde{\alpha}}}^p + \eps^{p \gamma} . \qquad 
\end{equs}
Furthermore, for every fixed $\tilde{\alpha}$, the constant $\gamma(\tilde{\alpha},\kappa)$ can be chosen arbitrarily close to $ \frac12 -  \tilde{\alpha}$ by taking $\kappa > 0$ sufficiently small. 
\end{theorem}

\begin{proof}[Proof of Theorem \ref{thm:main-one-mod}]
We start by introducing some $\hat{K} >K$ to be fixed later. We will also only use the fact that
the initial conditions satisfy the bound 
\begin{equ}
\big| u^0 \big|_{\cC^{\alpha_\star}}  < \hat K \quad  \text{and} \quad  \big| u^0_\eps \big|_{\cC^{\alpha_\star}}  < \hat K\;. 
\end{equ}
This will be useful later on.

The functions $F,G,\theta$ will only be evaluated for $u$ with $|u| \leq  K+1$. All the quantities of interest will remain unchanged if we change $F,G$ and $\theta$ outside a ball. Therefore, from now on we can and will make the additional assumption 
that 
\begin{equs}
\big| F \big|_{\cC^1} < \infty\;, \quad \big| G \big|_{\cC^3} < \infty\;, \quad \big| \theta \big|_{\cC^2} < \infty\;.
\end{equs}
For any $\tilde{\alpha}\in (\frac13,\alpha]$ we will derive a bound on the quantity 
\begin{equs}
\mathcal{E}^{\tilde{\alpha}} (t) :=& \bigg( \E    \| \bar{u} - \ue \|_{\cC^{\tilde{\alpha}}_{t_\eps } }^p  \bigg)^{  \frac{1}{p}} +  \bigg( \E    \| \bar{u} - \ue \|_{\cC^{\tilde{\alpha}/2,\tilde{\alpha}}_{[\eps^2, t_\eps] } }^p  \bigg)^{  \frac{1}{p}} \\
&\qquad +  \bigg( \E  \|  \Rub - \Rue \|_{\bB^{2 \tilde{\alpha}}_{[\eps^2,t_\eps] , {\ta}/{2}} }^p  \bigg)^{  \frac{1}{p}}.
\end{equs}
Using the equations \eref{e:MainFP1} and \eref{e:OFPB}  that $\bar{u}$ and $\ue$ satisfy, we get the bound
\begin{equs}[e:DefEE]
\mathcal{E}^{\tilde{\alpha}} (t) \ls \sum_{i=1}^5 I_i^{\tilde{\alpha}}(t),
\end{equs}
where
\minilab{e:contris}
\begin{equs}
I_1^{\tilde{\alpha}}(t) &:=   \big\| S(\cdot) u^0  -S_\eps(\cdot) u^0_\eps \big\|_{\cC^{\tilde{\alpha}/2, \tilde{\alpha}}_{[\eps^2, t] }} +  \big\| S(\cdot) u^0  -S_\eps(\cdot) u^0_\eps \big\|_{\cC^{\tilde{\alpha}}_{t }}  \\
& \qquad  +  \sup_{0  < s \leq t } s^{  \frac{\tbp}{2}} \big| S(s) u^0  -S_\eps(s) u^0_\eps  \big|_{\cC^{2 \tilde{\alpha}}}, \qquad  \label{e:contri1}\\
I_2^{\tilde{\alpha}}(t) &:=  \Big( \E   \big\|  \Psi^{  \theta(\bar{u})}  - \Psi_\eps^{\theta(\ue)} \big\|_{\cC^{\tilde{\alpha}/2,\tilde{\alpha}}_{t_\eps }}^p \Big)^{  \frac{1}{p}}
\\& \qquad  + \Big(  \E   \big\| R^{\theta(\bar{u})} - R_\eps^{\theta(\ue)}  \big\|_{\bB^{2 \tilde{\alpha}}_{[\eps^2,t_\eps], \tilde{\alpha}/2}  }^p  \Big)^{\frac{1}{p}}, 
\qquad \label{e:contri2} \\
I_3^{\tilde{\alpha}}(t) &:=  \Big(  \E   \big\| \Phi^{F(\bar{u}) }  - \Phi_\eps^{F( \ue)}  \big\|_{\cC^{\tilde{\alpha}/2, \tilde{\alpha}}_{t_\eps}}^p  \Big)^{  \frac{1}{p}}   \\
&\qquad +   \bigg(\E \Big(  \sup_{0  < s \leq t_\eps } s^{  \frac{\tbp}{2}} \big|\Phi^{F(\bar{u}) }(s)  - \Phi_\eps^{F( \ue )}(s)   \big|_{\cC^{2 \tilde{\alpha}}}   \Big)^p  \bigg)^{  \frac{1}{p}} ,  \label{e:contri3}\\
I_4^{\tilde{\alpha}}(t) &:=  \Big( \E   \big\|    \Upsilon^{\bar{u}}  - \Upsilon^{\ue}_\eps \big\|_{\cC^{\tilde{\alpha}/2, \tilde{\alpha}}_{t_\eps}}^p  \Big)^{  \frac{1}{p}}  \\
& \qquad  +   \Big( \E \Big(  \sup_{0  < s \leq t_\eps } s^{  \frac{\tbp}{2}} \big| \Upsilon^{\bar{u}}(s)  - \Upsilon^{\ue}_\eps(s) \big|_{\cC^{2 \tilde{\alpha}}}   \Big)^p  \Big)^{  \frac{1}{p}} , \qquad  \label{e:contri3B}\\
I_5^{\tilde{\alpha}}(t) &:=  \Big( \E   \big\| \Xi^{\bar{u}} - \Xi^{\ue}_\eps   \big\|_{\cC^{ \tilde{\alpha}}_{t_\eps}}^p  \Big)^{  \frac{1}{p}}   +\Big(  \E   \big\| \Xi^{\bar{u}} - \Xi^{\ue}_\eps   \big\|_{\cC^{\tilde{\alpha}/2, \tilde{\alpha}}_{[\eps^2,t_\eps]}}^p  \Big)^{  \frac{1}{p}}  \\
& \qquad  +   \Big( \E    \Big( \sup_{0  < s \leq t_\eps } s^{  \frac{\tbp}{2}} \big| \Xi^{\bar{u}}(s) - \Xi^{\ue}_\eps(s)    \big|_{\cC^{2 \tilde{\alpha}}} \Big)^p  \Big)^{  \frac{1}{p}}.   \label{e:contri4}
\end{equs}
Actually, in $I_1^{\tilde{\alpha}}$ -- $I_5^{\tilde{\alpha}}$ we give slightly more information than needed.  Note in particular, that in $I_3^{\tilde{\alpha}}, I_4^{\tilde{\alpha}},$ and $I_5^{\tilde{\alpha}}$ we only allow for blowup at $0$, not at $\eps^2$. This bound is stronger.

We start by giving a bound on  $I_1^{ \tilde{\alpha}}$. For every $t > \eps^2$ we get for any $\lambda_1<\aX-\tilde{\alpha}$ that
\begin{equs}[e:PMR0]
\big\| S(\cdot) & u^0  -S_\eps(\cdot) u^0_\eps \big\|_{\cC^{\tilde{\alpha}/2, \tilde{\alpha}}_{[\eps^2, t] }} \\
&\leq \big\|   S(\cdot)  \big(  u^0 - u^0_\eps \big)  \big\|_{\cC^{\tilde{\alpha}/2, \tilde{\alpha}}_{t }} + \big\| \big(  S(\cdot)  -S_\eps(\cdot) \big)  u^0_\eps \big\|_{\cC^{\tilde{\alpha}/2, \tilde{\alpha}}_{[\eps^2,t] }}\\
&\ls \big| u^0 - u^0_\eps \big|_{\cC^{  \tilde{\alpha}}} +  \big|u_\eps^0 \big|_{\cC^{\alpha_\star}} \eps^{\lambda_1}.
\end{equs}
Here we have used the fact that the heat semigroup is a contraction from $\cC^{\alpha}$  
to $\cC^\alpha$ as well as the time continuity of the heat semigroup in the first term. 
In the second term we have used Corollary \ref{cor:S-Seps}, which provides uniform 
bounds on the spatial regularisation due to the approximated heat semigroup. We use 
Lemma~\ref{lem:timecontS} to get the temporal regularity. 

The remaining terms in \eref{e:contri1},
\begin{equ} [e:PMR0A]
\big\| S(\cdot)  u^0  -S_\eps(\cdot) u^0_\eps \big\|_{\cC^{ \tilde{\alpha}}_{ t }} \qquad \text{and} \qquad 
\sup_{0  < s \leq t  } s^{  \frac{\tbp}{2}} \big| S(s) u^0  -S_\eps(s) u^0_\eps  \big|_{\cC^{2 \tilde{\alpha}}}\, ,
\end{equ}
can be bounded by the same quantity. Here we use that both Lemmas~\ref{le:ApOp} and \ref{cor:S-Seps} regarding the spatial regularity hold for arbitrary times.  Hence, using the boundedness of $|u_\eps|_{\cC^{\aX}}$ we can conclude that 
\begin{equs}[e:PMR0B]
I_1^{\tilde{\alpha}}(t) \ls \big| u^0 - u^0_\eps \big|_{\cC^{  \tilde{\alpha}}} +  \eps^{\lambda_1}.
\end{equs}
This is the only part in the argument, where we will use the boundedness of $u_\eps^0$ or $u^0$ in the $\cC^{\alpha_\star}$-norm. Note that the implicit constants are uniform for all $u^0,u^0_\eps$ satisfying  $|u^0|_{\cC^{\aX}} \leq \hat K$  and $|u_\eps^0|_{\cC^{\aX}}\leq \hat K$ .

The bounds on $I_2^{\tilde{\alpha}}$ are derived in Section \ref{sec:PrlmCalc}. More specifically, for
\begin{equs}
p >   \frac{6}{1 - 2\tilde{\alpha}} \qquad \text{and} \qquad \lambda_1 = 1 -  2 \tilde{\alpha} -  \frac{6}{p} 
\end{equs}
we get using  Corollary \ref{Cor:boundsonPSI}
\begin{equation}\begin{aligned}
\label{e:PMR1}
\Big( \E &  \big\|  \Psi^{  \theta(\bar{u})}  - \Psi_\eps^{\theta(\ue)} \big\|_{\cC^{ \tilde{\alpha}/2, \tilde{\alpha}}_{t_\eps }}^p \Big)^{  \frac{1}{p}}   \\
& \leq   \Big( \E   \big\|  \Psi^{  \theta(\bar{u})}  - \Psi^{\theta( \ue )} \big\|_{\cC^{ \tilde{\alpha}/2, \tilde{\alpha}}_{  t_\eps }}^p \Big)^{  \frac{1}{p}}  + \Big( \E   \big\|  \Psi^{  \theta(\ue)}  - \Psi_\eps^{\theta(\ue)} \big\|_{\cC^{ \tilde{\alpha}/2, \tilde{\alpha}}_{  t_\eps }}^p \Big)^{  \frac{1}{p}} \\
& \ls t^{  \frac{\lambda_1}{4} } \,   \Big( \E  \big\|  \theta(\bar{u})  - \theta( \ue ) \big\|_{\cN_{  t_\eps }}^p \Big)^{  \frac{1}{p}} + \eps^{ \lambda_1 \alpha} \, \E\Big(   \big\|  \theta( \ue ) \big\|_{\cC^{\alpha}_{  t_\eps }}^p \Big)^{  \frac{1}{p}} \\
& \ls  t^{  \frac{\lambda_1}{4}} \,  \Big( \E   \big\|  \bar{u}  - \ue  \big\|_{\cN_{  t_\eps }}^p \Big)^{  \frac {1}{p}} + \eps^{\lambda_1 \alpha} \,  \Big( \E\Big( 1+  \big\|   \ue  \big\|_{\cC^{ \alpha}_{  t_\eps }}^p \Big) \Big)^{  \frac{1}{p}}  \\
& \ls t^{  \frac{\lambda_1}{4}} \mathcal{E}^{\tilde{\alpha}}(t) + \eps^{ \lambda_1 \alpha} (  K+1).
\end{aligned}\end{equation}
In passing from the second to the third line, we have used \eref{e:PsiT} and \eref{e:Psiep} as well as the linearity of the map $\theta \mapsto \Psi^{\theta}$. When passing from the third to the fourth line, we have used the fact that the $\cC^1$-norm of $\theta$ is bounded by a deterministic constant
depending on $  K$. 

In particular, by choosing $p$ large enough and $\kappa$ small enough, the rate ${\lambda}_1 \alpha$ can be increased arbitrarily close to $ \frac12 - \tilde{\alpha}$.

In order to get a bound on the second quantity in $I_2^{\tilde{\alpha}}$ we write 
\begin{equs}
 \Big( \E &   \big\| R^{\theta(\bar{u} )} - R_\eps^{\theta(\ue)}  \big\|_{\bB^{2\tilde{ \alpha}}_{[\eps^2,t_\eps], \tilde{\alpha}/2}  }^p  \Big)^{  \frac{1}{p}} \label{e:PMR2}\\
 & \leq \Big( \E   \big\| R^{\theta( \bar{u})} - R_\eps^{\theta(\bar{u})}  \big\|_{\bB^{2  \tilde{\alpha}}_{[\eps^2,t_\eps], \tilde{\alpha}/2 }  }^p  \Big)^{  \frac{1}{p}} +  \Big( \E   \big\| R_\eps^{\theta(\bar{u} )} - R_\eps^{\theta(\ue)}  \big\|_{\bB^{2 \tilde{\alpha}}_{[\eps^2, t_\eps] , \tbp/2 }  }^p  \Big)^{  \frac{1}{p}},
\end{equs}
where we use the notations $R^{\theta}$ and $\Rte$ from Section \ref{sec:PrlmCalc}.

The first term on the right hand side of  \eref{e:PMR2} can be bounded directly using Corollary \ref{cor:Rte}. Actually, using the time regularity of $\theta(u)$ for all times in  $[0, \rho_{  K,\eps}]$ we even get a bound without blowup. Then for any
\begin{equs}
\lambda_2  <  \alpha\frac{\alpha + \aX -2 \ta }{\alpha +\aX}
\end{equs}
and for $p$ large enough we obtain 
\begin{equs}
\Big( \E    \big\| R^{\theta( \bar{u})} - R_\eps^{\theta(\bar{u})}  \big\|_{ \bB^{2  \tilde{\alpha}}_{[\eps^2,t_\eps], \tilde{\alpha}/2 }}^p  \Big)^{  \frac{1}{p}}  & \ls \eps^{\lambda_2 } \, \Big( \Big[ \E   \big\|  \theta( \bar{u} ) \big\|_{\cC^{\alpha/2,\alpha}_{ [\eps^2, t_\eps] }}^p \Big]^{  \frac{1}{p}}  + \Big[\E  \big\|  \theta( \bar{u} ) \big\|_{\cC^{\alpha}_{  t_\eps }}^p \Big]^{  \frac{1}{p}}\Big)
 \\
& \ls   \eps^{ \lambda_2 } . 
\end{equs}
Here we have used the fact that the stopping time $\rho_{  K,\eps}$ is almost surely 
smaller than the stopping time $\rho^X_{\eps,  K}$ defined in \eref{e:DefTau} and than 
$\rho^X_\infty$ defined in \eqref{e:DefTau11}.  Note that as above by choosing $\kappa$ 
small enough (and $p$ large enough, which is already implicit in the expression for 
$\lambda_2$), the rate $\lambda_2 $ can be increased arbitrarily close to  
$\frac{1}{2}- \tilde{\alpha}$.

For the remaining term in \eref{e:PMR2}, using \eqref{e:RRbo1} we get for any $ \lambda_3  <  \frac12 (\aX -\tilde{\alpha})$ and for $p$ large enough that
\begin{equs}
\Big( \E  \big\|& R_\eps^{\theta( \bar{u})} - R_\eps^{\theta( u_\eps)}  \big\|_{\bB^{2\tilde{\alpha}}_{[\eps^2, \tau] ,\tilde{\alpha}/2}}^p  \Big)^{\frac{1}{p}}      \\   
& \ls t^{\lambda_3 }  \Big(\Big(\E \, \| \theta(\bar{u})  - \theta( \ue ) \|_{\cC_{[\eps^2,\tau]}^{\alpha/2,\alpha}}^p\Big)^{\frac{1}{p}} + \Big(\E \, \| \theta(\bar{u})  - \theta( \ue )\|_{\cN_{\tau}}^p \Big)^{\frac{1}{p}}\Big)\\
& \ls t^{  \lambda_3} \mathcal{E}^{\tilde{\alpha}}(t) .
\end{equs}

The bounds on $I_3^{\tilde{\alpha}}$ and $I_4^{\tilde{\alpha}}$ are provided in Section \ref{sec:GauFluc}. Using Proposition  \ref{prop:Phi-bound}  twice we get 
\begin{equs}
  I_3^{\tilde{\alpha}}(t)  &=
 \Big( \E   \big\| \Phi^{F(\bar{u}) }  - \Phi_\eps^{F( \ue)}  \big\|_{\cC^{\tilde{\alpha}/2, \tilde{\alpha}}_{t_\eps}}^p  \Big)^{  \frac{1}{p}} \\
&   \quad + \Big( \E  \Big(  \sup_{0  < s \leq t_\eps } s^{  \frac{\tbp}{2}} \big|\Phi^{F(\bar{u}) }(s)  - \Phi_\eps^{F( \ue )}(s)   \big|_{\cC^{2 \tilde{\alpha}}}   \Big)^p  \Big)^{  \frac{1}{p}}\label{e:i4BO}\\
  &  \ls \,   t^{  1 - \frac{\tilde{\alpha}}{2} } \Big( \E  \| \bar{u} - u_\eps \|_{\cC_{t_\eps}^\ta}^p \Big)^{  \frac{1}{p}} +  \eps^{  1 - \frac{\ta}{2}} \;.
  \end{equs}
Here we have used that, as a consequence of our definitions, 
the norms $\| \bar{u}\|_{\cC^{ \alpha}}$ and  $\| u_\eps \|_{\cC^{\alpha}}$ 
are bounded by $K$  before the stopping time $\rho_{  K,\eps}$.

 Then using Proposition \ref{prop:Ups-bound}  and Proposition \ref{prop:rand-fluc} we get  
\begin{equs}
 I_4^{\tilde{\alpha}}(t)  &=  \,  \Big( \E   \big\|    \Upsilon^{\bar{u}}  - \Upsilon^{\ue}_\eps \big\|_{\cC^{\tilde{\alpha}/2, \tilde{\alpha}}_{t_\eps}}^p  \Big)^{  \frac{1}{p}}   +  \Big( \E  \Big(  \sup_{0  < s \leq t_\eps } s^{  \frac{\tbp}{2}} \big| \Upsilon^{\bar{u}}  - \Upsilon^{\ue}_\eps  \big|_{\cC^{2 \tilde{\alpha}}}   \Big)^p  \Big)^{  \frac{1}{p}} \\ 
&\ls \eps^{1 -\frac{ \tilde{\alpha}}{2}} +  \bigg( \E \sup_{t \in [0,T]} \Big|
  D_\eps \XX_\eps(t, \cdot) - \Lambda \Id 
  \Big|_{H^{-\et}}^p\bigg)^{\frac1p} \\
& \qquad   + t^{1 -\frac{ \tilde{\alpha}}{2}}\Big( \Big[ \E  \big\|  \bar{u}  -\ue  \big\|_{\cC^{\tilde{\alpha}}_{  t_\eps }}^p \Big]^{  \frac{1}{p}} +  \Big[\E   \big\|  \theta(\bar{u})  - \theta( \ue ) \big\|_{\cC^{\tilde{\alpha}}_{  t_\eps }}^p \Big]^{  \frac{1}{p}} \Big) \\
& \ls \eps^{\eta- \kappa} +  t^{1 -\frac{ \tilde{\alpha}}{2}}\mathcal{E}^{\tilde{\alpha}}(t).
 \end{equs}

We use Proposition \ref{prop:Xi-bounds} and Proposition \ref{prop:Xi-time}, the main results of Section \ref{sec:RouPathEst}, to  bound $I_5^{\tilde{\alpha}}$. We get
\begin{equ}[e:PMR3]
\Big( \E  \big\|  \Xi^{\bar{u}}  -\Xi_\eps^{\ue}  \big\|_{\cC^{\tilde{\alpha}/2, \tilde{\alpha}}_{[\eps^2,t_\eps]}}^p  \Big)^{  \frac{1}{p}} + \Big( \E   \big\| \Xi^{\bar{u}}  -\Xi_\eps^{\ue}  \big\|_{\cC^{ \tilde{\alpha}}_{t_\eps}}^p  \Big)^{  \frac{1}{p}}  
\ls  \big( \E \CE_\eps^p\big)^{  \frac{1}{p}} \,  t^{  \frac{1 - 2\tilde{\alpha}- \kappa}{2} } + \eps^{ 3 \alpha -1 }.
\end{equ}

Using Proposition \ref{prop:Xi-bounds} again for $\gamma = 2 \tilde{\alpha}$, we 
get a similar bound for the remaining term:
\begin{equs}[e:PMR4]
 \Big[ \E   \Big( \sup_{0  < s \leq t_\eps } s^{-  \frac{\tbp}{2}} \big| \Xi^{\bar{u}}(s) - \Xi^{\ue}_\eps (s)    \big|_{\cC^{2 \tilde{\alpha}}} \Big)^p  \Big]^{  \frac{1}{p}} &\ls  \big( \E \CE_\eps^p\big)^{  \frac{1}{p}} \, t^{  \frac{1 -2 \tilde{\alpha}- \kappa}{2} } \\
 & \qquad + \eps^{ 3 {\alpha} - 1 } + \eps^{1 - 2 \tilde{\alpha} - \kappa}. \qquad
\end{equs}
Note that we have used again that, thanks to the stopping time $\rho_{  K,\eps}$, all the relevant norms are bounded almost surely. In particular, the constants that are suppressed in the $\ls$ notation do depend on $K$. As above, in \eref{e:PMR3} and
\eref{e:PMR4} we have used the notation 
\begin{equs}
 \CE_\eps \, = \, &   \|X-X_\eps \|_{\cC^{\tilde{\alpha}}_T} +  \|\XX-\XX_\eps \|_{\bB^{2\tilde{\alpha}}_T}    +    \| \bar{u}- u_\eps \|_{\cC^{\tilde{\alpha}}_{t_\eps}} +   \|\theta(\bar{u})- \theta(u_\eps) \|_{\cC^{\tilde{\alpha}}_{[\eps^2,t_\eps]}} \\
 &\qquad +  \|\Rub - \Rue\|_{\bB^{2\tilde{\alpha}}_{[\eps^2, t_\eps], \tbp/2 }}  \;. 
\end{equs}
The quantity $\E[ \CE_\eps^p]^{  \frac{1}{p}} $ can be bounded by 
\begin{equs}[e:PMR5]
\big( \E \CE_\eps^p\big)^{  \frac{1}{p}}  \ls  & \,  \Big( \E  \|X-X_\eps \|_{\cC^{\tilde{\alpha}}_{t}}^p  \Big)^{  \frac{1}{p}} +   \Big( \E  \|\XX-\XX_\eps \|_{\bB^{2\tilde{\alpha}}_{t}}^p \Big)^{  \frac{1}{p}}    +   \Big( \E  \| \bar{u}- u_\eps \|_{\cC^{\alpha}_{t_\eps}}^p \Big)^{  \frac{1}{p}}  \\
&+  \Big( \E   \| \Rub - \Rue\|_{\bB^{2\tilde{\alpha}}_{[\eps^2,t_\eps], \tbp/2}}^p \Big)^{  \frac{1}{p}}    \\
\ls  & \,   \Big( \E  \|X-X_\eps \|_{\cC^{\tilde{\alpha}}_{t}}^p  \Big)^{  \frac{1}{p}} +  \Big( \E  \|\XX-\XX_\eps \|_{\bB^{2\tilde{\alpha}}_{t}}^p \Big)^{  \frac{1}{p}} + \mathcal{E}^{\tilde{\alpha}}(t)  .
\end{equs}
Here,  we have used the fact that $\bar{u}^{\prime}= \theta (\bar{u})$ and $\ue^{\prime}= \theta (\ue)$ as well as the bound
\begin{equs}
 \| \theta(\bar{u}) - \theta(\ue) \|_{\cC^{\alpha}_{[\eps^2, t_\eps]}} \leq \,  \big| \theta \big|_{\cC^1} \| \bar{u} -\ue \|_{\cC^{\alpha}_{t_\eps}} + \big| \theta \big|_{\cC^2} \| \bar{u} \|_{\cC^{\alpha}_{t_\eps}} \| \bar{u} -\ue \|_{\cN_{t_\eps}} .
\end{equs}
For the second-order information  $\XX_\eps$ and $\XX$ of our rough paths, 
Corollary~\ref{cor:GRP} implies that for
$\lambda_4  < 1 - 2 \tilde{\alpha}
$
we get
\begin{equ}[e:PMR6]
\Big( \E     \big\| X_\eps -X   \big\|_{\cC_{t}^{\alpha/2,\alpha}}^p \Big)^{  \frac{1}{p}}   \, \ls \,      
  \eps^{  \frac{\lambda_4}{2}} \quad \text{and} \quad 
\Big( \E  \big\| \XX_\eps -\XX   \big\|_{ \bB^{ 2 \tilde{\alpha}}_{t}  }^p  \Big)^{1/p}
 \, \ls \,      
  \eps^{  \frac{\lambda_4}{2} }\;.
\end{equ}
So finally, combining the estimates \eref{e:DefEE}, \eref{e:contris}, \eref{e:PMR0B}, \eref{e:PMR1}, \eref{e:PMR2}, \eref{e:PMR3}, \eref{e:PMR4} and \eref{e:PMR6}, we obtain
\begin{equs}[e:PMR9]
\mathcal{E}^{\tilde{\alpha}}(t) \ls t^{\tilde{\gamma}} \mathcal{E}^{\tilde{\alpha}}(t)  + \eps^{\gamma} + \big| u^0 -  u^0_\eps \big|_{\cC^{\tilde{\alpha}}}.
\end{equs}
Here the exponents $\tilde{\gamma}, \gamma>0$ are the minima of the corresponding exponents in the above calculations. Note that  $\gamma$ depends on $\kappa$ and $p$. and we have  $\gamma = \gamma(p,\tilde{\alpha},\kappa)$ increases to  $\frac{1}{2}- \tilde{\alpha}$ as $p \uparrow \infty$ and  $\kappa \downarrow 0$.

By choosing $t = t_{*}$ small enough we can absorb the first term on the right-hand side of \eref{e:PMR9} into the left-hand side. Then we get for some constant $C_*$,
\begin{equs}[e:UBB1]
\sup_{0 \leq s \leq t_* } \mathcal{E}^{\tilde{\alpha}}(s)  \leq  C_* \eps^{\gamma} +  C_* \big| u^0 -  u^0_\eps \big|_{\cC^{\tilde{\alpha}}}.
\end{equs}
Note that the specific values of $t_*$ and $C_*$ only depend on $K$ and $\hat K$.

We now iterate this argument taking $\bar{u}(t_*)$ and $\ue(t_*)$ as new initial data.
The definition of the stopping times  $\sigma^u_K$ and $\sigma^u_\eps$ in \eref{sigmau} and \eref{sigmaue}  ensure that for  $t< \rho_{K,\eps}$  we have the bounds   $|\bar{u}(t )|_{\cC^\alpha}\leq K$ and $|\ue(t)|_{\cC^\alpha}\leq K$, but in order to be able to iterate the argument we need a bound in the slightly stronger $\cC^{\aX}$-norm. 
For this, we make use of the following trick: we know already that $\bar{u}(t)$  is a  controlled rough path and that, for any pair $x,y$, we have the bound
\begin{equs}
\delta \bar{u} (t;x,y ) &\, = \, \theta \big( \bar{u} (t, x ) \big)\,  \delta X (t; x,y )  + \Rub(t; x,y ). 
\end{equs}
Using this decomposition, we can conclude that in order to prove boundedness of $\bar{u}(t)$ in $\cC^{\aX}$, it is sufficient to have bounds on  $\theta(\bar{u}(t))$ in $\cN$, on $X$ in $\cC^{\aX}$ as well as on $\Rue(t)$ in $\bB^{\aX}$. These norms are all bounded by deterministic constants due to the definition of the stopping time. Hence we can conclude that for $t \leq \rho_{\eps,K}$ we have $|\ue(t_*) |_{\cC^{\aX}} \leq \sup_{|u| \leq K+1} |\theta(u) | K + K $. Now if we choose $\hat{K} = \sup_{|u| \leq K+1} |\theta(u) | K + K $ we can restart the process at $t_*$ and obtain the estimate
%
%
%In order to apply the bound \eref{e:UBB1} to the restarted process hen restarting the process at $t_*$ also the process one also has to restart the processes $X$ and $X_\eps$. Denote these restarted processes by $\hat{X}$ and $\hat{X}_\eps$. Then for $s >t_*$ we have the identities
%\begin{equs}
%X(s) - \hat{X}(s) = S(s- t_*)  \,X(t_*) \quad \text{and} \quad X_\eps(s) - \hat{X}_\eps(s) = S_\eps(s- t_*) \, X_\eps(t_*). 
%\end{equs}

%At this point we will make use of the extra degree of freedom that we have introduced in the larger constant $\hat{K}$. We choose  $\hat{K} \geq \sup_{|u| \leq K+1} |\theta(u) | K + K$ and conclude that if $t_* \leq \rho_{K,\eps}$ we have $|\ue(t_*) |_{\cC^{\aX}}$ the required bound on the new initial condition.

%
%
%Unfortunately, we do not have an a priori bound on the $\cC^{\alpha_\star}$ norm of $u_\eps(t_*\wedge \rho_{K,\eps})$ at our disposal (but only on the $\cC^{\alpha}$ norm!). We resolve this problem for the moment by restricting ourselves to the set 
%\begin{equ}[e:DefXX]
%\mathfrak{X}_{t_*}:=  \big\{ \big| \bar{u} - u_\eps |_{C^{\alpha_\star}} \leq 1   \big \}.
%\end{equ} 
%On this set,  by definition we have that $ \big| u_\eps |_{C^{\alpha_\star}} < K+1$ which is enough to iterate the bound.
%
%
%Actually, one still has to be a bit careful. By restarting at $t_*$ one obtains the estimate
%
\begin{equs}
\Big( \E   &  \big\| \bar{u} - u_\eps \big\|_{\cC^{\tilde{\alpha}/2, \tilde{\alpha}}_{ [t_*+\eps^2 , 2 t_* \wedge \rho_{K,\eps}]}}^p  \Big)^{1/p} + \Big( \E    \big\| \bar{u} - u_\eps \big\|_{\cC^{\tilde{\alpha}}_{ [t_* , 2 t_* \wedge \rho_{K,\eps}]}}^p    \Big)^{1/p}  \\
&\qquad + \Big( \E   \big\| \mathcal{R}_{\bar{u}} - \mathcal{R}_{\ue}   \big\|_{\bB^{2 \tilde{\alpha}}_{[t_* + \eps^2,  2 t_* \wedge \rho_{K,\eps}], \tbp/2}  }^p    \Big)^{1/p} \\
& \leq C_* \eps^\gamma + C_* \, \big( \E | \bar{u}(t_* \wedge \rho_{K,\eps}) - u_\eps(t_* \wedge \rho_{K,\eps}) |_{\cC^{\tilde{\alpha}}}^p \,   \big)^{1/p} \\
& \leq  \big( C_*  + C_*^2 \big) \eps^\gamma  + C_*^2  |u^0- u_\eps^0 |_{\cC^{\tilde{\alpha}}}, 
\end{equs}
which can then be iterated recursively. Note that the values of $C_*$ and $t_*$ remain constant throughout the recursion (because they only depend on $K$ and $\hat K$) so that the final time $T$ is reached within a finite number of steps. 

The bound one obtains in this way not yet the desired result because of the  weight $(s- kt_*)^ {\tilde{\alpha}}$ for $k =1, 2, \ldots$. Note that in the definition of $\mathcal{E}^{\tilde{\alpha}}$ the $\bB^{2 \tilde{\alpha}}$-norm of $\Rub$ may only blow up at $\eps^2$ but not at every multiple time $ kt_* + \eps^2 $. For $\eps$ small enough this issue can be avoided if we additionally restart the process at  times $ \frac{2k +1}{2} t_*$, and then for every $s$ take the better of the two bounds obtained in this way.
\end{proof}

\begin{proof}[Proof of Theorem \ref{thm:main-one}]
In order to conclude Theorem \ref{thm:main-one} it is sufficient to show that we have 
\begin{equ}[e:UBB2]
\lim_{K \uparrow \infty} \lim_{\eps \downarrow 0} \P\Big[ \sup_{0 \leq s \leq \tau_K^{*}} \big\| \bar{u} - \ue \big\|_{\cN} \geq \eps^{\tilde{\gamma}} \Big]   =0.
\end{equ}
Indeed, then the sequence $\tau_\eps$ can be chosen as a suitable diagonal sequence. 

Recall the definitions of the stopping times $\rho_{K,\eps}$ and $\sigma_K$ in \eref{e:Stoppzeiten} and \eref{e:Stoppzeiteneps}.

In order to see \ref{e:UBB2} we write for any $\bar{K} >K$
\begin{equs}
\P\Big[ \sup_{0 \leq s \leq \tau_K^{*}}  \big\| \bar{u} - \ue \big\|_{\cN} \geq \eps^\gamma \Big] &\, \leq \, \P\Big[ \sup_{0 \leq s \leq \rho_{\bar{K},\eps}^{*}}  \big\| \bar{u} - \ue \big\|_{\cN} \geq \eps^{\tilde{\gamma}} \Big] \label{e:UBB3}\\
&\qquad + \P \Big[ \rho_{\bar{K},\eps} < \sigma_{\bar{K}} \Big] + \P \Big[  \sigma_{\bar{K}} < \tau_K \Big].
\end{equs}
Theorem \ref{thm:main-one-mod} above and Chebyshev's inequality directly imply that the first term goes to zero. In order to obtain the optimal rate, we choose $\tilde{\alpha}$ to be as small as possible, i.e. just a bit larger than $  \frac{1}{3}$.  In particular, by choosing $\kappa$ small enough we can increase $\gamma$ up to arbitrarily close to $\frac16$.

According to the definition of the stopping times we have 
\begin{equs}
\P & \Big[ \rho_{\bar{K},\eps} < \sigma_{\bar{K}} \Big] \\
& =  \P \Big[  \|X - X_\eps \|_{\cC^{\aX/2, \aX}_{\rho_{\bar{K},\eps}}} \geq 1,  \text{ or }  \|\XX - \XX_\eps \|_{\bB^{2\alpha}_{\rho_{\bar{K},\eps}}}  \geq 1, \label{e:248}\\
&  \quad \text{ or }  \sup_{0 \leq t \leq \rho_{\bar{K},\eps}}\Big|  X(t) -X_\eps(t)  \Big|_{\cN}  \geq \eps^{\aX}, \text{ or } \sup_{0 \leq t \leq \rho_{\bar{K},\eps}} \big|  X_\eps(t) |_{\cD^{\aX, \eps}} \geq \bar{K},\quad   \\
&  \quad \text{ or }  \sup_{0 \leq t \leq \rho_{\bar{K},\eps}}\Big|  D_\eps \XX_\eps(t) - \Lambda \Id  \Big|_{H^{-\et}}  \geq 1, \text{ or }  \big\|  \bar{u} - \ue \|_{\cC^{\alpha/2, \alpha}_{[\eps^2, \rho_{\bar{K},\eps}]}} \geq 1,\quad   \\
&  \quad  \text{ or }  \big\|  \bar{u} - \ue \|_{\cC^{  \alpha}_{\rho_{\bar{K},\eps}}} \geq 1,    \text{ or }   \sup_{\eps^2 < t \leq \rho_{\bar{K},\eps}}  (t- \eps^2)^{  \frac{\alpha}{2}}  |  \Rub(t) -\Rue(t)  |_{2\alpha} \geq  1, \\
& \quad  \text{ or} \!\!  \sup_{0 < t \leq \rho_{\bar{K},\eps}} (t-\eps^2)^{  \frac{\tbm}{2}}  |  \Rub(t) -\Rue(t)  |_{2\tilde{\alpha}} \geq  1 \Big].
\end{equs}
We have already provided all the bounds that imply that for any $\bar{K}$ this probability goes to zero. In fact, the bounds for $X - X_\eps $,  $|\XX - \XX_\eps|_{\bB^{2\alpha}_t}$ and $ \big|  D_\eps \XX_\eps(t) - \Lambda \Id  \big|_{H^{-\et}} $ are independent of $\bar{K}$ and given in Corollary \ref{cor:GRP},  Lemma \ref{le:strangenorm} and Proposition \ref{prop:rand-fluc}.

The bounds for the remaining terms in \eref{e:248} follow from applying Theorem \ref{thm:main-one-mod} again, once with $\tilde{\alpha}= \alpha$ and once for arbitrary $\tilde{\alpha}$. Note that it is crucial for this argument, that we allow for the choice $\tilde{\alpha} = \alpha$. In this case the convergence is very slow, but this does not matter.

Finally, we write for the last term on the right-hand side of \eref{e:UBB3} that  
\begin{equs}
\P &\Big[  \sigma_{\bar{K}} < \tau_K \Big] =  \P \Big[  \|X  \|_{\cC^{\aX/2, \aX}_{\tau_K}} \geq \bar{K},  \text{ or }  \|\XX  \|_{\bB^{2\alpha}_{\tau_K}}  \geq \bar{K},  \text{ or }  \big\|  \bar{u}  \|_{\cC^{\alpha/2, \alpha}_{\tau_K}} \geq \bar{K},  \\
& \qquad  \text{ or }   \sup_{0 <t \leq \tau_K}  t^{  \frac{\beta}{2}}  |  \Rub(t)  |_{2\alpha} \geq   \bar{K},  \text{ or }   \sup_{0 <t \leq \tau_K}  t^{  \frac{\tbm}{2}}  |\Rub(t)  |_{2\tilde{\alpha}} \geq  \bar{K}  \Big].
\end{equs}
It follows from the bounds in Corollary \ref{cor:GRP} that the probability of $ \|X  \|_{\cC^{\aX/2, \aX}_{\tau_K}} \geq \bar{K}$  and the probability of  $ \|\XX  \|_{\bB^{2\alpha}_t}  \geq \bar{K}$ go to zero as $\bar{K}$ goes to $\infty$. The same statement about the probabilities involving $\bar{u}$ and $\Rub$ follows from the global well-posedness of the equation with bounded $g$ and $\theta$. The details of this calculation can be found in the proof of Theorem 3.5 in  \cite{HW10} and will be omitted here. 

This finishes the proof of our main result, Theorem  \ref{thm:main-one}.
\end{proof}

\section{The stochastic convolution}
\label{sec:PrlmCalc}
In this section we provide the necessary bounds on the stochastic convolutions $\Psi^{\theta(u)}$ and $\Psi_\eps^{\theta(\bar{u})}$. We will adopt a slightly more general framework than the one adopted in Section \ref{sec:OFP}. Actually, we will fix an adapted $L^2[-\pi, \pi]$-valued proces $( \theta(t))_{t \geq 0}$ and consider the stochastic convolutions with the heat semigroup, i.e.
\begin{equ}
\Pte (t) \, = \, \int_0^t S_{\eps}(t-s) \, \theta(s) \, H_\eps \, dW(s)  \quad \text{ and } \quad \Pt (t) \, = \, \int_0^t S(t-s) \, \theta(s) \, dW(s).
\end{equ}
As in Section \ref{sec:OFP}, the Gaussian case $\theta \equiv 1$ will play a special role. We write 
\begin{equ}
X_\eps (t) \, = \, \int_{-\infty}^t \! S_\eps(t-s) \,  \Pi H_\eps \, dW(s)  \quad \text{ and } \quad X (t) \, = \, \int_{-\infty}^t \! S(t-s) \, \Pi dW(s).
\end{equ}
Here we have extended the cylindrical Brownian motion $W$ to negative times in order to ensure that $X$ and $X_\eps$ are stationary.

It will be useful to consider the Fourier expansion of $X_\eps$ given by
\begin{equs}[e:XSum]
X_\eps (t,x) \, = \,&  \frac{1}{\sqrt{2 \pi}} \sum_{k \in \Z^\star} \int_{-\infty}^t e^{ikx} \,  e^{-  k^2 f(\eps k)  (t-s) }  \, h( \eps k) \, d w_k(s) \\
= \, & \sum_{k \in \Z^\star}  q_\eps^k  \,   \xi^{k}_\eps(t) \, e^{ikx}\,.  
\end{equs}
Here the $w_k$ are $\C^n$-valued two-sided Brownian motions (i.e.\ real and imaginary part of every component are independent real-valued Brownian motions so that $\E  | w_k^i(t)|^2 =  |t|$), which are independent up to the constraint $w_k = \bar{w}_{-k}$ that ensures that $X_\eps$ is real-valued. Furthermore, we use the notation
\begin{equ}[e:Defq]
q_\eps^k  \, = \,
  \frac{h(\eps k) }{  |k|   \,  \sqrt{ 4 \pi f(\eps k)} } %\qquad \text{for } k \neq 0
\end{equ}
and  the $\xi^k_\eps$ are centred stationary  $\C^n$-valued Gaussian processes, independent up to $\xi^k_\eps = \bar{\xi}^{-k}_\eps$, so that for any $t>0$,
\begin{equ}[e:Defxi]
\E \big( \xi_\eps^{k} (0) \ot  \xi_\eps^{-k} (t) \big) \, = \, \KK_k^{t} \Id ,
\end{equ}
where
\begin{equs}[e:Defxi1]
  \KK_k^{t} = 
   %\left\{ \begin{array}{ll}
  e^{-f(\eps k)k^2 t} .
   %\qquad  \text{$k \neq 0$}.%\\
  %t,
   %& \text{$k =0$}.\end{array} \right.
\end{equs}

The series decomposition \eref{e:XSum} can be used to define the iterated integral 
\begin{equ}[e:ItInt]
\XX_\eps(t; x,y) \, = \, \int_x^y  \big( X_\eps(t, z) - X_\eps(t,x) \big)  \,  \otimes d_z X_\eps(t,z). 
\end{equ}
In fact, for fixed $t, x,y$ this integral can simply be defined as the limit in $L^2(\Omega)$ as $N\to \infty$ of the double series
\begin{equs}[e:Xsum]
\XX^N_\eps(t;x,y) \, = \,& \sum_{0 < |k|, |l| \leq N }  \xi_\eps^k (t) \otimes  \xi_\eps^l(t) \,  q_{\eps}^k \,  q_\eps^l \,  \int_x^y \big(  e^{ikz} - e^{ikx}  \big) \,  i l   \, e^{ilz} \, dz  \\
= \,&\sum_{0 < |k|, |l| \leq N }    \xi_\eps^k (t) \otimes   \xi_\eps^l(t)  \,  e^{i(k+l)x} \,   q_{\eps}^k \,  q_\eps^l  \,    I_{k,l}(y-x), 
\end{equs}
where 
\begin{equ}[e:DefI]
I_{k,l}(y) = \begin{cases}
 \frac{l}{k+l} \big(  e^{i(k+l) y}  -1\big) -   \big(  e^{il   y}  -1\big)  \qquad &\text{for } k \neq -l, \\
 \qquad  i l y  -  \big(  e^{il   y}  -1\big)   &\text{for }  k = -l . 
 \end{cases}
\end{equ}
The regularity properties of $\XX_\eps$ are discussed in Lemma \ref{le:GRP}. Note however, that the regularity of $X_\eps$ is not sufficient to give a pathwise  argument for the existence of \eref{e:ItInt}. Also the series does not converge absolutely in $L^2(\Omega)$ so that the symmetric choice of approximation matters. We will make use of some cancellation in Lemma \ref{le:GRP}.  Let us point out that the iterated integrals $\XX_\eps$ satisfy the consistency relation \eref{eq:cond-it-int}, and that for the symmetric part $\XX_\eps^+ :=  \frac{1}{2} \big( \XX_\eps + \XX_\eps^* \big)$ we have
\begin{equ}[e:Symmetric]
\XX_\eps^+ (t; x,y) \, = \,   \frac{1}{2} \big( X_\eps ( t,y) - X_\eps (t,x) \big) \otimes \big( X_\eps ( t,y) - X_\eps (t,x) \big) .  
\end{equ}
These relations can easily be checked for any $N$ and then follow by passing to the limit. The regularity results given in Lemma \ref{le:GRP} will then imply that for every $t$ the pair $(X_\eps(t, \cdot) , \XX_\eps(t; \cdot, \cdot))$ is indeed a geometric rough path in $x$ in the sense of definition \ref{def:RP}. 

A crucial tool to derive the moment estimates for $\XX_\eps$ will be the equivalence of moments for random variables in a given Wiener chaos, see  Lemma~\ref{lem:Nelson}. The decomposition \eref{e:Xsum} shows that $\XX_\eps$ is in the second Wiener chaos. Therefore, Nelson's estimate  implies that we can estimate all moments of $\XX_\eps$ in terms of the second moments.

Note that our definition of $\XX_\eps$ coincides with the canonical rough path lift of a Gaussian process discussed in \cite[Ch.~15]{FV10} and also used in \cite{Ha10,HW10}.  We prefer to work with the Fourier decomposition as it gives a direct way to prove moment bounds and avoids the notion of $2$-dimensional variation of the covariance, which seems a bit cumbersome in the present context. (See however \cite{Peter} for related calculations involving the two dimensional variation.)

A key step in the construction of solutions to \eref{e:SPDE} in  \cite{HW10} was to show that the process $\Pt(t, \cdot)$ is  controlled by $X(t, \cdot)$ as soon as  $\theta$ has a certain regularity. Its derivative process is then given by $\theta(t, \cdot)$.  

We will prove a similar statement for $\Pte$ and derive bounds that are uniform in $\eps$. For a given $\theta$ denote by $R^\theta_\eps$ the remainder in the rough path decomposition of $\Pte$ with respect to $X_\eps$, i.e.
\begin{equ}[e:DefRteA]
R^\theta_\eps(t;x,y) \, = \,  \big( \Pte(t,y)  - \Pte(t,x) \big) - \theta(t,x) \big( X_\eps(t,y) - X_\eps(t,x) \big)
\end{equ}
and  
\begin{equ}[e:DefRt]
R^\theta(t;x,y) \, = \,  \big( \Pt(t,y)  - \Pt(t,x) \big) - \theta(t,x) \big( X(t,y) - X(t,x) \big).
\end{equ}

The bounds on the spatial regularity of $\Rte$ are provided in Lemma \ref{le:mistakecorrected}. A key tool to derive these a priori bounds is provided by a higher-dimensional version of the Garsia-Rodemich-Rumsey Lemma that can be found in Lemma  \ref{lem:AnIsKol}.

For the bound on $\Pte$ we will impose a regularity assumption on $\theta$. For any stopping time $\tau$ recall the definition of the parabolic $\alpha$-H\"older norm in \eref{e:alphaHol1}. 
\begin{lemma}\label{le:PTE}
Let $\alpha \in (\frac13, \frac12)$. Let $\alpha_1, \alpha_2>0$ and $p \geq 2$ satisfy 
\begin{align}\label{eq:alp}
\alpha_1  \, < \, \frac{\lambda_1}4  -  \frac{1}{p}
  \,, \quad 
\alpha_2  \, < \, \frac{\lambda_2}2 -  \frac{1}{p}  
\end{align}
for some $\lambda_1, \lambda_2 >0$ with $\lambda_1 + \lambda_2 \leq 1$. Then for any stopping time $\tau \leq T$
\begin{align}\label{e:FF1}
\E  \big\| \Pte   & \big\|_{\cC_\tau^{\alpha_1} (\cC^{ \alpha_2})}^p    \, \ls \,      
     \E  \| \theta \|_{\cC_\tau} ^p \;.
\end{align}
and 
\begin{align}\label{e:FF2}
\E  \big\| \Pte -\Pt  & \big\|_{\cC_\tau^{\alpha_1} (\cC^{ \alpha_2})}^p    \, \ls \,      
  \eps^{(1 - \lambda_1 - \lambda_2) \alpha p}
     \, \E  \| \theta \|_{\cC^{\alpha}_\tau} ^p\;.
\end{align}
\end{lemma}

In our application of this lemma we shall need a small power of $T$ appearing in the right-hand side. This additional factor can be easily obtained by observing that as $\Pte(0) =0$ we have for any $0 < \kappa < \alpha_1$
\begin{equs}
\big\| \Pte  \big\|_{\cC_\tau^{\alpha_1- \kappa} (\cC^{ \alpha_2})} \leq   T^\kappa \big\| \Pte  \big\|_{\cC_\tau^{\alpha_1} (\cC^{ \alpha_2})}.
\end{equs}
Furthermore, we prefer to work with the space-time H\"older norms introduced in \eref{e:alphaHol1}, instead of working in spaces of  functions that are H\"older in time taking values in a H\"older space. To this end we observe that  
\begin{equs}
\|\Psi^\theta_\eps \|_{\cC^{\alpha/2, \alpha}_{\tau}} \leq \big\| \Pte  \big\|_{\cC_\tau^{\alpha/2} (\cN)} + \big\| \Pte  \big\|_{\cN_\tau (\cC^{ \alpha})}.
\end{equs}
In view of these remarks, the following result is an easy consequence of Lemma~\ref{le:PTE}.
\begin{corollary}\label{Cor:boundsonPSI}
Let $\alpha, \tilde{\alpha} \in (\frac13, \frac12)$. Suppose that $p$ satisfies 
\begin{equs}
 p > \frac{6}{1 - 2 \tilde{\alpha}}.
\end{equs}
Then for  $\lambda =  1 - 2  \tilde{\alpha} - \frac{6}{p}$ and for any stopping time $\tau \leq T$ we have
\begin{align}\label{e:PsiT}
\E \| \Pte  \|_{\cC ^{\tilde{\alpha}/2, \tilde{\alpha}}_\tau  }^p 
 \, \ls \,      
    T^{\lambda p/4}  \,  \E  \| \theta \|_{\cN_\tau} ^p \;
\end{align}
and 
\begin{align}\label{e:Psiep}
\E \| \Pte - \Pt  \|_{\cC ^{\tilde{\alpha}/2, \tilde{\alpha}}_\tau  }^p    \, \ls \,      
  \eps^{ \lambda p \alpha}\, 
   \,  \E  \| \theta \|_{\cC_\tau^\alpha}^p \;.
\end{align}
\end{corollary}
\begin{proof}[Proof of Lemma \ref{le:PTE}.]
We start by reducing the derivation of \eqref{e:FF1} and \eqref{e:FF2} to the case where $\tau = T$. In fact, for general $\tau$  we can define 
\begin{equ}
\tilde{\theta}(t) := \theta(t \wedge \tau) \qquad \text{and} \qquad \tilde{\Psi}^\theta_\eps := \int_0^t S_\eps(t-s) \,\tilde{\theta}(s) \, H_\eps \, dW_s.
\end{equ}
Observing that $\tilde{\Psi}^\theta_\eps = \Pte$ for $t \leq \tau$ we almost surely have the estimates
\begin{equ}
\big\| \Pte    \big\|_{\cC_\tau^{\alpha_1} (\cC^{ \alpha_2})} \leq \big\|  \tilde{\Psi}^\theta_\eps    \big\|_{\cC_T^{\alpha_1} (\cC^{ \alpha_2})} \quad \text{and} \quad  \big\| \Pte -\Pt   \big\|_{\cC_\tau^{\alpha_1} (\cC^{ \alpha_2})} \leq \big\|  \tilde{\Psi}^\theta_\eps - \tilde{\Psi}^\theta    \big\|_{\cC_T^{\alpha_1} (\cC^{ \alpha_2})}. 
\end{equ}
On the other hand we also have the almost sure identities
\begin{equ}
\| \tilde{\theta} \|_{\cN_T} = \| \theta \|_{\cN_\tau} \qquad \text{and}  \qquad \| \tilde{\theta} \|_{\cN_T} = \| \theta \|_{\cC^\alpha_\tau} .
\end{equ}
Hence the estimates \eqref{e:FF1} and \eqref{e:FF2} for general $\tau$ follow as soon as we have established them for $\tau=T$. We will make this assumption for the rest of the proof. 

Lemma \ref{lem:AnIsKol} applied to $F =  \Pte  $ will imply the desired bound \eref{e:FF1} as soon as we have established the following inequalities
\minilab{e:lotsofbounds}
\begin{equs}
\E  \big| \Pte(t_1 ,x) -\Pte(t_2,x) \big|^p  & \ls   \E  \| \theta \|_{\cN_T} ^p \,   (t_1-t_2)^{\frac{p}{4}},\label{e:Ptimebou} \\
\E  \big| \Pte(t, x_1) -\Pte(t,x_2) \big|^p  & \ls  \E  \| \theta \|_{\cN_T} ^p \,   |x_1-x_2|^{  \frac{p}{2}}, \label{e:Ptspacebou}    \\
\E  \big| \Pte(t ,x)  \big|^p  &\ls  \E  \| \theta \|_{\cN_T} ^p.\label{e:FF3}
\end{equs}
Then  Lemma \ref{lem:AnIsKol}, applied for  $F =  \Pte  -\Pt$ implies \eqref{e:FF2} as soon as we establish in addition that 
\minilab{e:lotsofbounds}
\begin{equ}[e:Ptdiffbou]
\E \big| \Pte(t ,x) -\Pt(t ,x) \big|^p  \ls   \E  \| \theta \|_{\cC_T^\alpha}^p \, \eps^{\alpha' p}
\end{equ}
for any $\alpha' < \alpha$. We state \eref{e:Ptimebou} and \eref{e:Ptspacebou} only for $\Pte$ noting that $\Pt$ is included as the special case $\eps= 0$. 

To see \eref{e:Ptimebou} we can write for $t_1 \geq t_2$,
\begin{equs}
\Pte(t_1,x) & -\Pte(t_2,x)  \\
 = \, & \int_0^{t_2}   \!\! \int_{-\pi}^{\pi}  H_\eps \Big( \,  \big( p_{t_1-s}^\eps (x- \cdot) -  p_{t_2-s}^\eps (x- \cdot )\big) \theta(s,\cdot)  \Big)   \,  dW(s) \\
& + \int_{t_2}^{t_1}  \int_{-\pi}^{\pi} H_\eps \,  \Big( \,  p_{t_1-s}^\eps (x - \cdot) \,  \theta(s,\cdot ) \,  \Big)  \,  dW(s). 
\end{equs}
Here we recall the definitions of the heat kernel $p_t^\eps$ in \eref{e:Defpt} and the smoothing operator $H_\eps$ in Assumption \ref{a:h}. 

Using the Burkholder-Davis-Gundy  inequality \cite[Theorem 3.28]{KS91} we get
\begin{equs}
\E  \big|&   \Pte(t_1,x) -\Pte(t_2,x) \big|^p   \label{e:pt1}
 \\  \ls \,&  \,  \E  \Big(  \int_0^{t_2} \!\! \int_{-\pi}^{\pi}
   \Big[ H_\eps \Big(\big( p_{t_1-s}^\eps (x- \cdot) -  p_{t_2-s}^\eps (x- \cdot )\big) \theta(s,\cdot ) \Big)(y)  \Big]^2 \,  ds \, dy \Big) ^{  \frac{p}{2}}  \qquad    
   \\ &+  \,  \E \Big( \int_{t_2}^{t_1}  \!\!  \int_{-\pi}^{\pi}  
   \Big[ H_\eps \Big( p_{t_1-s}^\eps (x - \cdot) \,  \theta(s,\cdot ) \Big)(y)  \Big]^2 \,  d s \, dy \Big)^{  \frac{p}{2}}\; .
\end{equs}
Observing that due to the boundedness of $h$ (see Assumption \ref{a:h}) the convolution with $H_\eps$ is uniformly bounded as an operator on $L^2$ we can bound the first expectation on the right-hand side of \eref{e:pt1} by a constant times 
\begin{equ}[e:pt2]
 \E  \|\theta \|_{\cN_T}^p  \, 
  \bigg( \int_0^{t_2} \!\!  \int_{-\pi}^\pi      \big( p_{t_1 -s}^\eps (x- y) -  p_{t_2-s}^\eps (x- y )\big)^2 \,  ds  \, d y \bigg)^{\frac{p}{2}}. 
\end{equ}
Using Parseval's identity the double integral in \eref{e:pt2} can be bounded by
\begin{equs}[e:newlabb]
  \sum_{k \in \Z} 
 & \Big(e^{- k^2 f(\eps k)\, ( t_1-t_2) } -1 \Big)^2  
  \int_0^{t_2} e^{-2  k^2 f(\eps k)\, ( t_2-s) } d s \\
 & \ls \,  \sum_{k \in \Z^\star} \Big(  k^2 f(\eps k) \, ( t_1-t_2)  \wedge 1 \Big)^2  \frac{1}{ k^2 f(\eps k) } \\
& \ls  \,  \sum_{k \in \Z^\star }   ( t_1-t_2)  \wedge  \frac{1}{ k^2 f(\eps k) }  \\ 
& \ls  \sum_{|k| \leq    ( t_1-t_2)^{-1/2}}   \!\!\!\!  ( t_1-t_2)  
+  \sum_{|k| >  ( t_1-t_2)^{-1/2} }   \frac{1}{ k^2} \ls  \,  ( t_1-t_2)^{  \frac{1}{2}}.
\end{equs}
Here in the second inequality we have used that $|a|^2 \leq |a|$ whenever $|a| \leq 1$, and in the third inequality we used the assumption that $f$ is bounded from below (see Assumption \ref{a:f}).

The second integral on the right-hand side of \eref{e:pt1} can be bounded in  a similar way by
\begin{equs}
 \E \Big( \int_{t_2}^{t_1}   \! \!  \int_{-\pi}^\pi  \,\Big( H_\eps \big(p_{t_1-s}^\eps &  (x - \cdot) \,  \theta(s,\cdot )\big) (y)  \Big)^2 \,  d s \, dy \Big)^{  \frac{p}{2}}  \, \\
  & \ls  \E  \|\theta \|_{\cN_T}^p  \, \Big(  \int_{t_2}^{t_1}  \!\! \int_{-\pi}^\pi  p_{t_1-s}^\eps (x - y)^2 \,  d s \, dz \Big)^{  \frac{p}{2}}.
\end{equs}
To bound the integral we calculate
\begin{equs}
 \int_{t_2}^{t_1} \int_{-\pi}^\pi  p_{t_1-s}^\eps (x - y)^2 \,   d s \, dy  \,  = \, &   \sum_{k \in \Z }   \int_{t_2}^{t_1}  e^{-2  k^2 f(\eps k)\, ( t_1-s) }\, d s \\
 \ls  \, &  \sum_{k \in \Z^\star}     \frac{  k^2 f(\eps k)\, ( t_1-t_2) \wedge 1  }{ k^2 f(\eps k) }  + (t_1-t_2) \\
  \ls \, &   (t_1-t_2)^{  \frac{1}{2}}.
\end{equs}
This finishes the proof of \eref{e:Ptimebou}.

Using the Burkholder-Davis-Gundy inequality and the uniform $L^2$-boundedness of the convolution with $H_\eps$ in the same way as before, the derivation of \eref{e:Ptspacebou} can be reduced to showing that 
\begin{equ}[e:uselat]
\int_0^t  \int_{-\pi}^\pi \big(  p_{t-s}^\eps(x_1-y) - p_{t-s}^\eps(x_2-y) \big)^2 dy \, d s \, \ls   \, \big| x_1-x_2 \big|.
\end{equ}
To prove the latter bound, we estimate
\begin{equs}
\int_0^t  \int_{-\pi}^\pi \big(&  p_{t-s}^\eps(x_1 -z) - p_{t-s}^\eps(x_2-z) \big)^2 dy \, d s  \\
&\, = \,   \sum_{k \in \Z^\star}  \big| e^{ik(x_2 - x_1)} -1 \big|^2   \int_0^t e^{-  k^2 f(\eps k) (t-s)} \, d s \\
& \ls \,    \sum_{k \in \Z^\star }\big( k |x_1 - x_2| \wedge 1 \big)^2  \frac{1}{ k^2 f(\eps k)} \\
& \ls \, \sum_{k \leq |x_1 - x_2|^{-1}} |x_1 - x_2|^2 +  \, \sum_{k > |x_1 - x_2|^{-1}}  \frac{1}{k^2}   \ls  |x_1 - x_2|.
\end{equs}
This shows \eref{e:Ptspacebou}.

The bound \eref{e:FF3} follows immediataly, by using the Burkholder-Davis-Gundy inequality in the same way as above.

In order to obtain \eref{e:Ptdiffbou} we write 
\begin{equs}
\Pte(t,x) -&  \Pt(t,x) \label{e:pt3}\\ 
 \, = \, & \int_0^{t} \int_{-\pi}^\pi     H_\eps \, \Big( \big( \,  p_{t-s}^\eps (x-\cdot )
    - p_{t- s}(x-\cdot) \big) \theta(s, \cdot) \Big)(y) \, dW(s, y)  \\
 & + \int_0^{t}  \int_{-\pi}^\pi    \big( \Id  -    H_\eps \big) \, \big(   p_{t- s}(x,\cdot)  \, \theta(s, \cdot)  \big)(y)\, dW(s, y).  
\end{equs}
The first term on the right-hand side of \eref{e:pt3} can be treated as before. Up to a constant its $p$-th moment is bounded by
\begin{equ}[e:pt4]
\E  \|\theta  \|_{\cN_T}^p   \Big( \sum_{k \in \Z} \int_0^t \Big(e^{-(t-s)  k^2 f(\eps k)} - e^{-(t-s) k^2 } \Big)^2 d s \Big)^{  \frac{p}{2}}.
\end{equ}
To get a bound on \eref{e:pt4} we write
\begin{equs}[e:pt5]
\sum_{k \in \Z} \int_0^t & \Big(e^{-(t-s) k^2 f(\eps k) } - e^{-(t-s) k^2 }  \Big)^2 d s  
\\ &= 
\sum_{k \in \Z^\star} \int_0^t 
e^{-2k^2 c_f s} 
   \Big(e^{-k^2(f(\eps k) - c_f) s } - e^{-k^2(1-c_f)s}  \Big)^2 d s\;
\\ &\leq \,
\sum_{k \in \Z^\star} \int_0^t
e^{-2 k^2 c_f s} 
   \big(1 \wedge s  k^2 |f(\eps k) -1| \big) ^2 d s\;.
\end{equs}

Recall that the constant $c_f$ is defined in Assumption \ref{a:f}. 
Using  Assumption \ref{a:f} on $f$ once more,  one can see that  for $|\eps k| < \delta$  we have
\begin{align*}
1 \wedge s  k^2 |f(\eps k) -1| 
\ls
  1 \wedge |sk^2 \eps k |.
\end{align*}
Hence, up to a constant the sum in \eref{e:pt5} can then be bounded by
\begin{equs}
\sum_{0 <|k| \leq \delta \eps^{-1} } \int_0^t &
e^{-2 k^2 c_f s} \, 
  s^2 \,  k^4 \,   (\eps^2 k^2 ) \,  d s + \sum_{ \delta \eps^{-1}  <|k|   } \int_0^t 
e^{-2 k^2 c_f s}  \,  d s  \, \ls \eps.
   \end{equs} 

Finally, to treat the second term in \eref{e:pt3} we need to impose a stronger regularity assumption on $\theta$. We have the estimate
\begin{equ}[e:pt10]
\big| \theta(s,  \cdot) \,  p_{t-s}(x - \cdot)  \big|_{H^{\alpha'}} \ls \big| \theta(s, \cdot) \,  \big|_{\cC^\alpha} \,  \big| p_{t-s}(x-\cdot) \big|_{H^{\alpha}} 
\end{equ}
which holds for every $\alpha' < \alpha$. In fact, to see \eref{e:pt10} write 
\begin{equs}
\int_{-\pi}^\pi   & \! \int_{-\pi}^\pi  \frac{ \big| \theta\,  p_{t-s} (x_1) - \theta \,  p_{t-s}(x_2)  \big|^2 }{|x_1 - x_2|^{2 \alpha' +1}} \dd x_1 \dd x_2 \\ 
&  \ls \sup_{x} | \theta(x)|^2 \, | p_{t-s} |_{ H^{\alpha' }}^2 +  |\theta |_{\cC^\alpha}^2  \int_{-\pi}^\pi    \! \int_{-\pi}^\pi  \frac{  |x_1 -x_2|^{2 \alpha}\,  p_{t-s}(x_2)^2  } {|x_1 -x_2|^{2 \alpha' +1} } \dd x_1 \dd x_2 \\
& \ls  \sup_{x} | \theta(x)|^2 \, | p_{t-s} |_{ H^{\alpha' }}^2 +  |\theta |_{\cC^\alpha}^2   |p_{t-s}|_{L^2}.
\end{equs}

Then the Burkholder-Davis-Gundy inequality yields 
\begin{equs}
\E  \bigg|  & \int_0^{t} \!\! \int_{-\pi}^\pi \Big( \big( \Id  -    H^\eps \big) \, 
 p_{t- s}(x -\cdot) \, \theta(s, \cdot)   \Big)(y) \, dW(s, y) \bigg|^p   \label{e:pt6} \\ 
& \ls \,  \E\Big(  \sup_{ s \in [0,T]}  \big| \theta(s, \cdot ) \big|_{\cC^\alpha}  \Big)^p   \Big(  \int_0^t  |p_{t-s}|_{H^\alpha}^2 \big|\Id - H^\eps \big|_{H^{\alpha'} \to L^2}^2 \, ds \Big)^{  \frac{p}{2}}. 
\end{equs}
Now for any $\alpha \in (0,1)$ 
\begin{equs}
|p_{t}(x - \cdot) |_{H^\alpha}^2 \, \eqsim \, &  \sum_{k \in \Z} \big( 1 + |k|^{2\alpha} \big) e^{- 2 t k^2}
 \ls   t^{-\frac{1}{2}- \alpha}.
\end{equs}
 On the other hand 
\begin{equs}
\big| \Id - H^\eps \,  \big|_{H^{\alpha'} \to L^2} \, \eqsim \,& \sup_{k \in \Z} \frac{1 - h(\eps| k |)}{1 + |k|^{\alpha'}}
\leq  \eps^{\alpha'}  \sup_{r \in \R} \frac{1 - h(|r|)}{\eps^{\alpha'} +|r|^{\alpha'}}
\ls     \eps^{\alpha'} ,  
\end{equs}
due to  Assumption \ref{a:h} on $h$. By assumption  $\alpha <  \frac{1}{2}$. Hence  the right-hand side of \eref{e:pt6} is integrable and we arrive at \eref{e:Ptdiffbou}. This finishes the proof. 
\end{proof}

As a next step we give bounds on the approximation of the Gaussian rough path $(X, \XX)$.
 
\begin{lemma}\label{le:GRP}
For any $\alpha \in (\frac13, \frac12]$, any $\eps \geq 0$ and any $t$, the pair $\big(X_\eps(t, \cdot), \XX_\eps(t; \cdot) \big)$ is a geometric $\alpha$-rough path in the sense of Definition \ref{def:RP}.

Furthermore, let $p \in [1, \infty)$ and let $\alpha_1, \alpha_2 > 0$ 
satisfy 
\begin{align} \label{eq:alp-weak}
\alpha_1  \, < \, \frac{\lambda_1}4
  \,, \quad 
\alpha_2  \, < \, \frac{\lambda_2}2
\end{align}
for some $\lambda_1, \lambda_2 >0$ with $\lambda_1 + \lambda_2 \leq 1$. Then we have for any $\eps \geq 0$
\begin{equ}[e:Xe1]
\E  \big\| X_\eps  \big\|_{\cC_T^{\alpha_1} (\cC^{ \alpha_2})}^p    \, \ls \,      
  1,
  \end{equ}

\begin{equ}[e:Xe]
\E  \big\| X_\eps -X   \big\|_{\cC_T^{\alpha_1} (\cC^{ \alpha_2})}^p    \, \ls \,      
  \eps^{\frac{1}{2} (1 - \lambda_1 - \lambda_2)   p},
  \end{equ}
and 
\begin{align}
\label{e:XXonly}
\E  \big\| \XX_\eps  \big\|_{\cC_T^{\alpha_1} (\bB^{ 2 \alpha_2})}^p    & \, \ls \,  1,
\\
\label{e:XX}
\E  \big\| \XX_\eps -\XX   \big\|_{\cC_T^{\alpha_1} (\bB^{ 2 \alpha_2})}^p  
& \, \ls \,      
  \eps^{  \frac{1}{2}(1 - \lambda_1 - \lambda_2)  p}\;.
\end{align}
\end{lemma}
  We will need uniform in time estimates  on $\XX_\eps$ and we will not make use of the H\"older in time regularity provided by Lemma \ref{le:GRP}. Therefore, we will actually use the following version of Lemma \ref{le:GRP}.
 %.
 \begin{corollary}\label{cor:GRP}
Let $\alpha \in (\frac13, \frac12)$. Suppose that $\lambda <  1 - 2 \alpha$. Then for any $p\geq 1$ and  $ T >0$ we have
\begin{align}
 \E \| X_\eps  \|_{\cC^{\alpha/2, \alpha}_T}^p    \, \ls \,&      
    T^{\frac{ \lambda p}{4}}, \label{e:Xe1b} \\
 \E \|     X_\eps -X  \|_{\cC^{\alpha/2, \alpha}_T}^p     \, \ls \,&      
      \eps^{  \frac{\lambda p}{2}  }. \label{e:Xeb}  
 \end{align}
Similarly,  we get 
\begin{align}
\label{e:XXonlyb}
\E  \big\| \XX_\eps  \big\|_{\bB_T^{ 2 \alpha}}^p    & \, \ls \,  T^{  \frac{\lambda p}{4}}\;,
\\
\label{e:XXb}
\E  \big\| \XX_\eps -\XX   \big\|_{\bB_T^{ 2\alpha }}^p  
& \, \ls \,      
  \eps^{   \frac{\lambda p}{2} }\;.
\end{align}
 \end{corollary}

\begin{proof}[Proof of Lemma \ref{le:GRP}.]
By the monotonicity of $L^p$-norms, we may assume without loss of generality that \eref{eq:alp-weak} is replaced by \eref{eq:alp}.
The bounds \eref{e:Xe1} and \eref{e:Xe} on $X_\eps$ can be proved as in Lemma \ref{le:PTE} for  the  special case of $\theta =1$. The only difference is that some integrals over $[0,t]$ have to be replaced by integrals over $(-\infty,t]$ and the zero Fourier mode has to be removed, but this does not change the bounds. The consistency relation \eref{eq:cond-it-int} and the symmetry condition were already observed above (see \eref{e:Symmetric} and above). Thus it only remains to show \eref{e:XXonly} and  \eref{e:XX}.

To apply Lemma \ref{lem:AnIsKol} to the $\XX_\eps$ we need to prove the following bounds: For every $\gamma <1$ we have
\begin{equs}
\E  \big| \XX_\eps(t;x,y) - \XX_\eps(s;x,y) \big|^p   &\ls   |t-s|^{  \frac{ \gamma p}{4}},\label{e:Xb1}\\
\E    \big| \XX_\eps(t;x,y)   \big|^p      &\ls |x-y|^{\gamma p} \label{e:Xb2}, \\
\E   \big| \delta \XX_\eps(t) \big|_{[x,y]}^p  &\ls |x-y|^{\gamma p}, \label{e:Xb3}\\
\E  \big| \XX_\eps(t;x,y) - \XX(t;x,y) \big|^p  &\ls \eps^{  \frac{ \gamma p }{2}   } \label{e:Xb4}.
\end{equs}

The bound \eref{e:Xb3} follows directly from the consistency relation \eref{eq:cond-it-int} as for all $x \leq y$
\begin{equs}[e:X1]
\E  \big| \delta \XX_\eps(t) \big|_{[x,y]}^p   &=  \E \Big(  \sup_{x \leq z_1 < z_2 < z_3 \leq y}\big|\delta X_\eps(t;z_1,z_2) \otimes  \delta X_\eps(t,z_2,z_3)   \big|\Big)^p \\
&\ls |x-y|^{\gamma} \,  \E  \big| X_\eps(t) \big|_{  \frac{\gamma}{2}}^{2p}\;.
\end{equs}
Lemma \ref{le:PTE} implies that the expectation on the right-hand side of \eref{e:X1} is finite, which shows \eref{e:Xb3}.

Due to the equivalence of all moments in the second Wiener chaos (see Lemma~\ref{lem:Nelson})  it is sufficient to prove the bounds in the special case $p=2$. To derive \eref{e:Xb2} we write
\begin{equ}
\XX_\eps(t;x,y)  =  \sum_{ k, l \in \Z^\star}    \xi_\eps^k (t) \otimes   \xi_\eps^l(t)  \,  e^{i(k+l)x} \,   q_{\eps}^k \,  q_\eps^l  \,   I_{k,l}(y-x)\; .
\end{equ}%
Recall the definitions \eref{e:Defq} for the  $q_k^\eps$, \eref{e:Defxi} for the  $\xi_\eps^k(t)$, and \eref{e:DefI} for the $I_{k,l}$. This sum is actually a slight abuse of notation because the sum may not converge absolutely in $L^2(\Omega)$. Hence as above in \eqref{e:Xsum} it should be interpreted as the limit as $N \to \infty$ of the symmetric approximations where the sum goes over all $k,l$ with absolute value bounded by $N$.

Then we can write
\begin{equs}
\E \big| \XX_\eps(t;x,y) \big|^2 &= \sum_{k,l, \bar k, \bar l \in \Z^\star} \E \big[ \tr    (\xi^k_\eps(t)  \otimes \xi^l_\eps(t))    (  \xi^{-\bar{l}}_\eps(t) \otimes \xi^{- \bar{k}}_\eps(t) )    \big] \,   q_{\eps}^k \,  q_\eps^l     \, q_{\eps}^{\bar{k}} \,  q_\eps^{\bar{l}}     \\
& \qquad \times  I_{k,l}(y-x) I_{-\bar{k},-\bar{l}}(y-x)\;.  
\end{equs}
For the first term in the sum we get
\begin{equs}
\E \big[ \tr \,     (\xi^k_\eps(t)  \otimes \xi^l_\eps(t) )    (  \xi^{-\bar{l}}_\eps(t)  \otimes \xi^{-\bar{k}}_\eps(t) )    \big]  &=    n^2 \,  \delta_{k , \bar{k}} \, \delta_{l , \bar{l}} \,  +\,  n \,  \delta_{k , \bar{l}} \, \delta_{l , \bar{k}}   + n \, \delta_{k , -l} \, \delta_{\bar k , - \bar{l}}\,\; ,
\end{equs}
and hence we can conclude that 
\begin{equs}[e:new1]
\E \big| \XX_\eps(t;x,y) \big|^2 & =  n^2 \sum_{k,l \in \Z^\star} \big( q_{\eps}^k \big)^2 \, \big(  q_\eps^l \big)^2  \, I_{k,l}(y-x) I_{-k,-l}(y-x) \, \\
&+ n \sum_{k,l \in \Z^\star} \big( q_{\eps}^k \big)^2 \, \big(  q_\eps^l \big)^2  I_{k,l}(y-x) I_{-l,-k}(y-x)\,  \\
&+n\Big(  \sum_{k \in \Z^\star} \big( q_{\eps}^k \big)^2 \,   I_{k,-k}(y-x)   \Big)^2.
\end{equs}

We treat the off-diagonal terms first. If $k  \neq -l$ the expression \eqref{e:DefI} shows that we have the bound
\begin{equ}
 |I_{k,l}(y)| \le  |k l| |y|^2 \wedge \Big( \Big| \frac{k- l}{k+l}\Big| +1 \Big) \ls   \Big(  \Big|\frac{k- l}{k+l}\Big| +1 \Big)^{ 1- \tilde{\gamma}}  |k l|^{\tilde \gamma} |y|^{2\tilde \gamma}    
\end{equ}
for every $\tilde{\gamma} \in [0,1]$. We plug that into the right hand side of \eqref{e:new1} and we get for the terms involving only $k \neq -l$
\begin{equs}
 &   n^2 \Big| \sum_{k \neq -l \in \Z^\star} \big( q_{\eps}^k \big)^2 \, \big(  q_\eps^l \big)^2  I_{k,l}(y-x) I_{-k,-l}(y-x) \, \Big|\\
&+ n \Big| \sum_{k\neq -l \in \Z^\star} \big( q_{\eps}^k \big)^2 \, \big(  q_\eps^l \big)^2  I_{k,l}(y-x) I_{-l,-k}(y-x)\,  \Big| \\
&\ls |y|^{4\tilde \gamma} \sum_{k\neq -l \in \Z^\star}  \frac{ |k l|^{ 2 \tilde \gamma} }{l^2 k^2}  \Big(  \Big|\frac{k-l}{k+l}\Big| +1 \Big)^{ 2- 2\tilde{\gamma}}  .
\end{equs}
The sums appearing in the last line are finite if $\tilde{\gamma} <  \frac{1}{2}$. In order to treat the diagonal terms for which $k=-l$ we recall that in this case
\begin{equ}[e:diago]
I_{k,-k}(y)= iky - \big(e^{iky} -1 \big).
\end{equ} 
Then the diagonal terms in the first two lines of \eqref{e:new1} can be treated directly with the brutal bound $|I_{k,-k}(y)| \ls |ky|$ which yields 
\begin{equs}
   n^2 \Big|& \sum_{k \in \Z^\star} \big( q_{\eps}^k \big)^4   I_{k,-k}(y-x) I_{-k,k}(y-x)  \Big| \\
 & + n  \Big| \sum_{k \in \Z^\star} \big( q_{\eps}^k \big)^4  I_{k,-k}(y-x)^2 \, \Big|  \ls \sum_{k \in \Z^\star} \frac{1}{k^2}  |y-x|^2 \ls |y-x|^2.
\end{equs}
In order to treat the last sum we need to make use of a cancellation. In fact, we can write for any $N$
\begin{equs}
   \sum_{0< |k| \leq N }& \big( q_{\eps}^k \big)^2 \,   I_{k,-k}(y-x)  \label{e:canc}\\
   &=  \sum_{0 < |k| \leq N} \big( q_{\eps}^k \big)^2 \,  ik (y-x)     +   \sum_{0< |k| \leq N} \big( q_{\eps}^k \big)^2 \, \big[e^{ik(y-x)} -1 \big]  .
\end{equs}
Now the summands in the first sum are antisymmetric with respect to changing the sign of $k$. Hence this sum vanishes identically for any $N$ and also in the limit. For the second term we use the easy bound $\big|e^{ik(y-x)} -1 \big| \ls |k(y-x)|^{\tilde \gamma} $ for any $\tilde \gamma \leq1$ to obtain 
\begin{equ}
\Big(  \sum_{k \in \Z^\star} \big( q_{\eps}^k \big)^2 \, \big(e^{ik(y-x)} -1 \big)   \Big)^2 \ls |y-x|^{2 \tilde \gamma} \Big(  \sum_{k \in \Z^\star} |k|^{\tilde \gamma-2}   \Big)^2.
\end{equ}
The sum converges for any $\tilde \gamma < 1$. This establishes \eref{e:Xb2}.

To derive the bound \eref{e:Xb1} on the time regularity of $\XX_\eps$ we write for $ t >0$
\begin{equs}
\XX_\eps(t;x,y) &- \XX_\eps(0;x,y) \, \label{e:Gau00}\\
&=  \sum_{k, l  \in \Z ^\star} \Big[  \xi^k_\eps(t) \otimes \xi^l_\eps(t) -  \xi^k_\eps(0) \otimes \xi^l_\eps(0) \Big] q_\eps^k \, q_\eps^l\,e^{i(k+l)x} I_{k,l} (y-x) . 
\end{equs}
We rewrite the term involving the Gaussian random variables as 
\begin{equs}[e:Gau1]
 \xi^k_\eps(t) \otimes & \,\xi^l_\eps(t) -  \xi^k_\eps(0) \otimes \xi^l_\eps(0)  =  \xi^k_\eps(t) \otimes \delta_{0,t} \xi^l_\eps +   \delta_{0,t} \xi^k_\eps    \otimes \xi^l_\eps(0) ,
\end{equs}
where $\delta_{0,t}\xi^l_\eps:= \xi^l_\eps(t) - \xi^l_\eps(0)$. Then using \eref{e:Defxi} we get  for the first term on the right hand side
\begin{equs}
 \,\E &  \tr  \big( \xi^k_\eps(t) \otimes  \delta_{0,t}\xi^l_\eps   \big)   \big(\delta_{0,t}\xi^{-\bar{l}}_\eps \otimes \xi^{-\bar{k}}_\eps(t) \big)     \\
  &= n^2 \delta_{k,\bar k} \delta_{l,\bar l}  \, \big( 2  - 2 \CK^{t}_l  \big) +  n  \delta_{k,\bar l} \delta_{l,\bar k} \big( 1 - \CK^{t}_k \big) \, \big( 1 - \CK^{t}_l \big)   \\  
  &\qquad +   n  \delta_{k,- l} \delta_{\bar k ,-\bar l} \big( 1- \CK^{t}_k \big) \, \big( 1- \CK^{t}_{\bar k} \big).  \label{e:Gau2}
\end{equs}
Then we observe that for any $k$ we have
\begin{equs}[e:Kss]
\big| 1- \CK^{t}_k  \big| \ls& 1 \wedge f(\eps k) k^2 (t).
\end{equs}
Plugging this into \eqref{e:Gau2} yields
\begin{equs}
 \big| \,\E   \tr & \big( \xi^k_\eps(t) \otimes  \delta_{0,t}\xi^l_\eps   \big)   \big(\delta_{0,t}\xi^{-\bar{l}}_\eps \otimes \xi^{-\bar{k}}_\eps(t) \big)   \big| \label{e:Gau3} \\
&\ls \big(   \delta_{k,\bar k} \delta_{l,\bar l}    +  \delta_{k,\bar l} \delta_{l,\bar k} \big) \big( 1 \wedge f(\eps l) l^2 t \big)  
+\delta_{k,- l} \delta_{\bar k,-\bar l} \,  \big( 1- \CK^{t}_k \big) \, \big( 1- \CK^{t}_{\bar k} \big).  \quad 
\end{equs}
%Here in the second term we have dropped the `good' factor $ \big( 1 \wedge f(\eps k) k^2 (t-s) \big)$ because we will not use it. 

Performing the same calculation for the second term on the right hand side of \eqref{e:Gau1} we obtain. 
\begin{equs}
 &\big| \,\E   \tr   \big( \delta_{0,t}\xi^k_\eps \otimes   \xi^l_\eps(0)   \big)   \big( \xi^{-\bar{l}}_\eps(0) \otimes \delta_{0,t}\xi^{-\bar{k}}_\eps(t) \big)   \big| \label{e:Gau33} \\
& \ls \big(   \delta_{k,\bar k} \delta_{l,\bar l}  + \delta_{k,\bar l} \delta_{l,\bar k}\big)   \big( 1 \wedge f(\eps k) k^2 (t-s) \big) 
+\delta_{k,- l} \delta_{\bar k,-\bar l} \, \big( 1- \CK^{t}_k \big) \, \big( 1- \CK^{t}_{\bar k} \big).\quad
\end{equs}

We treat the off-diagonal terms for which $k\neq -l$ first. In this case, as above, we use the simple bound $|I_{k,l}| \ls \big| (k-l) (k+l)^{-1} \big|+1 $. Summing over those terms we obtain
\begin{equs}
 \sum_{\substack{k \neq -l \in \Z ^\star\\ \bar k \neq -\bar l \in \Z ^\star}}& \E \tr \big(  \xi^k_\eps(t) \otimes \xi^l_\eps(t) -  \xi^k_\eps(0) \otimes \xi^l_\eps(0) \big) \big(  \xi^{-\bar l}_\eps(t) \otimes \xi^{-\bar k}_\eps(t) -  \xi^{-\bar l}_\eps(0) \otimes \xi^{-\bar k}_\eps(0) \big)  \\
 &\times q_\eps^k \, q_\eps^l q_\eps^{\bar k} \, q_\eps^{\bar l}  \,e^{i(k+l)x} e^{i(\bar k+\bar l)x} I_{k,l} (y-x) I_{\bar k,\bar l} (y-x)\\
 & \ls   \sum_{k,l \in \Z^{\star} }  \frac{1}{k^2} \frac{1}{l^2}   \Big( 1 \wedge  l^2 t \Big)\Big( \Big| \frac{k-l}{k+l}\Big| +1 \Big)^2  \ls |t|^{\frac{1}{2}} .
\end{equs}
For the diagonal terms $k=-l$ we make use of the same cancellation as in \eqref{e:canc}. More precisely, using \eqref{e:diago} we write
\begin{equs}
\sum_{k, \bar k  \in \Z ^\star}& \E \tr \big(  \xi^k_\eps(t) \otimes \xi^{-k}_\eps(t) -  \xi^k_\eps(0) \otimes \xi^{-k}_\eps(0) \big) \big(  \xi^{\bar k}_\eps(t) \otimes \xi^{-\bar k}_\eps(t) -  \xi^{\bar k}_\eps(0) \otimes \xi^{-\bar k}_\eps(0) \big)  \\
 & \times \big( q_\eps^k\big) ^2 \, \big(  q_\eps^{\bar k} \big)^2 \,   I_{k,-k} (y-x) I_{\bar k, -\bar k} (y-x)\\
 & \ls  \sum _{k, \in \Z^\star} \frac{1}{k^4}   \big( 1 \wedge  k^2 t \big) | ky| + \Big( \sum_{k \in \Z ^\star} \big( 1 - \CK^{t}_k \big) \, q_k^2I_{k,-k}\Big)^2. \label{e:Gau5}
\end{equs}
Here the first term on the right hand side corresponds to the first terms on the right hand side of \eqref{e:Gau3} and of \eqref{e:Gau33}. Here we have again made use of the estimate $I_{k,-k} \ls |k y|$. To treat the remaining two terms on the right hand side of \eqref{e:Gau5}  we observe in the same way as above in \eqref{e:canc} that the terms corresponding to $iky$ in \eqref{e:diago} cancel. Then bounding the other term by $1$ it is easy to see that the whole expression on the right hand side of \eqref{e:Gau5} is bounded by  $|t|^{\frac12}$ up to a constant. This establishes \eqref{e:Xb1}.

In order to derive the bound on the  $\eps$-differences we write for any $t,x,y$
\begin{equs}
\XX_\eps&(t;x,y) - \XX(t;x,y) \, \label{e:Gau01}\\
&=  \sum_{k, l  \in \Z ^\star} \big[  q^k_\eps \xi^k_\eps(t) \otimes \big( q^l_\eps  \xi^l_\eps(t)  -  q^l_0  \xi^l_0(t) \big)\\
&\qquad  +  \big( q^k_\eps  \xi^k_\eps(t)  -  q^k_0  \xi^k_0(t) \big) \otimes  q^l_0  \xi^l_0(t) \big] \,e^{i(k+l)x} \, I_{k,l} (y-x) .
\end{equs}
Here we have put the subscript $0$ to emphasise the limit case $\eps=0$. We set for any $\eps,\bar \eps \geq 0$
\begin{equ}
\CR^{k}_{\eps,\bar \eps} := \E  \, \big(  q^k_\eps \xi^k_{\bar \eps}(t)  q^{-k}_\eps \xi^{-k}_{\bar \eps}(t)  \big) = \frac{h(\eps k) h(\bar \eps k)}{k^2 \big(f(\eps k) + f(\bar \eps k) \big)},
\end{equ}
where we have made use of the definitions \eqref{e:Xsum} of $q^k_\eps$ and $\xi^k_\eps$. In particular, using the assumptions~\ref{a:f} and~\ref{a:h} on $f$ and $h$   we get for any $k$ and any $\eps,\bar \eps \geq 0$ that
\begin{equs}
\big| \CR^{k}_{\eps, \eps}\big|  &\ls |k|^{-2}, \\
\big| \CR^{k}_{\eps, \eps}  - \CR^k_{\eps,\bar \eps} \big|   &\ls  |k|^{-2} \big( |\eps-\bar \eps| |k| \wedge 1\big), \\
\big| \CR^{k}_{\eps, \eps}  + \CR^k_{\bar \eps, \bar \eps}  -2 \CR^k_{\eps,\bar \eps}  \big| &\ls  |k|^{-2} \big( |\eps-\bar \eps| |k| \wedge 1\big) .
\end{equs}
Then a calculations which is very similar to the proof for the time regularity shows that one gets
\begin{equs}
 \E  \big|&\XX_\eps(t;x,y) - \XX(t;x,y) \big|^2 \\
& \ls  \sum_{k,l \in \Z^{\star} }  \frac{1}{k^2} \frac{1}{l^2}   \Big( 1 \wedge | l \eps|  \Big)\Big( \Big| \frac{k-l}{k+l}\Big| +1 \Big)^2 + \sum_{k \in \Z^\star} \frac{1}{k^4}  \eps |k|  |ky| \\
& \qquad + \Big( \sum_{k \in \Z^\star} \big( \CR^{k}_{\eps, \eps}  - \CR^k_{\eps,0} \big) I_{k,-k}\Big)^2  + \Big( \sum_{k \in \Z^\star} \big( \CR^{k}_{\eps, 0}  - \CR^k_{0,0} \big) I_{k,-k}\Big)^2 \ls \eps.
\end{equs}
This finishes the proof for \eqref{e:Xb4}.
\end{proof}

In the proof of Corollary \ref{cor:Rte} we will need another bound on $X_\eps$. For the usual heat semigroup $S(t)$, a bound on the H\"older norm $|X|_{\cC^\alpha}$ is enough to conclude that the map $t\mapsto S(t) X$ is $\frac{\alpha}{2}$-H\"older continuous taking values in $\cC^0$. Unfortunately, as discussed for example in Lemma  \ref{lem:timecontS}, the same statement is not true for the approximate heat semigroup $S_\eps$. We thus define a norm that ensures this property. For $X \in \cC^0$ we set
\begin{equ}[e:DNorm]
|X|_{\dD^{\alpha,\eps}}:= \sup_{0 \leq r_1<r_2 \leq 1} |r_2-r_1|^{-\frac{\alpha}{2}}|S_\eps(r_2)X -S_\eps(r_1)X|_{\cC^0}
 \in [0,+\infty].
\end{equ}
Note that $r_1 = 0$ is allowed inside the supremum. 
For any $k \in \Z_\eps = \{k \in \Z^\star \,:\, f(\eps k) < \infty\}$, we then have
\begin{equ}
| e^{ikx} |_{\dD^{\alpha,\eps}} \eqsim |k|^{\alpha}f(\eps k)^{\frac{\alpha}{2}}.
\end{equ}
We also denote by $\dD^{\alpha,\eps}$ the closure of the vector space generated by $\{e^{ikx}\}_{k \in \Z_\eps}$ under this norm.  Note that this space is finite dimensional (with dimension depending on $\eps$) if $f$ is equal to $\infty$ outside a compact domain.

We now establish bounds for $X_\eps$ measured with respect to $| \cdot |_{\dD^{\alpha,\eps}}$. 

\begin{lemma}\label{le:strangenorm}
For any $\eps \geq 0$ and for any $\alpha<\frac12$ the process $(X_\eps(t))_{0 \leq t \leq T}$ is continuous with values in $\dD^{\alpha,\eps}$,
and for any $p<\infty$ we have that 
\begin{equ}
\E \Big( \sup_{0\leq t \leq T} | X_\eps(t) |_{\dD^{\alpha,\eps}}^p \Big) \ls 1\;,
\end{equ}
uniformly over $\eps \in (0,1]$.
\end{lemma}
\begin{proof}
We fix $0 \leq s_1 < s_2 \leq T  $ and an $0<r_1 <r_2 \leq1$. Then, using H\"older's inequality we get for any $0<\lambda<1$ and any $p >0$ that 
\begin{equs}
&\E\Big|  \bigl(S_\eps(r_1) - S_\eps(r_2)\bigr) \big( X_\eps(s_2) -X_\eps(s_1) \big) \Big|_{\cC^0}^p 
 \\ &\qquad \ls \Big[ \max_{t \in \{ s_1,s_2\}} \E\big| 
   ( S_\eps(r_1)   - S_\eps(r_2)) X_\eps(t)  \big|_{\cC^0}^p\Big]^\lambda   \\
 & \qquad   \qquad  \times\bigg[ \max_{r^{\prime} \in \{r_1,r_2 \}}\E\Big|  S_\eps(r^{\prime}) \big( X_\eps(s_2) -X_\eps(s_1) \big) \Big|_{\cC^0}^p \bigg]^{1-\lambda}.
\end{equs}
The second factor can be bounded easily using the regularity of $X_\eps$ (see \eqref{e:Xe1}) and the fact  that $|S_\eps(r)|_{\cC^\kappa \to \cC^0} \ls 1$ for any $\kappa>0$ by Lemma \ref{le:newOp}. In this way we obtain for any $\kappa >0$ that
\begin{equ}
 \max_{r^{\prime} \in \{r_1,r_2 \}}\E\Big|  S_\eps(r^{\prime}) \big( X_\eps(s_2) -X_\eps(s_1) \big) \Big|_{\cC^0}^p \ls |s_2-s_1|^{\frac{p}{4}-\kappa}.
\end{equ}
To treat the first term, for any $x \in [-\pi,\pi]$ a straightforward calculation yields
\begin{equs}
\E\big|  &\big(S_\eps(r_1)  - S_\eps(r_2) \big) X_\eps(t)(x)  \big|^2  \\
&\ls \sum_{k \in \Z_\eps} \int_{-\infty}^t \Big( e^{-k^2f(\eps k) (t+r_1-s)   } - e^{-k^2f(\eps k) (t+r_2-s)   } \Big)^2 ds \\
&\ls \sum_{k \in \Z_\eps} \big( 1 \wedge f(\eps k ) k^2 | r_1 -r_2| \big)^2 \frac{1}{k^2 f(\eps k)} \ls |r_1-r_2|^{\frac12},
\end{equs}
uniformly over $\eps \in (0,1]$.
Then an easy argument yields 
\begin{equ}
\E\big|  \delta \big( (S_\eps(r_1)   - S_\eps(r_2)) X_\eps(t)\big)(x_1,x_2)  \big|^2 \ls  |r_1-r_2|^{\frac12} \wedge|x_1-x_2|,
\end{equ}
where $\delta X(x_1,x_2) = X(x_2) - X(x_1)$. 
Hence,  Kolmogorov's criterion (see e.g. Lemma \ref{lem:Kolmogorov} applied to the Banach space $\R$) together with the equivalence of Gaussian moments  (Lemma~\ref{lem:Nelson}) yields that for any $\kappa>0$ we have 
\begin{equ}
 \E\big|\big(  S_\eps(r_1) - S_\eps(r_2)\big) X_\eps(t)  \big|_{\cC^0}^p \ls |r_1-r_2|^{\frac{p}{4} - \kappa}.
\end{equ}
Hence, by choosing an appropriate $\lambda$ (close to $1$) and small values of $\kappa$ we obtain for any $\alpha <\alpha^\prime<\frac12$ and for a small $\kappa^\prime>0$ that 
\begin{equ}
\E\Big| \big( S_\eps(r_1) - S_\eps(r_2) \big)\big( X_\eps(s_2) -X_\eps(s_1) \big) \Big|_{\cC^0}^p \ls |r_1-r_2|^{p\frac{\alpha^{\prime}}{2}} |s_2-s_1|^{p \kappa^\prime}, 
\end{equ}
so that applying Kolmogorov's Criterion once more (in $r$) yields that
\begin{equ}
\E \big| X_\eps(s_2) - X_\eps(s_1) \big|_{\dD^{\alpha,\eps}}^p \ls |s_2-s_1|^{\kappa^{\prime}}.
\end{equ}
Then using the fact that $X_\eps(0)=0$ and applying Kolmogorov's criterion \ref{lem:Kolmogorov}  once more we obtain the desired bound.
\end{proof}

Finally, we give the necessary bounds on the remainder term $\Rte$. The derivation of the uniform bounds requires more work than in the cases of $\Pte$ or $(X_\eps,\XX_\eps)$.  As in \cite{HW10} the regularity of $\Pte$ follows from the space-time regularity of $\theta$.  Actually, formally one obtains 
\begin{equs}
\Rte(t;x_1,x_2) =&  \!\int_0^t \!\! \!\int_{-\pi}^\pi\big(  p_{t-s}^\eps(x_2 -z) - p_{t-s}^\eps(x_1 -z) \big)\big(\theta(s,z) - \theta(t,x_1) \big)  \, dW(s,z)\\
& + \delta_{x_1,x_2} S_\eps(t)  X_\eps(0) \label{e:wrong}
\end{equs}
and a formal application of the Burkholder-Davis-Gundy inequality suggests that the space-time regularity of $\theta$ can be used to deduce higher spatial regularity for $\Rte$. Unfortunately, this reasoning is not correct because the process $(\theta(s,z) - \theta(t,x_1))$ is not adapted to the filtration generated by $W$ since $\theta(t,x_1)$ lies ``in the future". In particular, the Burkholder-Davis-Gundy inequality cannot be applied directly. This problem is overcome in Lemma~\ref{le:mistakecorrected} with a bootstrap argument.

The lack of time regularity of the process $u$ near zero also causes a slight inconvenience. Recall that in the application we have in mind we have $\theta = \theta(\ue)$.  We have seen in Subsections \ref{ss:summary} and \ref{ss:proof} that in general the process $u_\eps$ need not be time continuous near $0$. We only have the necessary ``almost $1\over4$" H\"older regularity for times $t \geq \eps^2$. Hence we will only obtain the $\frac{\alpha}{2}$-regularity for  times $t \geq \eps$ with a blowup for $t \downarrow \eps^2$.

Recall  the definition of the parabolic H\"older norms  $\|  \cdot \|_{\cC^{\alpha/2, \alpha}_{[s,t]}}$   and  $\| \cdot \|_{\cC^{\alpha/2, \alpha}_t }$ in \eref{e:Norm17}  and \eref{e:alphaHol1}. Then for  any $K>0$ and for a $0 < \aX< \frac12$ we introduce the stopping time 
\begin{equ}[e:DefTau]
\rho^{X,\aX}_{\eps, K} \, = \, \inf \Big\{  t \geq 0\colon   \| X  \|_{\cC^{\aX/2, \aX}_t}  \geq K \text{ or }   \| X_\eps -  X  \|_{\cC^{\aX/2, \aX}_t}    \geq 1  \Big\}.
\end{equ}
Observe (recalling the definitions \eref{sigmaX} and \eref{sigmaXe}) that $\rho^{X,\aX}_{\eps,K} \geq \sigma^X_K \wedge \rho^X_\eps$ almost surely. 

We start by showing that $R^\theta_{\eps  } $ has the required spatial regularity, uniformly in $\eps$.
\begin{lemma}\label{le:mistakecorrected}
Suppose that  $\alpha,\aX \in (0,\frac12)$.  Let $\tau$ be a stopping time that almost surely satisfies 
\begin{equs}[e:COndST1]
 0 \leq \tau  \leq \rho^{X,\aX}_{\eps, K} \wedge T.
\end{equs}
For every $0 \leq t \leq T$ we set 
\begin{equs}[e:DefTtPR]
\tilde{\theta}(t) &:=\theta(t \wedge \tau),\\
\tilde{\Psi}_\eps^\theta(t) &:= \int_0^{t\wedge\tau}  S_\eps(t-r) \, \tilde{\theta}(r) \, H_\eps dW(r),\\
\tilde{X}_\eps(t) &:=\mathbf{1}_{ \{\tau >0\}} \int_{-\infty}^{t \wedge\tau}  S_\eps(t-r) \, \Pi H_\eps dW(r),\\
\tRte(t;x,y) &:= \delta \tilde{\Psi}_\eps^\theta(t;x,y) - \tilde{\theta}(t ,x) \, \delta \tilde{X}_\eps (t;x,y).
\end{equs} 
Then for any $p$ large enough, and for any
\begin{equ}[e:recongam]
 \gamma <  \aX +\alpha -  \frac{1}{p}   - \sqrt{\frac{1}{2p} (1 + \alpha - \aX) } \;, 
\end{equ}
the following bound holds true:
\begin{equ}[e:Rt0]
\sup_{\eps^2 < t \leq T}  (t-\eps^2)^{\alpha p/2} \, \E  \big| \tRte(t)  |_{\bB^\gamma}^p   \,  \ls \,  \,    \E \, \| \theta \|_{\cC_{[\eps^2,\tau]}^{\alpha/2,\alpha}}^p + \E \, \| \theta \|_{\cN_{\tau}}^p  .
\end{equ}
\end{lemma}
 \begin{proof}
If $\tau =0$ the processes $\tilde{\Psi}_\eps^\theta(t)$, $\tilde{X}_\eps(t)$, and $\tRte(t;x,y)$ are zero for all $t \geq 0$. Else, for times $0 \leq t \leq \tau$ the processes $\tilde{\theta}, \tilde{\Psi}_\eps^\theta, \tilde{X}_\eps$, and  $\tRte$ coincide with $\theta, \Pte, X_\eps$, and  $\Rte$. 
For $t >\tau$ the processes $\tilde{\Psi}^\theta_\eps$ and $\tilde{X}_\eps$ satisfy the identities $\tilde{\Psi}_\eps^\theta(t) = S_\eps(t-\tau)\Pte(\tau)$ and $\tilde{X}_\eps(t) = S_\eps(t-\tau) X_\eps(\tau)$. Recalling the regularity properties of the approximated heat semigroup in Corollary \ref{cor:S-Seps}, we obtain  for any $\kappa>0$,
\begin{equ}[e:tXtP]
\sup_{0 \leq s \leq T}| \tilde{\Psi}_\eps^\theta(s)|_{\cC^{\aX-\kappa}} \ls  \| \Psi_\eps^\theta\|_{\cC^{\aX}_\tau}  \quad \text{and} \quad  \sup_{0 \leq s \leq T}| \tilde{X}_\eps(s)|_{\cC^{\aX-\kappa}} \ls  \| X_\eps\|_{\cC^{\aX}_\tau}\;.
\end{equ}
We will fix such a small $\kappa$ for the rest of the proof.

In particular, recalling the condition \eqref{e:COndST1} on the stopping time $\tau$ we have the almost sure estimate
\begin{equ}[e:StcondX]
\sup_{0\leq s \leq T} | \tilde{X}_\eps(s) |_{\cC^{\aX-\kappa}} \ls (K +1)\ls 1\;.
\end{equ}

After these preliminary considerations we are now ready to start the derivation of the estimate \eqref{e:Rt0}.  We first observe that the definition of $\tRte$ as well as the regularity results for $\Pte$ (Lemma \ref{le:PTE}) together with \eqref{e:tXtP}  immediately imply that for any $p$ satisfying 
  \begin{equ}
p > \frac{6}{1- 2 \aX }\;,
  \end{equ}
we have 
 \begin{equs}
 \E\bigg(   \sup_{0 \leq s \leq T} \big| \tRte(s)  \big|_{\bB^{\aX-\kappa}}\bigg)^p &\ls  \E   \big\| \Pte  \big\|_{\cC^{\aX}_\tau}^p + \E\Big( \|\theta\|_{\cN_{\tau}}\,  \big\|  {X}_\eps \big\|_{\CC^{\aX }_{\tau}} \Big)^p \\
& \ls \E   \big\| \theta \|_{\cN_{\tau}}^p\;,  \label{e:corrAA}
 \end{equs}
where in the second line we have made use of the deterministic a priori bound   \eqref{e:StcondX}.
 
The idea is to use this (very weak) a priori information on the regularity of $\tRte$ as the starting point for a bootstrap argument. 
For any $\eps^2 < s<t \leq T$  and for $x_1,x_2 \in [-\pi,\pi]$  we define the following three quantities
\begin{equs}
R_1(s,t;x_1,x_2) &:= \big( \tilde{\theta}(t,x_1)- \tilde{\theta}(s ,x_1)   \big) \delta \tilde{X}_\eps(t;x_1,x_2),\\
R_2(s,t;x_1,x_2)&:=    \delta\bigg[ \int_{s\wedge{\tau}}^{t\wedge\tau}  S_\eps(t-r) \, \tilde{\theta}(r)\, H_\eps \,dW(r)\bigg](x_1,x_2)\\
&\qquad  - \tilde{\theta}(s,x_1) \,  \delta\bigg[ \int_{s\wedge \tau}^{t\wedge\tau}  S_\eps(t-r) \,  H_\eps \, dW(r)\bigg](x_1,x_2) , \\
R_3 (s,t;x_1,x_2)&:= \delta \bigg[ \int_{0}^{s\wedge \tau} S_\eps(t-r) \, \tilde{\theta}(r) \,  H_\eps \,dW(r)\bigg](x_1,x_2)\\
& \qquad  -  \tilde{\theta}(s,x_1) \delta \bigg[ \int_{-\infty}^{s\wedge\tau} S_\eps(t-r) \,   H_\eps \,dW(r)\bigg](x_1,x_2). 
\end{equs}
Note that for all $s,t,x_1,x_2$ we have the identity
\begin{equ}
\tRte(t;x_1,x_2) = -R_1(s,t;x_1,x_2)  + R_2(s,t;x_1,x_2)  +R_3(s,t;x_1,x_2). 
\end{equ}
We will now bound the $R_i$ individually.
 
A bound on $R_1$ can be established easily. We get almost surely 
\begin{equ}[e:corrA]
\big| R_1 (s,t;x_1,x_2) \big|   \ls  \|\theta\|_{\cC_{[s,\tau]}^{\alpha/2,\alpha}} \,   |x_1 -x_2|^{\aX-\kappa } |t-s|^{\frac{\alpha }{2}},
\end{equ}
where we have again made use of the deterministic  bound  \eqref{e:StcondX} on $\big\|\tilde{X}\|_{\cC^{\aX-\kappa}_T}$. 

To bound $R_2$, we rewrite it as
\begin{equs}
R_2(s,t;x_1,x_2) &=    \int_{s\wedge{\tau}}^{t\wedge\tau} \int_{-\pi}^\pi \big(  p^\eps_{t-r}(x_2 -y) - p^\eps_{t-r}(x_1 -y) \big) \\
& \qquad \qquad  \qquad  \big( \tilde{\theta}(r,y) -\tilde{\theta}(s,x_1) \big) \,H_\eps \,  dW(r,y).
\end{equs}
This time the integrand is adapted, so we can apply the Burkholder-Davis-Gundy inequality. 
Combining this with the fact that $H_\eps$ is a contraction on $L^2$implies for $p>1$
\begin{equs}[e:corrB]
 \E & |R_2(s,t;x_1,x_2)|^p  \\
& \ls \E \Big( \int_{s\wedge \tau}^{t \wedge \tau}  \!\! \int_{-\pi}^\pi \big(  p_{t-r}^\eps (x_1-y) - p_{t-r}^\eps (x_2-y) \big)^2  \big( \tilde{\theta}(r,y) - \tilde{\theta}(s,x_1)   \big)^{2} \,  dy \, dr \Big)^{  \frac{p}{2}} \\
& \ls \E \| \theta \|_{\cC_{[s,\tau]}^{\alpha/2,\alpha}}^p \, \bigg(  \int_{s}^t  \!\! \int_{-\pi}^\pi  \! \big(  p_{t-r}^\eps (x_1-y) - p_{t-r}^\eps (x_2-y) \big)^2 \\
& \qquad\qquad  \qquad \qquad  \qquad \qquad  \times \big(  |s-r|^{\alpha} +|y-x_1|^{2 \alpha}  \big) \,  dy \, dr \, \bigg)^{  \frac{p}{2}} \\
 & \ls   \E \| \theta \|_{\cC_{[s,\tau]}^{\alpha/2,\alpha}}^p \, |x_1 - x_2|^{\frac{p}{2} } |t-s|^{\frac{\alpha p}{2}} .  
\end{equs}
Here we used the trivial bound $|s-r|\leq |t-s|$ for $r \in [s,t]$ and the bound \eqref{e:uselat} for the term involving $|s-r|$. The calculations for the term involving $|y-x_1|$ can be found in  Lemma~\ref{le:pepscalc} below.

For the third term we write \begin{equs}
R_3 (s,t;x_1,x_2)=& \int_{-\pi}^\pi \big(  p_{t-s}^\eps(x_2 - z)  - p_{t-s}^\eps(x_1 - z)    \big) \tilde{\Psi}^\theta_\eps(s,z) \, dz \label{e:corr1}\\
&- \tilde{\theta} (s, x_1) \int_{-\pi}^\pi \big(  p_{t-s}^\eps(x_2 - z)  - p_{t-s}^\eps(x_1 - z)    \big) \tilde{X}_\eps(s,z) \, dz . 
\end{equs}
We rewrite the first integral as
\begin{equs}[e:corr2]
  \int_{-\pi}^\pi& \big(  p_{t-s}^\eps(x_2 - z)  - p_{t-s}^\eps(x_1 - z)    \big) \tilde{\Psi}^\theta_\eps(s,z) \, dz \\
  &= \int_{-\pi}^\pi \Big( \int_{x_1}^{x_2}  (p_{t-s}^\eps)'(\lambda - z) \, d\lambda      \Big) \tilde{\Psi}^\theta_\eps(s,z) \, dz \\ 
  &= \int_{-\pi}^\pi  \int_{x_1}^{x_2}  (p_{t-s}^\eps)'(\lambda - z) \,\Big(  \tilde{\Psi}^\theta_\eps(s,z) - \tilde{\Psi}^\theta_\eps(s,x_1) \Big)  \, d\lambda \,dz\;.  
  \end{equs}
  Here in the last line we have made use of the identity
  \begin{equ}
  \int_{-\pi}^\pi   (p_{t-s}^\eps)'(\lambda - z) \, \tilde{\Psi}^\theta_\eps(s,x_1)   \, \,dz = 0\;,
  \end{equ}
 which holds for every $\lambda$ due to the periodicity of $p_t^\eps$.

 We can rewrite the second integral in \eqref{e:corr1} in the same way. Hence, inserting this back into \eqref{e:corr1} we obtain
 \begin{equ}
R_3(s,t;x_1,x_2) =  - \int_{-\pi}^\pi  \int_{x_1}^{x_2}  (p_{t-s}^\eps)'(\lambda - z) \,\tRte(s; x_1,z) \, d\lambda \,dz\;.
 \end{equ}
 In particular, we can conclude that as soon as $\big| \tRte(s) \big|_{\bB^\gamma}$ is finite for some $0 < \gamma < 1$,  we have
 \begin{equ}[e:corr3]
 \big|R_3 (s,t;x_1,x_2) \big|  \,
 \leq \, \big| \tRte(s) \big|_{\bB^\gamma}    \int_{-\pi}^\pi  \int_{x_1}^{x_2}  \big| (p_{t-s}^\eps)'(\lambda - z) \big|   \, |x_1 -z|^\gamma  d\lambda \,dz\;.
 \end{equ}
 For the integral in the last line of \eqref{e:corr3} we get
\begin{equs}[e:corr4]
\int_{-\pi}^\pi&  \int_{x_1}^{x_2}  \big| (p_{t-s}^\eps)'(\lambda - z) \big|  \, |x_1 -z|^\gamma  d\lambda \,dz \\
&\ls \int_{-\pi}^\pi  \int_{x_1}^{x_2}  \big| (p_{t-s}^\eps)'(\lambda - z) \big|  \,\big(  |x_1 -\lambda|^\gamma +  |\lambda -z|^\gamma \big) \, d\lambda \,dz \\
&\ls |x_1 -x_2|^{1+\gamma} \int_{-\pi}^\pi   \big| (p_{t-s}^\eps)'( z) \big| \, dz + |x_1 -x_2| \int_{-\pi}^\pi   \big| (p_{t-s}^\eps)'( z) \, \big|  \,  |z|^\gamma\, dz .
\end{equs}
Plugging the bounds obtained in Lemma \ref{le:pepscalc} into the right-hand side of \eqref{e:corr4}, we obtain 
\begin{equs}
 \big|R_3 (s,t;x_1,x_2) \big|  &\leq \big| \tRte(s) \big|_{\bB^\gamma}   \Big(|x_1 -x_2|^{1+\gamma} (t-s)^{-\frac12} \big|\log(t-s) \big|\\
 &  \qquad  \qquad \qquad   \qquad \qquad \qquad+ |x_1 -x_2|(t-s)^{\frac{-1+\gamma}{2}} \Big) .
\end{equs}
Now for fixed $x_1,x_2$ and $t$ and $s>0$ we can summarise the above calculations as follows. We have
\begin{equs}[e:bs]
\Big( \E \big|&\tRte(t;x_1,x_2) \big|^p \Big)^{\frac1p}  \ls  \,\Big( \E \Big| R_1(s,t;x_1,x_2) \big|^p \Big)^{\frac1p} + \Big(  \E \big| R_2(s,t;x_1,x_2) \big|^p \Big)^{\frac1p} \\
& \qquad +\Big(  \E\big| R_3(s,t;x_1,x_2) \big|^p\Big)^{\frac1p}\\
&\ls \Big(  \E \big\| \theta \big\|_{\cC_{[s,\tau]}^{\alpha/2,\alpha}}^p \Big)^{\frac1p} |x_1 - x_2|^{\aX - \kappa} \, | t-s|^{\frac\alpha2} + \Big( \E \big| \tRte(s) \big|_{\bB^\gamma}^p \Big)^{\frac1p}   \\
& \quad \times \Big(|x_1 -x_2|^{1+\gamma} (t-s)^{-\frac12} \big|\log(t-s) \big| + |x_1 -x_2|(t-s)^{\frac{-1+\gamma}{2}} \Big).
\end{equs}
Recall that the bound \eqref{e:corrAA} implies that all moments of 
the norm $| \tRte(s)|_{\bB^\gamma}$ are finite if we choose $\gamma =\aX-\kappa$.  
Now we shall use this weak a priori knowledge on the regularity of $\tRte$ as a starting point for a bootstrap argument based on the estimate \eqref{e:bs}. Let us briefly outline the argument before going into details. 

We fix a time $t$ and $x_1,x_2$ and choose the time difference $|t-s| =\Delta$ as $|x_1-x_2|^\nu$, where $\nu= \nu(\gamma,\alpha,\aX)$ is chosen such that the first and the last term  on the right hand side of \eqref{e:bs} scale in the same way. Plugging this into \eqref{e:bs} yields an improved bound for $\tRte$, which can in turn be used as the right hand side of \eqref{e:bs}, etc. 

There are two obstacles that have to be overcome. On the one hand,  \eqref{e:bs} is only a good estimate if we have control over the space-time regularity of $\theta$ on $[s,t]$. As this is only the case for $s > \eps^2$ we treat the case where $\Delta$ is too large in a brutal way causing a blowup like $(t-\eps^2)^{-\alpha/2}$. On the other hand, \eqref{e:bs} yields an  estimate on expectations of $\tRte(t,x_1,x_2)$ for fixed values of $t,x_1$ and $x_2$. But in order to plug that back into the right hand side of \eqref{e:bs} we need to turn this into an estimate on the expectation of a spatial supremum. We move the supremum under the expectation by an application of Gubinelli's version of the Garsia-Rodemich-Rumsey lemma, Lemma~\ref{lem:GRR}.

We begin the iteration by setting $\gamma_0 := \aX-\kappa$ as in  \eqref{e:corrAA}. Then, for fixed $t$, $x_1$ and $x_2$ we set  
\begin{equs}
\nu_0 = 2\; \frac{1-\gamma_0}{1 - \gamma_0+\alpha} \quad \text{and} 	\qquad \Delta_0  =|x_1 -x_2|^{\nu_0}\;.
\end{equs}
If $t-\eps^2 > 2\Delta_0$, we are `far enough' from time $\eps^2$ and we apply \eqref{e:bs} for $s = t-\Delta_0$. This yields
\begin{equs}[e:imp1]
\Big(& \E \big|\tRte(t;x_1,x_2) \big|^p \Big)^{\frac1p} 
 \ls  
\Big( \E \big\| \theta \big\|_{\cC_{[\eps^2,\tau]}^{\alpha/2,\alpha}}^p \Big)^{\frac1p}
 |x_1 - x_2|^{\tilde{\gamma}_1} 
  + \Big( \E \big| \tRte(s) \big|_{\bB^{\gamma_0}}^p \Big)^{\frac1p}   \\
& \qquad \times \Big(|x_1 -x_2|^{1+\gamma_0}\Delta_0^{-\frac12} \big|\log \Delta_0 \big| + |x_1 -x_2|(\Delta_0)^{\frac{-1+\gamma_0}{2}} \Big)\\
&\ls \bigg[\Big( \E   \big\| \theta \big\|_{\cC_{[\eps^2,\tau]}^{\alpha/2,\alpha}}^p \Big)^{\frac1p}  + \Big( \E \big| \tRte(s) \big|_{\bB^{\gamma_0}}^p \Big)^{\frac1p}\bigg]   |x_1 - x_2|^{\tilde{\gamma}_1} \\
&\ls \Big( \E   \big\| \theta \big\|_{\cC_{[\eps^2,\tau]}^{\alpha/2,\alpha}}^p \Big)^{\frac1p}  |x_1 - x_2|^{\tilde{\gamma}_1}\;,
\end{equs}
where 
\begin{equ}
 \tilde{\gamma}_1:= \frac{\gamma_0(1 - \gamma_0) + \alpha }{1 - \gamma_0 + \alpha}\;.
\end{equ}
In the second inequality we have used the identity $1 - \frac{\nu_0}{2}(1-\gamma_0) = \tilde\gamma_1$, as well as the inequality $1 - \frac{\nu_0}{2} +\gamma_0 > \tilde\gamma_1$, which holds since $\nu_0<2$. The latter inequality implies that
\begin{equ}
|x_1 -x_2|^{1+\gamma_0}\Delta_0^{-\frac12} \big|\log \Delta_0  \big| \ls  |x_1 -x_2|(\Delta_0)^{\frac{-1+\gamma_0}{2}}\;.
\end{equ}
In the  last line of \eqref{e:imp1} we have used the a priori information \eqref{e:corrAA} on the regularity of $\tRte$.

Otherwise, if $t-\eps^2 \leq 2\Delta_0 = 2 |x_1 -x_2|^{\nu_0}$ we use the a priori knowledge \eqref{e:corrAA} to obtain
\begin{equs}[e:appp]
\Big( \E \big|\tRte(t;x_1,x_2) \big|^p \Big)^{\frac1p}  \ls& \, \Big( \E   \| \theta \|_{\cN_{\tau}}^p \Big)^{\frac{1}{p}}|x_1 -x_2|^{\aX-\kappa} \\
\ls& \,  \Big( \E \| \theta \|_{\cN_{\tau}}^p\Big)^{\frac{1}{p}} |x_1 -x_2|^{\tilde{\gamma}_1} (t-\eps^2)^{-\alpha/2}\;. 
\end{equs}
%where 
%\begin{equs}[e:beta00]
%\beta := \frac{1}{\nu_0} \big( \tilde{\gamma}_1-(\aX - \kappa) \big) =\frac{\alpha}{2}.
%\end{equs}

As mentioned above, in order  to plug the estimate \eqref{e:imp1} back into \eqref{e:bs} we still have to replace the bound on the supremum over expectations by a bound on the expectation of the supremum. Therefore, in order to apply Lemma \ref{lem:GRR} we still need extra information on the behaviour of the spatial $\delta$ operator (defined in Appendix \ref{AppB}) applied to $\tRte$. Note that
\begin{equs}
\delta \tRte(t;z_1,z_2,z_3) :=& \, \tRte(t;z_1,z_3) - \tRte(t;z_1,z_2) - \tRte(t;z_2,z_3)\\
=& \,\delta \tilde{\theta}(t; z_1,z_2) \, \delta \tilde{X}_\eps(t;z_2,z_3)\;. 
\end{equs}
Recalling the deterministic a priori bound \eqref{e:StcondX} on the regularity of $\tilde{X}_\eps$ (and the definition of $| \cdot |_{[x_1,x_2]}$ in Appendix \ref{AppB}), this identity implies that
\begin{equ}[e:Gubingr]
\bigg( \E  \sup_{x_1,x_2}\bigg|
 \frac{|\delta \tRte(t)|_{[x_1,x_2]}}{|x_1 -x_2|^{\alpha + \gamma_0}}
\bigg|^p \bigg)^{\frac1p} \ls
  \Big( \E   \big\| \theta \big\|_{\cC_\tau^{\alpha/2,\alpha}}^p \Big)^{\frac1p}\;.
\end{equ}

Recall that, as mentioned above, $\tilde{\gamma}_1< \alpha + \gamma_0$.
Hence, combining \eqref{e:imp1} with \eqref{e:appp} and \eqref{e:Gubingr}, we finally obtain from  Lemma \ref{lem:GRR} that
\begin{equ}[e:Horray]
\sup_{\eps^2 \leq t \leq T} (t-\eps^2)^{\alpha/2}  \Big( \E \big\|\tRte(t) \big\|_{\bB_\tau^{\gamma_1}}^p \big)^{\frac1p}  \ls  \Big( \E   \big\| \theta \big\|_{\cC_{[\eps^2,\tau]}^{\alpha/2,\alpha}}^p + \E  \big\| \theta \big\|_{\cN_\tau}^p\Big)^{\frac1p}\;,
\end{equ}
%for any $p$ and $\gamma_1$ satisfying $\gamma_1 + \frac{1}{p}< \tilde{\gamma}_1$. 
where $\gamma_1 := \tilde{\gamma}_1 - \frac{1}{p} - \kappa$. 
Note that this definition of $\gamma_1$ guarantees the finiteness of the integral in Lemma \ref{lem:GRR}.
%\begin{equs}[e:aaa]
%\gamma_1 <& \tilde{\gamma}_1 - \frac{1}{p}
%=\frac{(\aX - \kappa) - \gamma_0( \aX - \kappa) +\alpha }{\alpha+1 - \gamma_0}  - \frac{1}{p}. 
%\end{equs}

Let us observe that our calculations have led to improved regularity bounds.  Indeed, we will apply \eqref{e:Horray} in the cases where $\kappa$ is very small, $\aX$ very close to $\frac12$, $p$ very large, and $\alpha$ either close to $\frac12$ or close to $\frac13$. In both cases the a priori information \eref{e:corrAA} yields a regularity exponent $\gamma_0= \aX-\kappa$, which is very close to $\frac12$. On the other hand, the regularity exponent $\gamma_1$ in the new bound \eqref{e:Horray} is  close to $\frac{7}{10}$ if $\alpha \approx\frac13$, and close to $\frac34$ if  $\alpha \approx\frac12$.

Now we iterate this argument. For $n \geq 1$ we define $\nu_n$, $\Delta_n$ and $\gamma_n$ recursively as
\begin{align*}
\nu_n:= {2\;\frac{1-\gamma_0}{1 - \gamma_n + \alpha}}\;, \qquad
\Delta_{n} :=  |x_1 -x_2|^{\nu_2}\;,  \qquad \text{where}
\end{align*}
\begin{align*}
\gamma_{n+1}&:= \tilde{\gamma}_{n+1} - \frac{1}{p}  - \kappa\;,\qquad \text{and}\\
\tilde{\gamma}_{n+1}&:=\frac{\gamma_0(1 - \gamma_n) + \alpha }{1 - \gamma_n + \alpha}\;. 
\end{align*}
%
%Observe that we have subtracted an extra $\kappa$ in the definition of $\gamma_{n+1}$ to ensure that the strict inequality in \eqref{e:aaa} is satisfied. 
It is readily checked that for all $n$,
\begin{equation}\begin{aligned}\label{eq:algebra}
 \gamma_0 + \frac12 \alpha \nu_n & = \tilde\gamma_{n+1} 
 = 1 - \frac{\nu_n}{2}(1-\gamma_n)\;,\\
 1 - \frac{\nu_n}{2} + \gamma_n  & > \tilde\gamma_{n+1}\;,
\end{aligned}\end{equation}%
the latter being equivalent to the inequality $\gamma_{n}<\alpha+\gamma_0$, which can be checked by induction. 

We now claim that, for every $n \ge 0$, one has the bound
\begin{equ}\label{eq:induction-n}
\sup_{\eps^2 < t \leq T} (t-\eps^2)^{\alpha/2}  \Big( \E \big|\tRte(t) \big|_{\bB_\tau^{\gamma_n}}^p \Big)^{\frac1p}  \ls  \Big( \E   \big\| \theta \big\|_{\cC_{[\eps^2,\tau]}^{\alpha/2,\alpha}}^p +   \E \big\| \theta \big\|_{\cN_\tau}^p\Big)^{\frac1p}\;.
\end{equ}
The case $n = 0$ has just been proved above, so it remains to obtain the remaining bounds
by induction. We now assume that \eref{eq:induction-n} holds and we aim to show that this implies
the same bound with $n$ replaced by $n+1$.

As above, we fix $t,x_1,x_2$ in each step. Taking \eqref{eq:algebra} and \eqref{eq:induction-n} into account and applying \eqref{e:bs} if $t-\eps^2 > 2 \Delta_n$, we obtain
\begin{equs}[e:imp2]
\big(& \E \big|\tRte(t;x_1,x_2) \big|^p \big)^{\frac1p}  \ls  \\
&\ls \bigg[\Big( \E   \big\| \theta \big\|_{\cC_{[\eps^2,\tau]}^{\alpha/2,\alpha}}^p \Big)^{\frac1p}  + \Big( \E \big| \tRte(t-\Delta_n) \big|_{\bB^{\gamma_n}}^p \Big)^{\frac1p}\bigg]   |x_1 - x_2|^{\tilde{\gamma}_{n+1}} \\
&\ls  \Big(1 + (t-\Delta_n-\eps^2)^{-\alpha/2} \Big) \Big( \E   \big\| \theta \big\|_{\cC_{[\eps^2,\tau]}^{\alpha/2,\alpha}}^p 
+  \E \big\| \theta \big\|_{\cN_\tau}^p\Big)^{\frac1p}  |x_1 - x_2|^{\tilde{\gamma}_{n+1}}\\
&\ls (t-\eps^2)^{-\alpha/2}  \Big( \E   \big\| \theta \big\|_{\cC_{[\eps^2,\tau]}^{\alpha/2,\alpha}}^p 
+   \E \big\| \theta \big\|_{\cN_\tau}^p\Big)^{\frac1p}   |x_1 - x_2|^{\tilde{\gamma}_{n+1}}\;.
\end{equs}
On the other hand we get in the same way as in \eqref{e:appp} that if  $t-\eps^2\leq 2{\Delta_n}$ we have
\begin{equs}[e:apppn]
\Big( \E \big|\tRte(t;x_1,x_2) \big|^p \big)^{\frac1p}  \ls& \,  \E \| \theta \|_{\cN_{\tau}}^p  |x_1 -x_2|^{\tilde{\gamma}_{n+1}} (t-\eps^2)^{-\alpha/2}\;. 
\end{equs}
The bound \eqref{e:Gubingr} on the spatial $\delta$ operator is strong enough to be applicable in every step. Applying Lemma \ref{lem:GRR} we obtain indeed that \eqref{eq:induction-n} holds with $n$ replaced by $n+1$, as required, so that it holds for every $n$.
%\begin{equ}[e:Horrayn]
%\sup_{\eps^2 < t \leq T} (t-\eps^2)^{\beta}  \Big( \E \big|\tRte(t) \big|_{\bB_\tau^{\gamma_n}}^p \big)^{\frac1p}  \ls  \Big( \E   \big\| \theta \big\|_{\cC_{[\eps^2,\tau]}^{\alpha/2,\alpha}}^p +   \E \big\| \theta \big\|_{\cN_\tau}^p\Big)^{\frac1p}\;.
%\end{equ}

Now on the one hand, we have that $\gamma_{n+1}>\gamma_n$ as long as $\gamma_n \neq [\gamma_-,\gamma_+]$ where
\begin{equs}
   \gamma_{\pm}& :=
  \frac12\Big(1 + \alpha + \gamma_0 - \frac1p - \kappa\Big)
 \\
   \pm &  \frac12 \sqrt{ {(1 - \gamma_0 -\alpha )^2}  +
     \Big( \frac1p + \kappa \Big)^2 + 2 \Big( \frac1p + \kappa \Big)(1 + \alpha - \gamma_0)   }.
\end{equs}
On the other hand, the mapping 
\begin{equ}
\gamma \mapsto \frac{\gamma_0 (1 - \gamma) +\alpha}{1 - \gamma + \alpha} - \kappa - \frac1p
\end{equ}
is monotonically increasing (as can be checked easily by calculating its derivative). As $\gamma_-$ is a fixed point of this map and for $\kappa$ small enough and $p$ large enough we have $\gamma_0 \leq \gamma_-$, this implies that $\gamma_n \leq \gamma_-$ for all $n$. Hence we can conclude that $\gamma_n$ converges to $\gamma_-$ as $n$ goes to infinity.

Now we use the elementary estimate $\sqrt{a^2 +b^2+c^2} < |a|+ |b|+|c|$ for $a,b,c \neq 0$  to get
\begin{equs}
\gamma_- >&  \,  \alpha+ \gamma_0  -  \frac{1}{p} -\kappa  - \sqrt{\frac12 \Big( \frac{1}{p} +\kappa\Big) (1 + \alpha - \gamma_0) }  .
\end{equs}
Therefore,  for a given $\gamma$ satisfying \eref{e:recongam}, the desired bound \eqref{e:Rt0} follows from \eqref{eq:induction-n} if we choose $\kappa$ small enough and $p$ large enough.
\end{proof}
 With Lemma \ref{le:mistakecorrected} in hand, we will now derive bounds on the 
 space-time regularity as well as the dependence on $\eps$. To this end we 
 introduce yet another stopping time. We write (recalling the definition of 
 the norm $ |X|_{\dD^{\alpha,\eps}}$ in \eqref{e:DNorm}).
 \begin{equ}[e:DefTau11]
\rho^{X}_{ \infty} \, = \, \inf \Big\{  t \geq 0\colon  \| X_\eps -  X  \|_{\cN_t}    \geq \eps^{\aX} \quad  \text{or} \quad |X_\eps(t)|_{\dD^{\aX,\eps}} \geq K   \Big\}.
\end{equ}
Then we get the following statement.
 \begin{corollary}\label{cor:Rte}
Suppose that $\frac13 <  \alpha <  \aX <   \frac{1}{2}$ and let $\tilde{\alpha}< \frac{\alpha+\aX}{2}$. Let $\tau$ be a stopping time that almost surely satisfies 
\begin{equs}[e:COndST]
 \tau  \leq \rho^{X,\aX}_{\eps, K}\wedge \rho^{X}_{\infty}  \wedge T,
\end{equs}
where $\rho^{X,\aX}_{\eps, K}$ is defined as above in \eqref{e:DefTau}.
Then for any $\lambda < \frac{\aX-\tilde{\alpha}}{2}$  and for $p$ large enough  we get 
\begin{equs}[e:RRbo1]
\E  \Big( \big\| \Rte \big\|_{\bB^{2\tilde{\alpha}}_{[\eps^2, \tau] ,\tilde{\alpha}/2}}  \Big)^p   \ls   \,   
T^{\lambda p}  \Big(\E \, \| \theta \|_{\cC_{[\eps^2,\tau]}^{\alpha/2,\alpha}}^p + \E \, \| \theta \|_{\cN_{\tau}}^p \Big) .
\end{equs}
%
%Suppose in addition that 
%\begin{equs}[e:COndST3]
% \tau  \leq \rho^{X}_{\infty} \wedge T
%\end{equs}
%almost surely. 
Furthermore, for any
\begin{equ}
\bar{\lambda} < \alpha  \, \frac{\alpha +\aX -2\tilde{\alpha}}{\alpha +\aX   } 
\end{equ}
and for $p$ large enough, we have
\begin{equs}[e:RRbo2]
\E  \Big(\big\| \Rte - \Rt \big\|_{\bB^{2\ta}_{[\eps^2, \tau] ,\tilde{\alpha}/2}} \ \Big)^p   \ls \eps^{\bar{\lambda} p }      \, \Big(   \E \, \| \theta \|_{\cC_{[\eps^2,\tau]}^{\alpha/2,\alpha}}^p + \E \, \| \theta \|_{\cC^\alpha_{\tau}}^p  \Big)  .
\end{equs}
\end{corollary}
\begin{proof}[Proof of Corollary \ref{cor:Rte}.] 
We define the processes $\tilde{\theta}, \,\tilde{\Psi}^\theta_\eps, \tilde{X}_\eps$ and  $\tRte, $ as above in \eqref{e:DefTtPR}. Furthermore, we denote by $\tRt$ and $\tPt$ the analogous quantities for $\eps=0$. We will actually establish the slightly stronger statements
\begin{equs}[e:RRbo1a]
\E  \Big( \big\| \tRte \big\|_{\bB^{2\tilde{\alpha}}_{[\eps^2, T] ,\tilde{\alpha}/2}}  \Big)^p   \ls   \,   
T^{\lambda p}  \Big(\E \, \| \tilde{\theta} \|_{\cC_{[\eps^2,T]}^{\alpha/2,\alpha}}^p + \E \, \| \tilde{\theta} \|_{\cN_{T}}^p \Big) .
\end{equs}
and
\begin{equs}[e:RRbo2a]
\E  \Big(\big\| \tRte - \tilde{R}^\theta \big\|_{\bB^{2\ta}_{[\eps^2, T] ,\tilde{\alpha}/2}} \ \Big)^p   \ls \eps^{\bar{\lambda} p }      \, \Big(   \E \, \| \tilde{\theta} \|_{\cC_{[\eps^2,T]}^{\alpha/2,\alpha}}^p + \E \, \| \tilde{\theta} \|_{\cC^\alpha_{T}}^p  \Big)  .
\end{equs}
To this end, we will apply the key bound \eqref{e:Rt0} from Lemma \ref{le:mistakecorrected} to this situation. More precisely, we will use that 
\begin{equ}
 \sup_{\eps^2 < s \leq T} (s-\eps^2)^{\alpha_L p/2}  \,\,\E   \big\| \tRte(s)   \|_{\bB^{\gamma_L}}^p   \,  \ls \,  \,    \E \, \| \tilde{\theta} \|_{\cC_{[\eps^2,T]}^{\alpha_L/2,\alpha_L}}^p + \E \, \| \tilde{\theta} \|_{\cN_{T}}^p   \label{e:Rt10}
\end{equ}
for different values of $p$, $\gamma_L$, and $\alpha_L$ which will be specified below.

In order to improve the estimate \eqref{e:Rt10} to the desired estimate \eqref{e:RRbo1a} we have to move the temporal supremum under the expectation. We get for any  $\eps^2 < s <t$  that
\begin{equs}[e:cr1]
\E\Big(  \Big| &(t-\eps^2)^{\tilde{\alpha}/2} \, |\tRte(t) |_{\bB^{2\tilde{\alpha}}} - (s-\eps^2)^{\tilde{\alpha}/2} |\tRte(s ) |_{\bB^{2\tilde{\alpha}}} \, \Big|^p  \Big) \\
& \ls  (t-s)^{{\tilde{\alpha}p/2}}  \, \E \big|\tRte(t)  \big|_{\bB^{2\tilde{\alpha}}}^p \\
&\qquad  \,+\,  (s-\eps^2)^{{\tilde{\alpha}p/2} } \,  \E  \Big(\big| |\tRte(t) |_{\bB^{2\tilde{\alpha}}}  - |\tRte(s) |_{\bB^{2\tilde{\alpha}}}\big|^p\,  \Big)  .  \qquad 
\end{equs}
To bound the first term on the right hand side of \eqref{e:cr1}, we fix $0 < \lambda_1 \leq \frac{\tilde\alpha}{2}$ and write
\begin{equs}
(t-s)^{{\tilde{\alpha}p/2}}  \, \E \big|\tRte(t) \big|_{\bB^{2\tilde{\alpha}}}^p &\ls \,  (t-s)^{{\lambda_1}  p} \, (t-\eps^2)^{ (\frac{\tilde{\alpha}}{2}-{\lambda_1})  p}  \, \E \big|\tRte(t) \big|_{\bB^{2\tilde{\alpha}}}^p. %\label{e:cr2}
\end{equs}

Then we apply  \eqref{e:Rt10} with 
\begin{equs}
\gamma_L &:= 2 \tilde{\alpha}, \\
\alpha_L &:= 2 \tilde\alpha - \aX + \frac{1}{p}   + \sqrt{\frac{1}{2p}  (1 + \alpha - \aX) }     +\kappa,
\end{equs}
for some small value of $\kappa$. We make the assumption that $p$ is large enough and $\kappa$ small enough to ensure that $\alpha_L < \alpha$ and $\alpha_L < \tilde{\alpha}-2{\lambda_1}$. The condition $\alpha_L < \alpha$ can always be realised due to the assumption $\aX +\alpha>2 \tilde{\alpha}$. In order to  also ensure the second condition $\alpha_L < \tilde{\alpha}-2{\lambda_1}$ one has to choose ${\lambda_1}$ sufficiently small. 
In this way we obtain 
\begin{align}
\label{e:cr9}
(t-s)^{{\tilde{\alpha}} p/2}  \, \E \big|\tRte(t) \big|_{\bB^{2\tilde{\alpha}}}^p &\ls   (t-s)^{{\lambda_1} p} 
(t-\eps^2)^{(\frac12(\tilde{\alpha}-{\alpha}_L)-\lambda_1)p}
\\& \qquad \times
\Big( \E \, \| \tilde{\theta} \|_{\cC_{[\eps^2,T]}^{\alpha/2,\alpha}}^p + \E \, \| \tilde{\theta} \|_{\cN_{T}}^p\Big)\;.
\end{align}

For the second term on the right hand side of \eqref{e:cr1} we fix $\lambda_2>0$. Then we get using H\"older's inequality  once
\begin{equs}
(s&-\eps^2)^{{\tilde{\alpha}} p/2} \,  \E  \Big( \big| |\tRte(t) |_{\bB^{2\tilde{\alpha}}}  - |\tRte(s) |_{\bB^{2\tilde{\alpha}}}\big|^p \,  \Big) \\
&\ls \bigg[ (s-\eps^2)^{ \frac{(\tilde{\alpha} + {\lambda_2})p}{2}} \,  \E  \Big(  |\tRte(t)  |_{\bB^{2\tilde{\alpha}+2{\lambda_2}}}^p + |\tRte(s) |_{\bB^{2\tilde{\alpha}+2{\lambda_2}}}^p  \Big)\bigg]^{\frac{\tilde{\alpha} }{\tilde{\alpha}+ {\lambda_2}}} \\
&\qquad \times\bigg[  \E \Big(  \Big| |\tRte(t)  |_{\bN}  - |\tRte(s) |_{\bN}\Big|^p   \,  \Big) \bigg]^{\frac{{\lambda_2} }{\tilde{\alpha}+ {\lambda_2}}}.\label{e:cr4}
\end{equs}
Now we apply \eqref{e:Rt10} with
\begin{equs}
\gamma_L &:= 2 (\tilde{\alpha}+\lambda_2), \\
\alpha_L &:= 2 (\tilde\alpha +\lambda_2)- \aX + \frac{1}{p}   + \sqrt{\frac{1}{2p}  (1 + \alpha - \aX) }     +\kappa,
\end{equs}
for some small value of $\kappa>0$. In order to ensure that $\alpha_L<\alpha$ we choose $p$ very large and $\kappa$ very small.  Then due to the assumption that $\alpha+\aX>2\tilde{\alpha}$ it is possible to choose $\lambda_2$ small enough to ensure that $\alpha_L<\alpha$. 
The condition on the blowup in this situation is $\tilde\lambda := \frac{ \tilde{\alpha} - \alpha_L + {\lambda_2}}{2} > 0 $ which is satisfied as soon as 
\begin{equ}
0 <   \aX - \tilde\alpha -\lambda_2 - \frac{1}{p}   - \sqrt{\frac{1}{2p}  (1 + \alpha - \aX) }     -\kappa.
\end{equ}
This can be achieved by choosing $p$ even larger and $\kappa$ even smaller.

In this way, we can estimate the first factor on the right-hand side of \eqref{e:cr4} by
\begin{equs}[e:cr3]
 (s &-\eps^2)^{ \frac{(\tilde{\alpha} + \lambda_2)  p}{2}} \,  \E  \Big(   |\tRte(t)  |_{\bB^{2\tilde{\alpha}+2\lambda_2}}^p + |\tRte(s) |_{\bB^{2\tilde{\alpha}+2{\lambda_2}}}^p \Big)\\
 &\ls
 (s-\eps^2)^{\tilde\lambda p}
  \Big(\E \, \| \tilde{\theta} \|_{\cC_{[\eps^2,T]}^{\alpha/2,\alpha}}^p + \E \, \| \tilde{\theta} \|_{\cN_{T}}^p\Big).
\end{equs}
In order to get a bound on the last factor on the right-hand side of \eqref{e:cr4} we write using the definition \eqref{e:DefRteA} of $\tRte$ 
\begin{equs}
\sup_{x,y}|\tRte&(t; x,y) - \tRte(s ;x,y)|   \\
 \, \ls  \, &  \sup_{x,y} \Big(  \big| \tPte(t,x) - \tPte(s,x) \big|+     \big| \tilde{\theta}(t,x) -\tilde{\theta}(s,x) \big| \,  \big| \tilde{X}_\eps (t,y)  \big| \\
 & \qquad +     \big| \tilde{\theta}(s,x) \big|  \big| \tilde{X}_\eps(t,y) - \tilde{X}_\eps( s,y)   \big| \Big).
 \end{equs}
We take the $p$-th moments of this inequality. Then we use the bound \eqref{e:FF1} on the regularity of $\tPte$ (note here that $\tPte$ satisfies the assumption of Lemma  \ref{le:PTE} with $\theta(t)$ replaced by the adapted process $\theta(t) \mathbf{1}_{\{ t < \tau \}}$). For $p$ large enough we can choose the temporal regularity $\alpha_1$ in this bound to be $\frac{\aX}{2}$. For the process $X_\eps$ we make use of the  condition \eqref{e:COndST} on the stopping time $\tau$ to get a deterministic bound: 
\begin{equs}
\big| &\tilde{X}_\eps(t,y) - \tilde{X}_\eps( s,y)   \big|\\
& \ls 
 \begin{cases}
|t-s|^{\frac{\aX}{2}}\|\tilde{X}_\eps\|_{\cC^{\aX/2,\aX}_{\tau} }\qquad &\text{if $t \leq \tau$}
\\   |\tau -s|^{\frac{\aX}{2}}\|\tilde{X}_\eps\|_{\cC^{\aX/2,\aX}_{\tau} } + 
 |t-\tau|^{\frac{\aX}{2}}\|\tilde{X}_\eps(\tau)\|_{\dD^{\aX,\eps} }
 \; &\text{if $s \leq \tau <t  \leq \tau+1$}
 \\ |t-s|^{\frac{\aX}{2}} \|\tilde{X}_\eps(\tau)\|_{\dD^{\aX,\eps} }
 \quad &\text{if $\tau < s <t \leq \tau+1$}
\end{cases}
 \\ &\ls K |t-s|^{\frac{\aX}{2}} .
\end{equs}
In the two remaining cases $s \leq \tau < \tau +1 \leq t$ and $\tau+1 \leq s$ we get the same bound easily.
Note that this is the only point in the whole article where we actually
make use of the norm $\dD^{\aX,\eps}$.

We obtain for $p$ large enough
\begin{equs}[e:zweitens]
 \E\Big(   \Big| |\tRte(t)&  |_{\bN}  - |\tRte(s) |_{\bN}\Big|^p   \Big)\\
 & \ls |t-s|^\frac{\aX p}{2} \E \, \| \tilde{\theta} \|_{\cN_{\tau}}^p + |t-s|^\frac{\alpha p}{2} \E \, \| \tilde{\theta} \|_{\cC_{[\eps^2,\tau]}^{\alpha/2,\alpha}}^p.
\end{equs}
 Then summarising \eqref{e:cr4} -- \eqref{e:zweitens} we obtain 
\begin{equs}[e:cr8]
(s&-\eps^2)^{{\tilde{\alpha}} p/2} \,  \E  \Big( \big| |\tRte(t) |_{\bB^{2\tilde{\alpha}}}  - |\tRte(s) |_{\bB^{2\tilde{\alpha}}}\big|^p \,  \Big)\\
& \ls  |t-s|^\frac{\lambda_2\alpha p}{2(\lambda_2 +\tilde{\alpha})}
 (s-\eps^2)^{\frac{\tilde\lambda\tilde\alpha}{\tilde\alpha+\lambda_2}}
 \Big(\E \, \| \tilde{\theta} \|_{\cN_{\tau}}^p +\E \, \| \tilde{\theta} \|_{\cC_{[\eps^2,\tau]}^{\alpha/2,\alpha}}^p\Big) .
\end{equs}
Taking $\lambda_2$ and $\kappa$ sufficiently small and $p$ sufficiently large, the exponent ${\frac{\tilde\lambda\tilde\alpha}{\tilde\alpha+\lambda_2}}$ can be arbitrarily close to $\frac{\aX-\tilde\alpha}{2}$.
Combining \eqref{e:cr9} and \eqref{e:cr8} and applying again Lemma~\ref{lem:GRR} yields the desired bound \eqref{e:RRbo1}. 

To get the bound \eqref{e:RRbo2} we use an interpolation argument.  In fact, in the same way as above we can write 
\begin{equs}
\sup_{x,y,t} &|\tRte(t; x,y) - \tRt(t ;x,y)|   \\
 \, \ls  \, &    \sup_{x,t}  \big| \tPte(t,x) - \tPt(t,x) \big|+ \sup_{x,y,t}  \big| \tilde{\theta}(t,x) \big|  \big| \tilde{X}_\eps(t,y) - \tilde{X}( t,y)   \big|.
 \end{equs}
Hence the regularity results from Lemma \ref{le:PTE} on $\Pte$ as well as the condition \eqref{e:COndST}  on $\tau$ imply that for any small $\kappa>0$ and for any  $p$ large enough we have
\begin{equ}
\E \big\|  \tRte - \tRt   \big\|_{\bN_T}^p \ls \eps^{\alpha-\kappa}  \,\E\big\| \tilde{\theta} \big\|_{\cC^\alpha_T}^p. 
\end{equ}
On the other hand we can apply \eqref{e:RRbo1} for close to maximal spatial regularity to obtain
\begin{equs}
\E  \Big( \big\| \tRte \big\|_{\bB^{\alpha+\aX-\kappa}_{[\eps^2, T] ,\tilde{\alpha}/2}} \Big)^p   \ls   \,     \E \, \| \tilde{\theta} \|_{\cC_{[\eps^2,T]}^{\alpha/2,\alpha}}^p + \E \, \| \tilde{\theta} \|_{\cN_{T}}^p  .
\end{equs}
Then finally, we can conclude by using the deterministic bound 
\begin{equ}
\big\|  \tRte - \tRt   \big\|_{\bB^{2\tilde{\alpha}}_t} \ls \Big(  \|  \tRte  \big\|_{\bB^{\frac{2\tilde{\alpha}}{\lambda}}_t} + \big\| \tRt   \big\|_{\bB^{\frac{2\tilde{\alpha}}{\lambda}}_t}  \Big)^{\lambda} \, \big\|  \tRte - \tRt   \big\|_{\bN_t}^{1-\lambda}
\end{equ}
for $\lambda = \frac{2\tilde{\alpha}}{\alpha +\aX -\kappa}$. We obtain the desired bound \eqref{e:RRbo2} by choosing $\kappa$ small enough.
\end{proof}

\section{The reaction term}

In this section we derive bounds for the convergence of the reaction term.
\label{sec:GauFluc}
\subsection{The Gaussian process and stochastic fluctuations}\label{ss:GauFluc}

As before, we consider the  solution to the approximate stochastic heat equation
\begin{equ}[e:R1]
X_\eps(t,x) \, = \, \sum_{k \in \Z^*}    q_\eps^k  \,   \xi^{k}_\eps(t) \, e^{ikx}\;,
\end{equ}
together with its area process $\XX_\eps(t;x,y)$, where we use the notation from Section \ref{sec:PrlmCalc}.
We recall in particular that
\begin{align*}
 \E \big[ \xi_\eps^k(s) \ot \xi_\eps^l(t) \big]
    = \delta_{k,-l} \KK_k^{|t-s|} \Id\;,
\end{align*}
where 
\begin{align*}
  \KK_k^{t} = 
  e^{-f(\eps k)k^2 t}
\end{align*}
%\begin{align*}
%  \KK_k^{s,t} = 
%   \left\{ \begin{array}{ll}
%  e^{-f(\eps k)k^2|t-s|} - e^{-f(\eps k)k^2(t+s)}\;,
%   & \text{$k \neq 0$},\\
%  s \wedge t,
%   & \text{$k =0$}.\end{array} \right.
%\end{align*}
for $k \in \Z^*$.
Recall also that 
\begin{align*}
  \Lambda = \frac{1}{2\pi}  \int_{\R_+} 
    \int_\R \frac{(1- \cos(z t))h^2(t)}{t^2 f(t)}  \, \mu(dz) \, dt \,.
\end{align*}
The goal of this section is to prove the following result, which yields a sharp bound, uniform in time, on the difference
\begin{align*}
 %\phi_\eps(t,x) := 
 D_\eps \XX_\eps(t, \cdot)
     - \Lambda \Id
\end{align*}
measured in the spatial negative Sobolev norm $| u |_{H^{-\et}}$ for $\et < \frac12$. 
\begin{proposition}\label{prop:rand-fluc}
Let $1 \leq p < \infty$ and $0 < \et < \frac12$. Then for all $T >0$ and for all $\kappa > 0$ sufficiently small, we have
\begin{align}
\label{e:RF1}
 \E \bigg[ \sup_{t \in [0,T]} \Big|
  D_\eps \XX_\eps(t, \cdot) - \Lambda \Id 
  \Big|_{H^{-\et}}^p\bigg]^{\frac1p}
   &\lesssim \eps^{\et - \kappa}\;.
%   \\\label{e:RF2}
%   \sup_{t \in [0,T]} t^{\frac{\et}{2}} \big|
%  \Lambda_\eps(t)- \Lambda \big|
%  & \lesssim \eps^{\et}\;.
\end{align}
\end{proposition}
\begin{proof}
For fixed $z \in \R$ and for $\eps>0$ we introduce  the quantity%
\begin{equ}[e:Le]
\Lambda_{z,\eps} \, = \,  \frac{1}{\eps} \sum_{k \in \Z^\star}  \big(q_\eps^k \big)^2 \, \big(1 - \cos(z  \eps k) \big) 
= 
\frac{1}{4 \pi } \sum_{k \in \Z^\star}  \eps  \, \frac{h(\eps k)^2  }{ \,    f(\eps k)     } \,
\frac{1 - \cos (z  \eps k)}{ |\eps k|^2}  \, .
\end{equ}
We also set
\begin{equ}
\Lambda_{z,0} = \frac{1}{2\pi}  
    \int_{\R^+} \frac{(1- \cos(z t))h^2(t)}{t^2 f(t)}  \,  \, dt,
\end{equ}
and for $\eps>0$ we define the integrated version  
\begin{align*}
\Lambda_{\eps} = \int_{\R}\Lambda_{z,\eps} \, \mu(dz)\;.
\end{align*}
It follows from Lemma \ref{lem:asym} and Lemma \ref{lem:sym} below that
\begin{align*}
& \E \bigg[ \sup_{t \in [0,T]} \Big|
  D_\eps \XX_\eps(t, \cdot) - \Lambda_\eps \Id 
  \Big|_{H^{-\et}}^p\bigg]^{\frac1p}
 \\&\lesssim
 \int_\R 
   \E \bigg[ \sup_{t \in [0,T]} \Big|
    \frac1\eps\XX_\eps(t; \cdot, \cdot + \eps z) - \Lambda_{z,\eps} \Id 
  \Big|_{H^{-\et}}^p\bigg]^{\frac1p}
 |\mu|(dz) 
 \\&\lesssim \eps^{\et - \kappa} 
 \int_\R 
  |z|^{1 + \et - \kappa}
 |\mu|(dz) 
 \lesssim\eps^{\et - \kappa}\;.
\end{align*}
Furthermore, we claim that for any $t \in [0,T]$ and
$\kappa > 0$, 
\begin{align}
\label{eq:La-claim}
  | \Lambda_{z,\eps} - \Lambda_{z,0} | 
      \lesssim  
       \eps |z|^2 \;.
\end{align}
Integrating this inequality with respect to $z$ and using the fact that the second moment of $|\mu|$ is finite, the result follows.

It thus remains to prove the claim \eref{eq:La-claim}. 
%For this purpose we consider the quantity
%%
%\begin{align*}
%  \tilde\Lambda_{z,\eps}
%    = \frac1\eps
% \sum_{k \in \Z^\star}
%   \big(q_\eps^k \big)^2 \,  \big( 1 - \cos(  z \eps k ) \big)
%\end{align*}
%%
%and estimate for $\lambda \in [0,1]$,
%%
%\begin{equation}\begin{aligned}
%\label{eq:cl-1}
% | \Lambda_{z,\eps}(t) -  \tilde\Lambda_{z,\eps} |
%  &\leq
%  \frac1\eps
% \sum_{k \in \Z^\star}
%   \big(q_\eps^k \big)^2 \, | \KK_k^{t,t} - 1 |  
%      \big( 1 - \cos(  z \eps k ) \big) 
%\\&  \lesssim
%  \frac{1}{\eps}
% \sum_{k \in \Z^\star}
%   \frac{1}{k^2 } \, e^{-2 c_f k^2 t} 
%        (k\eps z)^{2\lambda} 
%\\&  \lesssim
%  \eps^{2 \lambda -1}|z|^{2\lambda} \, t^{ \frac12 -\lambda}\;.
%\end{aligned}\end{equation}
%%
As in \cite[Proposition 4.6]{HM10} we use the fact that 
for any function $g : \R \to \R$ of bounded variation,
\begin{equ}
\Bigl|\sum_{k \in \Z^\star}\eps g(\eps k) - \int_{\R} g(t)\,dt\Bigr|
		 \le \eps \big( |g|_\BV + |g(0)| \big) \;.
\end{equ}
Using this fact together with the assumptions on $h$ and $f$, we obtain
\begin{equs}
 | \Lambda_{z,\eps}& -  \Lambda_{z,0} |
  \lesssim 
   \eps  \bigg| s \mapsto \frac{h^2(s)}{f(s)} 
   \frac{ 1- \cos(sz)}{s^2} \bigg|_\BV + \eps |z|^2
\\  &\lesssim  
   \eps  \Big| \frac{h^2}{f} \Big|_{L^\infty}  
     \Big| s \mapsto \frac{1 - \cos(sz)}{s^2} \bigg|_\BV \!\!
  + \eps \Big| \frac{h^2}{f} \Big|_\BV  
     \Big| s \mapsto \frac{ 1- \cos(sz)}{s^2} \bigg|_{L^\infty} \!\!\!\!+ \eps |z|^2
\\  &\lesssim 
   \eps |z|^2\;.  \label{eq:cl-2}
\end{equs}
which proves the claim.
\end{proof}

For a matrix $A$, it will be convenient to work with the decomposition $A = A^{+} + A^{-}$, where
\begin{align*}
   [A^{\pm}]_{ij} := \frac12 \Big(   A_{ij} \pm    A_{ji}\Big)\;.
\end{align*}   
The following two lemmas are the main ingredients in the proof of \eref{e:RF1}.

\begin{lemma}\label{lem:asym} 
Let $1 \leq p < \infty$ and $0 < \et < \frac12$. Then for all $T>0$, for all $\kappa >0$ small enough and $z \in \R$ we have
\begin{equ}[e:RF1-asym]
\E \bigg[  \sup_{ t \in [0, T]}\Big|
 \frac{1}{\eps} \XX_\eps^-(t;\cdot, \cdot + \eps z) \Big|_{H^{-\et}}^p\bigg]^{\frac1p}
   \lesssim \eps^{\et - \kappa} |z|^{1 + \et - \kappa}\;.
\end{equ}
\end{lemma}

\begin{proof}
We will show that for any $0<\et< \frac12$ and for $0\leq \kappa <\frac12$   we have for all $0 \leq s <t \leq T$,
\begin{align}\label{eq:A-min2}
\E \Big|\frac{1}{\eps}\Big(\XX_\eps^-(t;\cdot, \cdot +\eps z) - \XX_\eps^-(s;\cdot, \cdot +\eps  z)\Big)\Big|_{H^{-\et}}^2 
&\lesssim \eps^{2(\et -  \kappa)}  |z|^{2(1+\et - \kappa )} |t-s|^{\kappa }\;,\\
\E \Big|\frac{1}{\eps}\Big(\XX_\eps^-(t;\cdot, \cdot +\eps z)\Big)\Big|_{H^{-\et}}^2 
&\lesssim \eps^{2(\et -  \kappa)}  |z|^{2(1+\et - \kappa )}\;.\label{eq:A-min}
\end{align}
The result then follows from Lemma~\ref{lem:Kolmogorov}.
As above in Lemma \ref{le:GRP}, we use the fact that $\XX_\eps^-(t,\cdot)$ belongs to the $H^{-\et}$-valued second order Wiener chaos to see that \eqref{eq:equiv-mom} is satisfied (see Lemma~\ref{lem:Nelson}).

Recall that
\begin{align*}
\XX_\eps (t;x, x + \eps z)
 =  \sum_{k, l \in \Z^{\star}} {q_\eps^k q_\eps^l} \, I_{kl} (\eps z) \,\big(\xi_k(t) \ot \xi_l(t) \big)
  e^{i(k+l)x}\;,
\end{align*}
where for $k\neq -l$ we have
\begin{align*}
 I_{kl}(\eps z) & :=  \int_{0}^{\eps z}  
      ( e^{ikw} - 1 ) \, i l e^{ilw} \dd w \
    % \\&
     	= 
     		\frac{l}{k+l} \  \big( e^{i (k+l) \eps z} -1\big) 
-  \big(  e^{i l \eps z} -1\big) .
\end{align*}
As above this sum has to be interpreted as the limit as $N \to \infty$ of the sums over $0 < |k|,|l| \leq N$.

As a consequence, we have the identity
\begin{align*}
 \XX_\eps^- (t;x, x + \eps z)
  =  \sum_{ k, l \in \Z^\star}{q_\eps^k q_\eps^l}  \, {J^{-}_{kl}( \eps z) }  \, \big(\xi_k(t) \ot \xi_l(t) \big)\, e^{i (k+l)x}\;,
\end{align*}
where for $|k| \neq |l|$,
\begin{align*}
 J^{-}_{kl}(\eps z) & := \frac{1}{2}( I_{kl}(\eps z) -  I_{lk}(\eps z))  = \frac12 (l-k)  
     		\bigg(  \frac{ e^{i(k+l)\eps z}-1}{k+l}  - \frac{  
				   e^{i l \eps z} - e^{i \eps k z} }{l-k} \bigg) \;.
\end{align*}
For $|k|=|l|$ this expression must be read as the appropriate limit.
Writing $\zeta_{k,l}^{i,j}(s,t) := \xi_k^i(t)\xi_{l}^j(t) - \xi_k^i(s)\xi_{l}^j(s)$ for brevity, we obtain
\begin{align*}
 \E \Big| &\frac{1}{\eps}\Big(\XX_\eps^-(t;\cdot, \cdot +\eps z) - \XX_\eps^-(s;\cdot, \cdot +\eps  z)\Big)\Big|_{H^{-\et}}^2 
\\& \ = \sum_{\substack{ k, l,\bar k ,\bar l \in \Z^\star\\ k+l = \bar k +\bar l }} \big(1+(k+l)^2\big)^{-\et}
    {q_\eps^k q_\eps^{l} q_\eps^{\bar k} q_\eps^{\bar l}} \,\,
    \frac{1}{\eps}{J^{-}_{k,l}(\eps z) \,   \frac{1}{\eps}J^{-}_{-\bar k, -\bar l}}(\eps z) 
\\& \qquad \qquad \times      \E\bigg[\sum_{i \neq j}
   \zeta_{k,  l}^{i,j}(s, t)    
   \zeta_{-\bar k, -\bar l}^{i,j}(s,t)  \bigg]
\;.
\end{align*}
Note that the diagonal entries of $\zeta_{k,l}^{i,j}(s,t) $ vanish so that in the last line we only have to sum over indices $i \neq j$.
For such indices $i \neq j$, 
we use the fact that 
\begin{align*}
\E\Big[ \xi_k^i(s) \, \xi_{l}^j(s) \,\xi_{-\bar k}^i(t) \, \xi_{-\bar l }^j(t) \Big]
 &= \E\Big[\xi_k^i(s) \xi_{-\bar k}^i(t)\Big]
 \E\Big[\xi_{l}^j(s)\xi_{-\bar l}^j(t)\Big]
\\	& = \delta_{k,\bar k} \, \delta_{l,\bar l} \,\KK_k^{t-s} \KK_{l}^{t-s}\;,
\end{align*}
to obtain, for any $\kappa \in [0,1]$ and for all indices satisfying $k+l=\bar k+ \bar l$,
\begin{align}
\E\Big[
   \zeta_{k, l}^{i,j}(s,t)  \, 
   	\zeta_{-\bar k, -\bar l}^{i,j}(s,t)  \Big] \notag
	& = 2 \delta_{k,\bar k} \Big( 1 - \KK_k^{t-s} \KK_{l}^{t-s}
   \Big) \notag
 \\& \lesssim \delta_{k,\bar k} 
 \big(1 - e^{- (f(\eps k) k^2 + f(\eps l) l^2 ) |t-s|}\big)
 \notag
  \\&\lesssim  
   \delta_{k,\bar k} \Big( f(\eps k)^\kappa |k|^{2\kappa} + f(\eps l)^\kappa |l|^{2\kappa}\Big)
     |t-s|^{\kappa}\;. \label{e:zettcalc}
\end{align}
Furthermore, a simple calculation yields that 
\begin{align*}
 \Big| \frac{1}{\eps}J^{-}_{kl}(\eps z) \Big|^2
    =  (k-l)^2 z^2 \,  
       \big[S((k+l)\eps z) - S((k-l)\eps z) \big] ^2\;,
\end{align*}
where $S(x) = \sin(x/2)/x$.
Using the Assumption \ref{a:h} that $h$ is bounded, we obtain for $k, l \neq 0$ that, 
\begin{align*}
q_\eps^k q_\eps^l
   \lesssim \frac{1}{\sqrt{f(\eps k)f(\eps l)}|k l|}
   \lesssim \frac{1}{\sqrt{f(\eps k)f(\eps l)}\big|(k+l)^2 - (k-l)^2\big|}\;.
\end{align*}
Moreover, since  by Assumption \ref{a:f} $f \geq c_f$, it follows that for all $0 \leq \kappa < \frac12$,
\begin{align*}
&\E  \Big|\frac{1}{\eps}\Big(\XX_\eps^-(t;\cdot, \cdot +\eps z) - \XX_\eps^-(s;\cdot, \cdot +\eps  z)\Big)\Big|_{H^{-\et}}^2
 \\&  \lesssim     \sum_{ k,  l \in \Z^\star } \big(1+(k+l)^2\big)^{-\et}
    (q_\eps^k q_\eps^{l})^2
    \Big|\frac{1}{\eps}J^{-}_{kl}(\eps z)\Big|^2
\\& \qquad \qquad \times 
   \Big( f(\eps k)^\kappa |k|^{2\kappa} + f(\eps l)^\kappa |l|^{2\kappa}\Big)
     |t-s|^{\kappa}
 \\&  \lesssim      |t-s|^{\kappa}  |z|^2 \sum_{k, l \in \Z^\star} 
 \big(1+ (k+l)^2\big)^{-\et}
 |k-l|^2 \bigg|\frac{
		{ S((k+l)\eps z) - S((k-l)\eps z) }}
       {(k+l)^2 - (k-l)^2} \bigg|^2 
   \\& \qquad \qquad \times    \Big( |k|^{2\kappa} + |l|^{2\kappa}\Big)
  \\& \lesssim |t-s|^{\kappa} |z|^2
     \sum_{|k| \neq  |l| \in \Z} 
   	\big(1+ k^2\big)^{-\et}
 l^2 \bigg|\frac{
		{ S(k\eps z) - S(l\eps z)  }}
       {k^2 - l^2} \bigg|^2 \Big( |k|^{2\kappa} + |l|^{2\kappa}\Big)
	\;,
\end{align*}
where in the last line we used the change of variables $(k+l, k-l) \leadsto (k,l)$.
We infer that
\begin{align}
\E &\Big|\frac{1}{\eps}\Big(\XX_\eps^-(t;\cdot, \cdot +\eps z) - \XX_\eps^-(s;\cdot, \cdot +\eps  z)\Big)\Big|_{H^{-\et}}^2 \label{e:integ}
\\  & \ls  |t-s|^\kappa \eps^{2\et - 2 \kappa} |z|^{2 + 2\et - 2\kappa}\int_{\R^2} 
     \bigg( \frac{x}{y^\et}
       \frac{S(x) - S(y)}{x^2 - y^2} \bigg)^2 (x^2 + y^2)^{\kappa}
    \dd x \dd y\;. \notag
\end{align}
It is an elementary exercise to show that
\begin{align*}
 \frac{|S(x) - S(y)|}{|x^2 - y^2|} 
   \lesssim 1 \wedge \frac{1}{x^2 + y^2}\;.
\end{align*}
Hence, taking into account that $\et < \frac12$  it follows that the integral on the right hand side of \eqref{e:integ}  converges. This establishes the desired estimate  \eqref{eq:A-min}.

Furthermore, to prove \eqref{eq:A-min} we observe that for $i \neq j$, 
\begin{align*}
\E\Big[ \xi_k^i(t) \, \xi_{l}^j(t) \,\xi_{-\bar k}^i(t) \, \xi_{-\bar l }^j(t) \Big]
 &= \E\Big[\xi_k^i(t) \xi_{-\bar k}^i(t)\Big]
 \E\Big[\xi_{l}^j(t)\xi_{-\bar l}^j(t)\Big]
= \delta_{k,\bar k} \, \delta_{l,\bar l}\;.
\end{align*}
Using this identity we obtain
\begin{align*}
 \E \Big| &\frac{1}{\eps}\Big(\XX_\eps^-(t;\cdot, \cdot +\eps z) 
 \Big|_{H^{-\et}}^2 
\\& \ = \sum_{\substack{ k, l,\bar k ,\bar l \in \Z^\star\\ k+l = \bar k +\bar l }} \big(1+(k+l)^2\big)^{-\et}
    {q_\eps^k q_\eps^{l} q_\eps^{\bar k} q_\eps^{\bar l}} \,\,
    \frac{1}{\eps}{J^{-}_{k,l}(\eps z) \,   \frac{1}{\eps}J^{-}_{-\bar k, -\bar l}}(\eps z) 
\\& \qquad \qquad \times   
\E\Big[\sum_{i \neq j} \xi_k^i(t) \, \xi_{l}^j(t) \,\xi_{-\bar k}^i(t) \, \xi_{-\bar l }^j(t) \Big]
 \\&  \lesssim     \sum_{ k,  l \in \Z^\star } \big(1+(k+l)^2\big)^{-\et}
    (q_\eps^k q_\eps^{l})^2
    \Big|\frac{1}{\eps}J^{-}_{kl}(\eps z)\Big|^2\;.
\end{align*}
The desired estimate \eqref{eq:A-min} follows from this expression, by repeating the argument for \eqref{eq:A-min2} with $\kappa = 0$.
\end{proof}
  
The following result gives the corresponding estimate for the symmetric part $\XX_\eps^+$. Recall that $\Lambda_{z,\eps}$ has been defined in \eqref{e:Le}.
   
\begin{lemma}\label{lem:sym}
Let $1 \leq p < \infty$ and $0 < \et < \frac12$. For any $T>0$, for  $\kappa > 0$ small enough, and for any $z \in \R$ we have
\begin{equ}[e:RF0]
 \E \bigg[ \sup_{t \in [0,T]}\Big|
 \frac{1}{\eps} \XX_\eps^+(t;\cdot, \cdot + \eps z) 
 - \Lambda_{z,\eps} \Id
   \Big|_{H^{-\et}}^p\bigg]^{\frac1p}
   \lesssim \eps^{\et - \kappa} |z|^{1 + \et - \kappa}\;.
\end{equ}
\end{lemma} 
   
\begin{proof}
In view of Lemma \ref{lem:Kolmogorov}, it suffices to show that 
\begin{align}\label{eq:sym1}
\E \Big| \frac{1}{\eps} \XX_\eps^+ (t;x, x + \eps z) 
  		 - \frac{1}{\eps} \XX_\eps^+ (s;x, x + \eps z) 
		 \Big|_{H^{-\et}}^2
& \lesssim |t-s|^\kappa		 \eps^{2 \et - 2\kappa}
  |z|^{2 + 2 \et - 2\kappa}\;,\\
\E \Big| \frac{1}{\eps} \XX_\eps^+ (t;x, x + \eps z) 
 -  \Lambda_{z,\eps} \Id 		 \Big|_{H^{-\et}}^2
& \lesssim  \eps^{2 \et}  |z|^{2 + 2 \et}\;.\label{eq:sym2}
\end{align}
As in the proof of Lemma \ref{lem:asym}, we write
\begin{align*}
  \XX_\eps^+ (t;x, x + \eps z)
 &  =  \sum_{k, l \in \Z^\star}{q_\eps^k q_\eps^l} \, J^+_{kl}(\eps z) \big(\xi_k(t) \ot \xi_l(t) \big)e^{i (k+l)x}\;,
\end{align*}
where
\begin{align*}
 J_{kl}^+(\eps z) & := \frac{1}{2} \big( I_{kl}(\eps z) +  I_{lk}(\eps z)\big)
          = \frac{1}{2}
     		\big(1 - e^{ik\eps z}\big)\big(1 - e^{il\eps z}\big)\;,
\end{align*}
and $I_{kl}(\eps z)$ is as in the proof of Lemma \ref{lem:asym}.
As above we write
\begin{equs}
\tilde\zeta_{k,l}(s,t) =&  \xi_k(t) \ot \xi_{l}(t) 
-   \xi_k(s) \ot \xi_{l}(s)
\end{equs}
for brevity. Then we obtain 
\begin{align*}
 &\E \Big| \frac{1}{\eps} \XX_\eps^+ (t;x, x + \eps z) 
  		 - \frac{1}{\eps} \XX_\eps^+ (s;x, x + \eps z) 
		 \Big|_{H^{-\et}}^2
 \\ & = \sum_{\substack{ k, l,\bar k ,\bar l \in \Z^\star\\k+l=\bar k +\bar l}} \big(1+(k+l)^2\big)^{-\et}
    {q_\eps^k q_\eps^{l} q_\eps^{\bar k} q_\eps^{\bar l}} \,
    \frac{1}{\eps}{J_{kl}^+(\eps z) \, \frac{1}{\eps} J_{-\bar k,-\bar l}^+}(\eps z)
\\& \qquad \qquad  \times  
     \E \tr\big( \tilde\zeta_{k,l}(s,t) \,   {\tilde\zeta_{-\bar l,-\bar k}(s, t)}\big) \;.
\end{align*}
A case by case argument yields 
\begin{align*}
   &  \E \tr \big( \tilde\zeta_{k,l}(s,t)  {\tilde\zeta_{-\bar l,-\bar k}(s,t)}
	    \big)
 =  
2 \Big(n^2 \delta_{k, \bar l} \delta_{\bar k, l} + n \delta_{k,\bar k} \delta_{l, \bar l} \Big)
      	\Big( 1
		   - \KK_k^{t-s} \KK_{l}^{t-s}\Big)\;,
\end{align*}
and using the definition of $\KK$ we infer as above in \eqref{e:zettcalc} that
\begin{align*}
 &\Big| 1
		   - \KK_k^{t-s} \KK_{l}^{t-s}\Big|  \lesssim \Big( f(\eps k)|k|^{2\kappa}
  + f(\eps l)|l|^{2\kappa} \Big)|t-s|^\kappa
\end{align*}
for $\kappa \in [0,1]$.
Using the estimate
$|q_k| \lesssim \frac{1}{\sqrt{f(\eps k)}|k|}$ for $k \neq 0$, together with the identity
\begin{align*}
|J_{kl}^+(z \eps)|^2 = 2 \big( 1 - \cos(z \eps k) \big) \,\big( 1 - \cos(z \eps l)\big)\;,
\end{align*}
we obtain
\begin{align*}
&\E \Big| \frac{1}{\eps} \XX_\eps^+ (t;x, x + \eps z) 
  		 - \frac{1}{\eps} \XX_\eps^+ (s;x, x + \eps z) 
		 \Big|_{H^{-\et}}^2
 \\& \lesssim  
  \sum_{k, l \in \Z^\star} \big(1+(k+l)^2\big)^{-\et}
    (q_\eps^k q_\eps^{l})^2 \,
   \frac{1}{\eps^2} |J_{k,l}^+|^2 
	\Big( 1 - \KK_k^{t-s} \KK_{l}^{t-s} \Big)		   
 \\& \lesssim  
  \sum_{k, l \in \Z^\star} \big(1+(k+l)^2\big)^{-\et}
     \frac{1 - \cos(k\eps z)}{f(\eps k) \eps k^2 }
		\frac{1 - \cos(l\eps z)}{f(\eps l) \eps l^2 }
		\\& \qquad \times \Big(f(\eps k)^\kappa |k|^{2\kappa} + f(\eps l)^\kappa |l|^{2\kappa} \Big)|t-s|^\kappa
 \\& \lesssim  
  \sum_{k, l \in \Z^\star} \frac{1}{|k|^\et |l|^\et}
     \frac{  1 - \cos(k\eps z) }{k^2\eps }
		\frac{  1 - \cos(l\eps z) }{l^2\eps } \Big( |k|^{2\kappa} + |l|^{2\kappa} \Big)|t-s|^\kappa
 \\& \lesssim |t-s|^\kappa		 \eps^{2 \et - 2\kappa}
  |z|^{2 + 2 \et - 2\kappa}\;,
\end{align*}
which proves \eqref{eq:sym1}.

In order to prove \eqref{eq:sym2}, we note that for all $x \in [-\pi,\pi]$ and $t \geq 0$
\begin{align*}
 \frac{1}{\eps}\E \, \XX_\eps^+ (t;x, x + \eps z)
  = \frac{1}{\eps}\sum_{k\in \Z^\star} |q_\eps^k|^2 \, J_{k,-k}^+(\eps z) \Id =  \Lambda_{z,\eps} \Id.
\end{align*}
Moreover, we write
\begin{equs}
\hat\zeta_{k,l}(t) = \big( \xi_k(t) \ot \xi_{l}(t) \big)- \delta_{k,-l}\Id
\end{equs}
for brevity. Then we obtain 
\begin{align*}
 &\E \Big| \frac{1}{\eps} \XX_\eps^+ (t;x, x + \eps z) -  \Lambda_{z,\eps} \Id  \Big|_{H^{-\et}}^2
 \\ & = \sum_{\substack{ k, l,\bar k ,\bar l \in \Z^\star\\k+l=\bar k +\bar l}} \big(1+(k+l)^2\big)^{-\et}
    {q_\eps^k q_\eps^{l} q_\eps^{\bar k} q_\eps^{\bar l}} \,
    \frac{1}{\eps}{J_{kl}^+(\eps z) \, \frac{1}{\eps} J_{-\bar k,-\bar l}^+}(\eps z)
\\& \qquad \qquad  \times  
     \E \tr\big( \hat\zeta_{k,l}(t) \,   {\hat\zeta_{-\bar l,-\bar k}(t)}\big) \;.
\end{align*}
A case by case argument yields 
\begin{align*}
   &  \E \tr\big( \hat\zeta_{k,l}(t)  {\hat\zeta_{-\bar l,-\bar k}(t)}
	    \big)
  =  n^2 \delta_{k, \bar l} \delta_{\bar k, l} + n \delta_{k,\bar k} \delta_{l, \bar l}\;.
\end{align*}
%and using the definition of $\KK$ we infer as above in \eqref{e:zettcalc} that
%\begin{align*}
% &\Big| \KK_k^{t,t} \KK_{l}^{t,t}
%		   - 2\KK_k^{s,t} \KK_{l}^{s,t}
%		   +  \KK_k^{s,s} \KK_{l}^{s,s}\Big|  \lesssim \Big( f(\eps k)|k|^{2\kappa}
%  + f(\eps l)|l|^{2\kappa} \Big)|t-s|^\kappa
%\end{align*}
%%
%for $\kappa \in [0,1]$.
Arguing as above we obtain
%$|q_k| \lesssim \frac{1}{\sqrt{f(\eps k)}|k|}$ for $k \neq 0$, together with the identity
%\begin{align*}
%|J_{kl}^+(z \eps)|^2 = 2 \big( 1 - \cos(z \eps k) \big) \,\big( 1 - \cos(z \eps l)\big)\;,
%\end{align*}
%we obtain
%
\begin{align*}
&\E \Big| \frac{1}{\eps} \XX_\eps^+ (t;x, x + \eps z) -  \Lambda_{z,\eps} \Id \Big|_{H^{-\et}}^2
 \\& \lesssim  
  \sum_{k, l \in \Z^\star} \big(1+(k+l)^2\big)^{-\et}
    (q_\eps^k q_\eps^{l})^2 \,
   \frac{1}{\eps^2} |J_{k,l}^+|^2 
 \\& \lesssim  
  \sum_{k, l \in \Z^\star} \big(1+(k+l)^2\big)^{-\et}
     \frac{1 - \cos(k\eps z)}{f(\eps k) \eps k^2 }
		\frac{1 - \cos(l\eps z)}{f(\eps l) \eps l^2 }
 \\& \lesssim  
  \sum_{k, l \in \Z^\star} \frac{1}{|k|^\et |l|^\et}
     \frac{  1 - \cos(k\eps z) }{k^2\eps }
		\frac{  1 - \cos(l\eps z) }{l^2\eps } 
 \\& \lesssim \eps^{2 \et}
  |z|^{2 + 2 \et}\;,
\end{align*}
which proves \eqref{eq:sym2}.
\end{proof}

\subsection{Bounds for the reaction term}
\label{ss:ReactionTerm}

We will now derive the estimates for the reaction terms $\Phi^{F(\bar{u})}, \, \Phi_\eps^{F(u_\eps)}, \,\Upsilon^{\bar{u}}$, and $\Upsilon_\eps^{u_\eps}$  defined in Section \ref{sec:OFP}. 

To this end let $\eps \in (0,1)$. We fix  H\"older exponents $0 <  \alpha < \frac12$ and let $0<\eta<\frac12$. Furthermore, we fix $\R^{n\times n}$-valued functions $\phi \in  \cC_T^\alpha$ and  $\phie = 
 \phi + \CE_\eps$ with $\CE_\eps \in \cN([0,T]; H^{-\et})$. (See \eqref{e:alphaHol2} and \eqref{e:alphaHolHor} above for the definition of the H\"older norms with blowup).
We shall use the notation $\| u \|_{H_t^{-\et}}$ to denote the temporal  supremum of the negative spatial Sobolev norm, i.e.,
\begin{align*}
 \| u \|_{H_t^{-\et}} := \sup_{s \in [0,t]} |u(s)|_{H^{-\et}}\;.
\end{align*}
In our application we will have 
\begin{equation*}
  \phi(t) = \Lambda \Id\;, \quad 
  \CE_\eps(t) =  D_\eps \XX_\eps(t, \cdot) - \Lambda\Id\;. 
\end{equation*}
Of course, $\phi$ is constant in space and not merely $\cC^\alpha$, but we will not make use of this.

Furthermore, we fix $\R^n$-valued functions $u, \ue$ and $\R^{n\times n}$ -valued functions $v,\ve$ in $L^\infty([0,T]; \cC^\alpha)$.
In our application, $u$ and $u_\eps$ will be as in the previous sections, and $ v = u'$, $ v_\eps = u_\eps'$.

Then we set
\begin{align*}
 \Phi(t) = \int_0^t S(t-s)\, F(u(s)) \dd s\;,\qquad
 \Upsilon(t) =\int_0^t S(t-s) \BB(u(s),v(s)) \dd s\;,
\end{align*}
where
\begin{align*}
 \BB^i(u, v) =  \partial_k G^i_j(u) v_l^k  \phi^{l,m}  \,  v_m^j\;.
\end{align*}
Similarly, for $\eps  >0$ we define
\begin{align*}
 \Phi_\eps(t) = \int_0^t S_\eps(t-s) F(\ue(s)) \dd s\;,\quad
 \Upsilon_\eps(t) =\int_0^t S_\eps(t-s) \BB_\eps(\ue(s),\ve(s)) \dd s\;,
\end{align*}
where
\begin{align*}
 \BBe^i(u,v)
 = 
  \partial_k  G^i_j(u)
    \,v^k_l    \, \phie^{l,m}  v^j_m\;.
\end{align*}

Throughout the remainder of this section we assume that the norms
\begin{align*}
    \| \phi \|_{\cC^{\alpha}_T}\;,\qquad
  \| u \|_{\cC^{\alpha}_T}\;,\qquad
   \| v\|_{\cC^{\alpha}_T}\;,\qquad
  \| u_\eps \|_{\cC^{\alpha}_T}\;, \qquad
    \| \ve \|_{\cC^{\alpha}_T}\;,\qquad
\end{align*}
are bounded by a constant $K > 0$, which does not depend on $\eps$.
This constant will often be suppressed below.

We shall first prove a bound on the difference between $\Phi$ and $\Phi_\eps$.
\begin{proposition}\label{prop:Phi-bound}
Let $0 \leq \alpha \leq \gamma \leq 1$. Then, for any $t \in [0,T]$  we have
\begin{align*}
\big|  \Phi(t) - \Phi_\eps(t) \big|_{\cC^{\gamma}}
&   \lesssim 
     t^{1 - \frac12(\gamma - \alpha)} \| u - u_\eps \|_{\cC_t^\alpha} + 
      	\eps
 \;,\\
 \big\|  \Phi  - \Phi_\eps \big\|_{\cC^{  \frac{\gamma}{2}}([0,t],\cN)}
&   \lesssim 
   t^{1 - \frac{\gamma}{2}}\|u - u_\eps \|_{\cC_t^\alpha}  +\eps^{1-\frac{\gamma}{2}}\;,
\end{align*}
with implied constant depending on $K$ and $T$.
\end{proposition}
\begin{remark}
In the statements of  Proposition \ref{prop:Phi-bound} and Proposition \ref{prop:Ups-bound} we do not include the positive powers of $t$ after the terms involving powers of $\eps$. We simply bound these by powers of $T$, which will be absorbed into the implied constants, because we do not need them.
\end{remark}

\begin{proof}
Using Corollary \ref{cor:S-Seps} we obtain for any $\kappa >0$ small enough,
\begin{align*}
&   \big| \Phi(t) - \Phi_\eps(t) \big|_{\cC^\ga}
 \\& \leq  \bigg|  \int_0^t S(t-s)  ( F(u) -  F(u_\eps))(s) \dd s \bigg|_{\cC^\ga}
   +  \bigg|  \int_0^t (S - S_\eps)(t-s)  F(u_\eps)(s) \dd s \bigg|_{\cC^\ga}
  \\& \lesssim
   t^{1 - \frac12(\gamma - \alpha)} \| F(u) - F(u_\eps) \|_{\cC_t^\alpha} 
   		+ \eps \,  t^{1 - \frac12(\ga - \alpha + 1 +\kappa )}
				\| F(\ue) \|_{\cC_t^\alpha}
  \\& \lesssim
   t^{1 - \frac12(\gamma - \alpha)} \| u - u_\eps \|_{\cC_t^\alpha} + 
      	\eps \, t^{1 - \frac12(\ga - \alpha + 1+\kappa )}\;, 
\end{align*}
which proves the first bound.

To prove the second inequality, we write
\begin{align*}
 \big(& \Phi  - \Phi_\eps \big)(t) 
  - \big(\Phi - \Phi_\eps \big)(s) 
 \\& = \int_0^s \Big( (S - S_\eps)(t-r) -  (S - S_\eps)(s-r) \Big)
   F(u_\eps(r)) \dd r
   \\& \quad + 
    \int_s^t  (S - S_\eps)(t-r) 
   F(u_\eps(r)) \dd r
\\& \quad+  \int_0^s \big( S(t-r) -  S(s-r) \big)
     \big( F(u(r)) - F(u_\eps(r))\big) \dd r
   \\& \quad + 
    \int_s^t   S(t-r) 
     \big( F(u(r)) - F(u_\eps(r)) \,\big)  \dd r\;   \;
    \\&   =: I_1 + I_2 + I_3 + I_4\;.
\end{align*}

We shall estimate the first term in two ways. First, using \eref{eq:S-diff} and the fact that $|(S - S_\eps)(t-s)|_{\cC^\alpha \to \cN} \lesssim 1$, we obtain $\kappa>0$ small enough,
\begin{equation}\begin{aligned}\label{eq:I-1}
 |I_1|_{\cN} & \lesssim \bigg| \int_0^s \Big( (S - S_\eps)(t-r) -  (S - S_\eps)(s-r) \Big)
   F(u_\eps(r)) \dd r \bigg|_{\cN}
\\&    \lesssim 
  \int_0^s | (S - S_\eps)(s-r)|_{\cC^\alpha \to \cN}
  | F(u_\eps(r)) |_{\cC^\alpha} \dd r 
\\&    \lesssim \eps
  \int_0^s (s-r)^{-\frac12(1-\alpha+\kappa)}  \| u_\eps \|_{\cC_t^\alpha} 
 \dd r
    \lesssim \eps  \;.
\end{aligned}\end{equation}
On the other hand, using Lemma \ref{le:newOp} and Lemma \ref{lem:timecontS} for every and any $\kappa>0$ sufficiently small we have the bound
\begin{equs}\label{eq:I-1b}
 \Big| \int_0^{s-\eps^2 \vee 0}&\big( S_\eps(t-r) -  S_\eps(s-r)  \big)   F(u_\eps(r))  \dd r \Big|_{\cN}
\\ & \ls
  \int_0^{s-\eps^2  \vee 0} \Big| S_\eps (t-s+\eps^2 ) -  S_\eps ( \eps^2 \Big)  \Big|_{\cC^{ 2 +\kappa} \to \cN } \,\\
  &\qquad \qquad \times \Big| S_\eps ( s-r -\eps^2)  \Big|_{\cC^{\alpha} \to \cC^{2  +\kappa}}  \big| F(u_\eps(r)) \big|_{\cC^\alpha} \,ds\\
 &\ls (t-s) \int_0^{s-\eps^2 \vee 0}   (s-r -\eps^2)^{-\frac{ 2 -\alpha +2\kappa }{2}} dr \ls (t-s).
  \end{equs}
  The analogous bound 
  \begin{equ}[eq:I-1d]
  \Big| \int_0^{s}  \big( S(t-r) -  S(s-r)  \big)   F(u_\eps(r))  \dd r \Big|_{\cN} \ls (t-s),
  \end{equ}
  follows in the same way. Using Lemma \ref{lem:timecontS} once more we also obtain for any $0 \leq \lambda < \frac{2+\alpha}{4}$ and for $\kappa>0$ small enough,
\begin{equs}\label{eq:I-1e}
 \Big| \int_{s-\eps^2 \vee 0}^s & \big( S_\eps(t-r) -  S_\eps(s-r)  \big)  F(u_\eps(r))  \dd r \Big|_{\cN}
\\ & \ls
  \int_{s-\eps^2 \vee 0}^{s} \Big| S_\eps \Big(t-s+\frac{s-r}{2}\Big) -  S_\eps \Big( \frac{s-r}{2} \Big)  \Big|_{\cC^{ 2\lambda +\kappa} \to \cN } \,\\
  &\qquad \qquad \times \Big| S_\eps \Big(\frac{s-r}{2} \Big)  \Big|_{\cC^{\alpha} \to \cC^{2 \lambda +\kappa}}  \big| F(u_\eps(r)) \big|_{\cC^\alpha} \,ds\\
 &\ls (t-s)^{\lambda} \int_{s-\eps^2 \vee 0}^s  \eps^{ 2 \lambda }(s-r)^{-\lambda}  (s-r)^{-\frac{2 \lambda - \alpha +2\kappa}{2}} dr\ls (t-s)^{ \lambda} \eps^{2\lambda}.
  \end{equs}
Combining the  estimates~\eqref{eq:I-1},~\eqref{eq:I-1b},~\eqref{eq:I-1d} and~\eqref{eq:I-1e}  we infer that for $\gamma \in [0,1]$,
\begin{equation}\begin{aligned}
\label{eq:I-1c}
 |I_1|_{\cN}
 & \lesssim \eps^{1-\frac{\gamma}{2}} (t-s)^{\frac{\gamma}{2}}.
\end{aligned}\end{equation}
Using \eref{eq:S-diff} once more, the term $I_2$ is bounded by
\begin{equation}\begin{aligned}\label{eq:I-2}
  |I_2|_{\cN}
&  \lesssim \int_0^{t-s} |(S - S_\eps)(r)|_{\cC^\alpha \to \cN} 
  		\|F(u_\eps) \|_{\cC_t^\alpha} \dd r 
\\&  \lesssim \eps \int_0^{t-s} r^{-\frac12(1 - \alpha+\kappa)} 
  		\|u_\eps \|_{\cC_t^\alpha} \dd r 
 \lesssim \eps \, (t-s)^{\frac12(1 + \alpha -\kappa)} \;.
\end{aligned}\end{equation}
To bound $I_3$ we proceed as in \eref{eq:I-1} to infer that
\begin{equs}[e:I-33]
  |I_3|_{\cN} \notag
   & \leq \Big| \int_0^s  \big( S(t-r) -  S(s-r) \big)
     \big( F(u(r)) - F(u_\eps(r))\big) \dd r  \Big|_{\cC^0}\notag
  \\&\lesssim  \int_0^s
    (t-s)^{\frac{\gamma}{2}}
    (s-r)^{-\frac12(\gamma-\alpha)}
       \| F(u) - F(u_\eps) \|_{\cC_t^\alpha} \dd r \notag
  \\&\lesssim      (t-s)^{\frac{\gamma}{2}}  s^{1- \frac12(\gamma-\alpha)}
    \|u - u_\eps \|_{\cC_t^\alpha} \;.
\end{equs}
Note that here we have used the fact that the true heat semigroup $S(t)$ satisfies the regularisation properties without introducing a small constant $\kappa$.
The last term can be bounded brutally by
\begin{align*}
 |I_4|_{\cN}
   & \leq |t-s|   \| F(u) - F(u_\eps) \|_{\cC_t^\alpha} 
  \lesssim |t-s|   \| u - u_\eps \|_{\cC_t^\alpha}
 \\& \lesssim  
    (t-s)^{\frac{\gamma}{2}}  t^{1 - \frac{\gamma}{2}}
    \|u - u_\eps \|_{\cC_t^\alpha} \;.
\end{align*}
Combining all of these estimates, we obtain the desired bound.
\end{proof}

The following result gives a bound on the difference between $\Upsilon$ and 
$\Upsilon_\eps$.

\begin{proposition}\label{prop:Ups-bound}
Let $0 < \alpha < \frac12$ and suppose that
$\alpha \leq \gamma < 1$ and $0 <\et < \frac12$. Let $\kappa > 0$. 
Then, for any $t \in [0,T]$ we have 
\begin{align*}
 |   \Upsilon(t) - \Upsilon_\eps(t) |_{\cC^{\gamma}}
&   \lesssim 
 \eps
+   \| \CE_\eps \|_{H_{t}^{-\et}} 
\\& \qquad 
 + t^{1 - \frac12(\gamma - \alpha) - \kappa}
	 \Big( \| u - \ue\|_{\cC_t^\alpha} +  \| v - \ve\|_{\cC_t^\alpha} \Big)\;,\\
 \|   \Upsilon  - \Upsilon_\eps 
 \|_{\cC^{  \frac{\gamma}{2}}([0,t],\cN)}& \lesssim
  \eps^{1-\frac{\gamma}{2}} 
+ 
       \big \| \CE_\eps \big\|_{H_{t}^{-\et}}
%\\
%&\qquad      
+ t^{1 - \frac{\gamma}{2}} 
      \big(\| u - u_\eps \|_{\cC_t^\alpha}
     + \| v - v_\eps \|_{\cC_t^\alpha}\big)\;,
\end{align*}
with implied constants depending on $K$ and $T$. 
\end{proposition}

\begin{proof}
We rewrite the difference $\Upsilon- \Upsilon_\eps$ as
\begin{align*}
  \Upsilon(t) - \Upsilon_\eps(t) 
   & =
    \int_0^t  S_\eps(t-s)  \big( \BB(\ue, \ve)- \BBe(\ue, \ve) \big)(s) \dd s 
  \\& \qquad  +   \int_0^t S_\eps(t-s) \big( \BB(u, v) - \BB(\ue, \ve) \big)(s) \dd s
  \\& \qquad  + \int_0^t  (S - S_\eps)(t-s) \BB(u, v)(s) \dd s\;.
  \\& =: I_1 + I_2 + I_3\;.
\end{align*}
Using  Lemma \ref{lem:conv-estimate} in the second term  we obtain
\begin{align*}
   | I_1 |_{\cC^\ga}
& \leq  \Big| \int_0^t S_\eps(t-s)
    \Big(   D G (\ue) \ve \CE_\eps \ve  \Big)(s)
  \dd s \Big|_{\cC^\ga}
 \\& \lesssim  
 \| \CE_\eps \|_{H_{t}^{-\et}} 
      \int_0^t  
      (t-s)^{-\frac{\alpha + \gamma}{2} - \frac14} 
       \big| \big( \ve D G (\ue)  \ve \big)  (s)\big|_{\cC^\alpha} \dd s
\\ & \lesssim 
   t^{\frac{3}{4} - \frac{\alpha+\gamma}{2}} \| \CE_\eps \|_{H_{t}^{-\et}} \;.
\end{align*}
Here we have made use of the conditions $\eta<\alpha <\frac12, \,\gamma<1$ to ensure that the  exponent is positive.
Furthermore, 
\begin{align*}
 | I_2 |_{\cC^\ga}
 &  \lesssim
     t^{1- \frac12(\gamma - \alpha)-\kappa}
	 \|  \BB(u, v) - \BB(\ue, \ve) \|_{\cC_t^\alpha}
  \\& \lesssim
     t^{1- \frac12(\gamma - \alpha)-\kappa}
	 \Big( \| u - \ue\|_{\cC_t^\alpha} +  \| v - \ve\|_{\cC_t^\alpha} \Big)\;,
\end{align*}
and Corollary \ref{cor:S-Seps} yields for any $\kappa>0$ small enough,
\begin{align*}
  | I_3 |_{\cC^\ga}
   & \lesssim \eps \, t^{1- \frac12(1 +\kappa + (\gamma - \alpha) \vee 0)}
  		 \|  \BB(u, v)\|_{\cC_t^\alpha}
\\ & \lesssim \eps\, t^{ \frac{1-\kappa +( \alpha  -\gamma) \wedge 0}{2}}\;.
\end{align*}
Combining these bounds, we obtain the first estimate.

In order to prove the second estimate, we fix $0 \leq s < t \leq T$ and write
\begin{align*}
 \big( \Upsilon - \Upsilon_\eps \big)(t)
   - \big(\Upsilon - \Upsilon_\eps\big)(s)
 =: \sum_{i = 1}^2 \sum_{j=1}^3 J_{ij}\;,
\end{align*}
where
\begin{align*}
  J_{11} & =  \int_0^s \big( S_\eps(t-r) -  S_\eps(s-r) \big)
      \big( \BB - \BBe\big)(\ue, \ve)(r) \dd r\;,\\
  J_{12} & =   \int_0^s \Big( (S - S_\eps)(t-r) -  (S - S_\eps)(s-r) \Big)
 \BB(\ue,\ve)(r) \dd r\;, \\
  J_{13} & =     \int_0^s \big( S(t-r) -  S(s-r) \big)
     \big( \BB(u,v) - \BB(\ue, \ve)\big)(r) \dd r\;,
\end{align*}
and 
\begin{align*}
J_{21} & =     \int_s^t   S_\eps(t-r) 
   \, \big( \BB(u,v) - \BB(\ue, \ve)\big)(r) \dd r\; ,\\
 J_{22} & =     \int_s^t   S_\eps(t-r) 
    \big( \BB - \BBe\big)(\ue, \ve)(r) \dd r\;,\\
 J_{23} & =     \int_s^t  (S - S_\eps)(t-r) 
    \,\BB(u,v)(r) \dd r\; .
  \end{align*}
%Arguing as above in \eref{eq:I-1b} and \eqref{eq:I-1e} we obtain for  any $\kappa>0$ small enough
%\begin{align*}
% J_{11}^{(1)} 
%  \lesssim & \int_{0}^{s}  \Big| S_\eps\Big(t-s + \frac{s-r}{2} \Big) -  S_\eps\Big(\frac{s-r}{2} \Big) \Big|_{\cC^{\gamma+\kappa} \to \cN}  \,\\
%  & \qquad \times 
%  \Big| S_\eps\Big(\frac{s-r}{2} \Big) \Big|_{\cC^{\alpha} \to \cC^{\gamma +\kappa}}  \big| \CE_\eps^{(1)}(r) \big|_{\cC^\alpha} \, dr \\
%%%%%
%\ls & (t-s)^{\frac{\gamma}{2}}   \|  \CE_\eps^{(1)}\|_{\cC_{t,\frac{\et}{2}}^{\alpha}}   \int_{0}^{s}  \big( 1+ (s-r)^{-\frac{\gamma}{2}} \eps^{\gamma} \big) \, \\
%%%%%
%& \qquad \times\Big| S_\eps\Big(\frac{s-r}{2} \Big) \Big|_{\cC^{\alpha} \to \cC^{\gamma +\kappa}}  r^{-\frac{\et}{2}} \, dr. 
%\end{align*}
%Using again that $ \big| S_\eps(t) \big|_{\cC^{\alpha} \to \cC^{\gamma+\kappa}}  \ls t^{ \frac{\alpha -\gamma- 2\kappa}{2}} \vee 1$  by Lemma \ref{le:newOp}, we obtain
%\begin{equ}
%J_{11}^{(1)} \ls  (t-s)^{{\frac{\gamma}{2}}}  \|  \CE_\eps^{(1)}\|_{\cC_{t,\frac{\et}{2}}^{\alpha}} s^{ \big(1-\gamma +\frac{\alpha -\eta}{2} - \kappa\big) \wedge \frac14}.
%\end{equ}
%
%
%Note that here we have not used the positive power of $\eps$. It is possible to obtain a better bound by splitting the integral into two parts as in \eref{eq:I-1b} and \eqref{eq:I-1e}, but we do not need this improved statement. 

In order to bound $J_{11}$ we will use  Lemma \ref{lem:conv-estimate} below. In this way we obtain for $\kappa>0$ small enough,
\begin{align*}
 |J_{11}|_{\cN}
 & \lesssim 
  \int_0^s 
   \bigg| \big(S_\eps(t-r) - S_\eps(s-r) \big)
     \big( \ve D G(\ue) \ve \CE_\eps   \big)(r) \bigg|_{\cN} \dd r
\\ & \lesssim 
   (t-s)^{\frac{\gamma}{2}} \int_0^s (s-r)^{-\frac{\alpha + \gamma}{2} -\frac14} 
       \big| \CE_\eps(r)\big|_{H^{-\et}}
     \big|(\ve D G(\ue) \ve)(r)\big|_{\cC^\alpha}   \dd r
\\ & \lesssim 
       s^{\frac34 -\frac{\alpha  + \gamma  }{2}} (t-s)^{\frac{\gamma}{2}}
     \|  \CE_\eps\|_{H_{t}^{-\et}} .
\end{align*}
The argument for $J_{12}$ is the same as the argument for the bound on $I_1$ in the Proposition \eqref{prop:Phi-bound}. Arguing as in \eref{eq:I-1}--\eqref{eq:I-1c} we obtain
\begin{align*}
    |J_{12}|_{\cN}
    & \ls \eps^{1-\frac{\gamma}{2}} (t-s)^{\frac{\gamma}{2} - \kappa}
  		\| \BB(\ue,\ve) \|_{\cC_t^\alpha}
   \lesssim \eps^{1-\frac{\gamma}{2}} (t-s)^{\frac{\gamma}{2}}\;.
\end{align*}
The term $J_{13}$ can be treated in the same way as the term $I_3$ in the proof of Proposition \eqref{prop:Phi-bound}. Using the  argument from \eref{e:I-33} we obtain
\begin{align*}
 |J_{13}|_{\cN}
  & \lesssim |t-s| \,  \| \BB(u,v) - \BB(\ue, \ve) \|_{\cC_t^\alpha} 
% \\&  \lesssim |t-s| \, 
%    \big(  \| u - u_\eps \|_{\cC_t^\alpha}
%     + \| v - v_\eps \|_{\cC_t^\alpha}\big)
 \\&  \lesssim t^{1 - \frac{\gamma}{2}} |t-s|^{\frac{\gamma}{2}} 
    \big(  \| u - u_\eps \|_{\cC_t^\alpha}
     + \| v - v_\eps \|_{\cC_t^\alpha}\big)\;.
\end{align*}
For the first term involving an integral over $[s,t]$ we get the same bound
\begin{align*}
  |J_{21}|_{\cN}
  & \lesssim |t-s|  \| \BB(u,v) - \BB(\ue, \ve) \|_{\cC_t^\alpha} 
% \\ & \lesssim |t-s| 
%   \big(  \| u - u_\eps \|_{\cC_t^\alpha}
%     + \| v - v_\eps \|_{\cC_t^\alpha}\big)
 \\ & \lesssim t^{1-\frac{\gamma}{2}}|t-s|^{\frac{\gamma}{2}} 
   \big(  \| u - u_\eps \|_{\cC_t^\alpha}
     + \| v - v_\eps \|_{\cC_t^\alpha}\big)\;.
\end{align*}
To bound the term $J_{22}$ we invoke Lemma \ref{lem:conv-estimate} one more time. We obtain\begin{align*}
  |J_{22}|_{\cN}
 & \lesssim 
  \int_s^t 
   \Big| S_\eps(t-r)
       \big(  \ve D G (\ue) \ve \CE_\eps  \big)(r)
   \Big|_{\cN} \dd r
\\ &\lesssim 
   \int_s^t 
    (t-r)^{-\frac{\et}{2} - \frac14}
       \big| ( \ve D G (\ue)  \ve) (r)\big|_{\cC^{\alpha}}\,
       \big |\CE_\eps(r)\big|_{H^{-\et}} \dd r
\\ & \lesssim  
    t^{\frac34 - \frac{\et + \gamma}{2} }(t-s)^{\frac{\gamma}{2}}
       \big \| \CE_\eps \big\|_{H_{t}^{-\et}}\;.
\end{align*}
%The first term can easily be estimated by
%\begin{align*}
%J_{22}^{(1)} &\lesssim  
%  \int_s^t r^{-\frac{\et}{2}}  \| \CE_\eps^{(1)}\|_{C_{t,\frac{\et}{2}}^{\et}} \dd r
% \lesssim (t-s)^{1- \frac{\et}{2} } \| \CE_\eps^{(1)} \|_{C_{t,\frac{\et}{2}}^{\alpha}} 
%\\& \lesssim (t-s)^{\frac{\gamma}{2}} t^{{1- \frac{\et}{2} -\frac{\gamma}{2}}} \| \CE_\eps^{(1)}\|_{C_{t,\frac{\et}{2}}^{\alpha}} .
%\end{align*}
%Note that for $\kappa>0$ small enough the exponents of $t$ appearing in this bound is larger than $\frac{1}{4}$ by the assumptions on $\alpha$ and $\gamma$.
%
Finally, by the argument in \eref{eq:I-2},
\begin{equ}
    |J_{23}|_{\cN}
   \lesssim \eps  \, (t-s)^{\frac{1 + \alpha}{2} -\kappa} 
  		\| \BB(u,v) \|_{\cC_t^\alpha}
   \lesssim \eps (t-s)^{\frac{1 + \alpha}{2} - \kappa} \;.
\end{equ}
Putting everything together, we obtain the desired bound.
\end{proof}

The following lemma has been used in the proof above.

\begin{lemma}\label{lem:conv-estimate}
Let $0 < \eta< \alpha<\frac12$ and $0<\gamma<1$. Then for  $\psi \in \cC^\alpha$ and $\phi \in H^{-\et}$ we have for any $t>0$ and $\eps \in [0,1]$,
\begin{align}\label{e:impbo1}
 \big| S_\eps(t)    (\psi  \,\phi)  \big|_{\cC^\ga}  \ls t^{-\frac{\gamma +\alpha}{2} - \frac{1}{4}} | \phi |_{H^{-\et}} 
   | \psi|_{\cC^\alpha}.
\end{align}
Furthermore, for $s<t$ and for any $0 \leq \lambda \leq 1$ we have
\begin{align}\label{e:impbo2}
 \big| \big( S_\eps(t)  -S_\eps(s) \big)   (\psi  \,\phi)  \big|_{\cC^\ga}  \ls (t-s)^{\lambda} s^{-\frac{\gamma + \alpha}{2}  -\lambda - \frac{1}{4}} | \phi |_{H^{-\et}} 
   | \psi|_{\cC^\alpha}.
\end{align}
\end{lemma}

\begin{proof}
By the assumption $\et<\alpha$ we have $|  \psi  \,\phi |_{H^{-\alpha}} \ls |  \psi |_{\cC^\alpha}  \, |\phi |_{H^{-\eta}} $. (This elementary multiplicative inequality follows by duality from the estimate \eqref{e:pt10} proved above). Therefore the bounds \eqref{e:impbo1} and \eqref{e:impbo2} follow immediately from the bounds
\minilab{e:sec4lem}
\begin{equs}
|S_\eps(t)|_{H^{-\alpha} \to H^{\gamma+\frac{1}{2}}} & \ls t^{-\frac{\gamma +\alpha }{2} -\frac{1}{4}} \label{e:sec4lema}\\
|S_\eps(t) -S_\eps(s)|_{H^{-\alpha} \to H^{\gamma+\frac{1}{2}}} & \ls (t-s)^{\lambda} s^{-\frac{\gamma + \alpha}{2}  -\lambda - \frac{1}{4}} \label{e:sec4lemb},
\end{equs}
and from  Sobolev embedding.

Actually, the identity \eqref{e:sec4lema} follows immediately from the lower bound on $f$. To prove \eqref{e:sec4lemb}, we write $\theta= \gamma + \alpha + \frac{1}{2}$ and estimate
\begin{equs}
\sup_{k \in \Z} |k|^{\theta} &\ \big( e^{-k^2 f(\eps k) s} -e^{-k^2 f(\eps k ) t}  \big) \\
&\ls \sup_{k \in \Z} |k|^{\theta} \  e^{-k^2 f(\eps k) s} \big( 1  -e^{-k^2 f(\eps k ) (t-s)}  \big)\\
& \ls(t-s)^\lambda \sup_{k \in \Z} | k|^{\theta} \  e^{-k^2 f(\eps k) s} |k|^{2\lambda} f(\eps k )^\lambda  \ls  (t-s)^\lambda \, s^{-\lambda-\frac{\theta }{2}}.
\end{equs}
The estimate \eqref{e:sec4lemb} follows from this bound.
\end{proof}
\begin{remark}
Note that these $L^2$ based regularity properties for the heat semigroup are significantly easier to derive than the estimates in H\"older spaces in Section \ref{sec:HarmAn}. Also note that we do not encounter any problems in the time regularity for $s \leq \eps^2$. 
\end{remark}

\section{Rough path estimates}\label{sec:RouPathEst}

In this section we treat the stability of approximations of the term involving $G(u) \, \partial_x u$. We will make heavy use of the rough path bounds provided in Appendix \ref{AppA}.  
We will fix deterministic data ($u, u_\eps, X$, etc.) and derive bounds based on the regularity of this data. There will be no randomness involved. 

We fix H\"older exponents $\frac13<\tilde{\alpha} \leq \alpha < \frac12$. We also fix rough path valued mappings $ (X(t), \XX(t))$ and $(X_\eps(t) ,\XX_\eps(t))$. To be more precise, we will assume that the mappings $[0,T ] \ni t     \mapsto X(t) \in \cC^{\alpha}$ and $[0,T ] \ni t     \mapsto \XX(t) \in \bB^{2\alpha}$ are continuous and that for every $t$ the functions $x \mapsto X(t,x)$ and $(x,y) \mapsto \XX(t;x,y)$ satisfy the consistency relation \eref{eq:cond-it-int}. The functions $(X_\eps(t) ,\XX_\eps(t))$ will be assumed to satisfy the same conditions.  

We will also fix functions $u, \ue \in \cC^\alpha_T$. We assume that for every $t$ the function $u$ is controlled by $X$. More precisely, we will assume that there are bounded functions
\begin{align*}
[0,T] \ni t \mapsto u(t), u'(t)  \in  \cC^{\alpha}\;, \qquad
[0,T] \ni t \mapsto \Ru(t) \in \bB^{2\alpha};\,
\end{align*}
such that for every $t \in (0,T]$ the maps 
\begin{equs}[e:relrp]
x \mapsto u(t,x), \ X(t,x) \quad \text{and} \quad (x,y) \mapsto \Ru(t;x,y)  
\end{equs}
satisfy the relation \eref{eq:contr-rp}. 
In the same way we will assume that the $\ue$ are controlled by $X_\eps$, but only for $t > \eps^2$. More precisely, we will assume that 
there are bounded functions
\begin{align*}
[0,T] \ni t \mapsto u_\eps(t), u_\eps'(t) \in \cC^{\alpha}\;, \qquad 
(\eps^2,T] \ni t \mapsto \Rue(t) \in \bB^{2\alpha}\;,
\end{align*}
such that for every $t \in (\eps^2,T]$ the maps 
\begin{equs}[e:relrpe]
x \mapsto u_\eps(t,x), \ X_\eps(t,x) \quad \text{and} \quad (x,y) \mapsto \Rue(t;x,y)
\end{equs}
satisfy the relation \eref{eq:contr-rp}. Let us emphasise that although we use the notation $u_\eps'(t)$ for all $t \in [0,T]$, we only assume that $u_\eps'(t)$ is the rough path derivative of $u_\eps(t)$ for $t \in (\eps^2, T]$.

%for every $t \in [0,T]$ there exists a function $u_\eps'(t) \in \cC^\alpha$ and for every $t \in (\eps^2, T]$ there exists a function $\Rue(t)$ such that 
%  We will assume that for $t \in (\eps^2,T]$ there are functions $u_\eps'$ and $\Rue$ that are continuous in time so that for every $t$ the $\ue, \ue',\Rue,\XX_\eps$ satisfy the relation \eref{eq:contr-rp}.   
%

Throughout this section we will make the standing assumption that the norms 
\begin{equs}
\| X \|_{\cC^{\alpha}_T}, \,  \| X_\eps \|_{\cC^{\alpha}_T}, \|\XX \|_{\bB^{2\alpha}_T },  \|\XX_\eps \|_{ \bB_T^{2\alpha}} , \, \| u \|_{\cC^{\alpha}_T}, \,  \| u_\eps \|_{\cC^{\alpha}_T},  \,  \| u' \|_{\cC^{\alpha}_{T}}, \,  
\| u_\eps' \|_{\cC^{\alpha}_{T}} 
%\| u_\eps' \|_{\cC^{\alpha}_{[\eps^2,T]}} 
\end{equs}
are bounded by a large constant $K$. We will also assume that for $t > 0$
\begin{equs}[e:lamRem]
|\Ru(t) |_{2 \alpha }  \leq K t^{-  \frac{\alpha}{2} },&  \qquad  |\Ru(t) |_{2 \tilde{\alpha}}  \leq K  t^{-  \frac{\ta}{2} },
\end{equs}
and for $t >\eps^2$
\begin{equs}[e:lamRembar]
 |\Rue(t) |_{ 2 \alpha} \leq K  (t-\eps^2)^{- \frac{\alpha}{2} }, \qquad  |\Rue(t) |_{ 2 \tilde{\alpha}} \leq K  (t-\eps^2)^{- \frac{\tbm}{2} }.
\end{equs}

Most of the constants that appear in this section (or that are suppressed when we write $\ls$) depend on the choice of this constant $K$.   
%
%%
%\begin{equs}[e:epsiFP]
% \XE^i( & t,\cdot) \, = \, \int_0^t S_\eps(t-s) \, \Big[ G(\ue(s))^i_j \,  \De \ue(s)^j \Big] \, ds \,  \\
%& + \int_0^t \!\! \int_{-\pi}^\pi  p^{\eps}_{t-s}(\cdot-y ) \, \partial_k  G\big(\ue(s,y) \big)^i_j  \, u_\eps'(s,y)^k_l    \, D_\eps \XX_\eps(s;y)^{l,m}  u_\eps'(s,y)^j_m   \,  dy \,  ds\;.
%\end{equs}
%%

The main objects under consideration in this section are the quantities
\begin{align*}
 Z(t,x) &= \iint_{-\pi}^x G(u(t,y)) \, d_y u(t,y)\;,\\
 \Xi(t,x) &= \int_0^t  S(t-s)  \partial_x Z(s)  \, ds = \int_0^t   \partial_x S(t-s)   Z(s)  \, ds \;, 
\end{align*}
along with their approximations
\begin{align*}
 Z_\eps(t,x) &= \int_{-\pi}^x \bigg[G(u_\eps(t,y))
 D_\eps u_\eps(t,y)\\
& \qquad\qquad + 
 \partial_k  G\big(\ue(t,y) \big)^i_j  \, u_\eps'(t,y)^k_l    \, D_\eps \XX_\eps(t;y)^{l,m}  u_\eps'(t,y)^j_m \bigg]\, dy\;,\\
 \Xi_\eps(t,x) &= \int_0^t S_\eps(t-s) \partial_x Z_\eps(s) \, ds\; = \int_0^t \partial_x S_\eps(t-s)  Z_\eps(s) \, ds.    
\end{align*}
Here we have made use of the fact that the heat semigroup $S$ as well as the approximated heat semigroup $S_\eps$ commute with the spatial derivative. As above, we have used the notation 
\begin{align*}
D_\eps \XX_\eps(t;y) =  \frac1{\eps}  \int_\R \XX_\eps(t;y, y + \eps z) \, \mu(dz)\;.
\end{align*}
Note that  we have included indices to capture the trilinear structure in the second term on the right-hand side. The linear algebra does not play a crucial role for our argument, and as above we will omit the indices for most of the argument.

% Denote by $\Xi(t,x)$ the function
% %
%\begin{equs}[e:v]
%\Xi(t,x) \,= \,&   \int_0^t  \Big[ \iint_{-\pi}^\pi
%   p_{t-s}(x-y ) \, G \big( u(s,y) \big) \, d_y u(s,y) \Big] \, ds.  
%\end{equs}
%% 
%
 Throughout this section we will make the additional assumption that the function $G$ is bounded with bounded derivatives up to order three. This assumption is removed in Section \ref{sec:OFP} using an appropriate stopping time.

As explained in Section \ref{sec:OFP}, the term
\begin{align}\label{e:epsiFP}
 \partial_k  G\big(\ue(t,y) \big)^i_j  \, u_\eps'(t,y)^k_l    \, D_\eps \XX_\eps(t;y)^{l,m}  u_\eps'(t,y)^j_m \;,
\end{align}
which appears on the right-hand side in the definition of $Z_\eps$, is included
%we have included the extra term
%%
%\begin{equ}
% \int_0^t \!\! \int_{-\pi}^\pi  p^{\eps}_{t-s}(\cdot-y ) \, \partial_k  G\big(\ue(s,y) \big)^i_j  \, u_\eps'(s,y)^k_l    \, D_\eps \XX_\eps(y)^{l,m}  u_\eps'(s,y)^j_m   \,  dy \,  ds
%\end{equ}
%%
%on the right-hand side of \eref{e:epsiFP} 
to ensure that -- at least formally -- $Z_\eps$ approximates $Z$, and therefore $\XE$ approximates $\Xi$. As discussed  in  Sections  \ref{sec:OFP} and \ref{sec:GauFluc}, this term gives rise to the extra term in the limit. 

The main objective of this section is to show that $\Xi_\eps$ indeed approximates $\Xi$. A precise error bound is given by the following result.
In this section we shall use $ \CE_\eps $ as abbreviation for
\begin{equs}
 \CE_\eps \, = \, &  \Big[ \|X-X_\eps \|_{\cC^{\tilde{\alpha}}_T} +  \|\XX-\XX_\eps \|_{\bB^{2\tilde{\alpha}}_T}    +    \| u- u_\eps \|_{\cC^{\tilde{\alpha}}_T} +   \|u'-u_\eps' \|_{\cC^{\tilde{\alpha}}_{[\eps^2,T]}} \\
 &\qquad +  \|\Ru - \Rue\|_{\bB^{2\tilde{\alpha}}_{[\eps^2,T], \tilde{\alpha}}}      \Big]. 
\end{equs}

%
%The main goal of this section is to establish the following bounds.
%We consider the quantities
%\begin{align*}
% Z(t,x) &= \iint_{-\pi}^x G(u(t,y)) \, d_y u(t,y)\;,\\
% \Xi(t,x) &= \int_0^t  \Big[ \iint_{-\pi}^\pi
%   p_{t-s}(x-y) \, d_y Z(s,y) \Big] \, ds\;,   
%\end{align*}
%along with their approximations
%\begin{align*}
% Z_\eps(t,x) &= \iint_{-\pi}^x \bigg[G(u_\eps(t,y))
% D_\eps u_\eps(t,y)\\
%& \qquad\qquad + 
% \partial_k  G\big(\ue(s,y) \big)^i_j  \, u_\eps'(s,y)^k_l    \, D_\eps \XX_\eps(s;y)^{l,m}  u_\eps'(s,y)^j_m \bigg]\, dy\;,\\
% \Xi_\eps(t,x) &= \int_0^t  \Big[ \iint_{-\pi}^\pi
%   p_{t-s}^\eps(x-y) \, d_y Z_\eps(s,y) \Big] \, ds\;.    
%\end{align*}
%

We let $r_+ := 0 \vee r$ denote the positive part of a real number $r$.

\begin{proposition}\label{prop:Xi-bounds}
Let $0 < \gamma < 1$ and $\kappa>0$. Then, for all $t \in [0,T]$ we have
\begin{align*}
 |\Xi(t) - \Xi_\eps(t)|_{C^\gamma}
 &\lesssim 
    		\CE_\eps  \, (t-\eps^2)_+^{\frac12(1-\gamma-\kappa)}
+ 
  \eps^{3 \alpha -1} 
 \,
 (t-\eps^2)_+^{\frac12(1-\gamma-\alpha-\kappa)}
\\& \qquad	 +
        \eps^{1  -\gamma -\kappa}\;.
\end{align*}
\end{proposition}

The constant which is suppressed in the notation depends on $T$ and $K$. 
We shall also prove the following result concerning the time regularity of the difference $ \Xi - \Xi_\eps$.

\begin{proposition}\label{prop:Xi-time}
Let $0 \leq \gamma < \frac12$.Then we have for any $\kappa>0$ small enough and for $t \geq \eps^2$ 
\begin{align*}
 \| \Xi - \Xi_\eps \|_{\cC^{\frac{\gamma}{2}}([\eps^2, t], \cN) }
  \lesssim 
   t^{\frac12(1 -2 \gamma   -\kappa)} \CE_\eps + \eps^{3 \alpha -1} .
%+   \eps^{1- \gamma - \kappa}
%+   \eps^{2\alpha - \gamma - \kappa}\;.
\end{align*}
\end{proposition}

%\begin{lemma}\label{lem:holder-kernel}
%For $\eta \in [0,1]$ we have \comment{add to lemma in Section 4}
%\begin{align*}
%|p_t^\eps - p_t|_{\alpha} &\lesssim
%  t^{-\frac12(1+\alpha+2\eta)} \eps^{2\eta}\;.
%\end{align*}
%
%\end{lemma}
%
%\begin{proof}
%Let $\delta \in [0,1]$. It follows from the assumptions on $f$ (see \cite[Lemma 3.4]{HM10} for details) that \ldots
%\end{proof}

First we obtain bounds on the difference between $Z$ and $Z_\eps$. The bounds \eqref{eq:Z-bound0} -- \eqref{eq:Z-bound2} will be proved using rough path techniques. For $0 < t < \eps^2$ we shall give a crude elementary bound in \eqref{eq:Z-bound-easy2}.

\begin{lemma}\label{lem:Z-bounds}
For $t > 0$ we have 
\begin{align}
\label{eq:Z-bound0}
|Z(t)|_{\cC^\alpha}&\lesssim  t^{-\alpha/2}\;.
\end{align}
Moreover, for $t > \eps^2$ we can write  
\begin{align}
%\label{eq:Z-bound1}
%|Z_\eps(t)|_{\cC^\alpha}&\lesssim  (t-\eps^2)^{-\alpha/2}\;,
%\\
 \label{eq:Z-bound2}
 Z(t) - Z_\eps(t) = T_1(t) + T_2(t) \;,
\end{align}   
where
\begin{align}
 |T_1(t)|_{\cC^{\tilde\alpha}} &\ls (t-\eps^2)^{-\tilde\alpha/2} \CE_\eps\;,\notag\\
%  |T_2(t)|_{\cC^{0}} 
%&  \ls  (t-\eps^2)^{-\tilde\alpha/2} \CE_\eps +
  |T_2(t)|_{\cC^{0}}  & \ls  (t-\eps^2)^{-\alpha/2}  \eps^{3 {\alpha} -1}\;.\label{eq:Zbound2A}
\end{align}
%
% |Z(t) - Z_\eps(t)|_{\cC^{\tilde\alpha}}&\lesssim 
% (t-\eps^2)^{-\tilde\alpha/2} \CE_\eps +
% (t-\eps^2)^{-\alpha/2}  \eps^{3 {\alpha} -1}\;.
% |Z(t) - Z_\eps(t)|_{\cC^{\tilde\alpha}}&\lesssim 
% (t-\eps^2)^{-\tilde\alpha/2} \CE_\eps +
% (t-\eps^2)^{-\alpha/2}  \eps^{3 {\alpha} -1}\;.
Finally, for all $0 < t \leq  \eps^2$ we have 
\begin{align}
\label{eq:Z-bound-easy2}
 |Z_\eps(t)|_{\cC^1}&\lesssim 
  \eps^{\alpha -1}\;.
\end{align}  
\end{lemma}

The estimate \eqref{eq:Z-bound-easy2} holds for all $t > 0$, but we shall only use it for $0 < t < \eps^2$.

\begin{proof}[Proof of Lemma~\ref{lem:Z-bounds}]
The estimate \eqref{eq:Z-bound-easy2}   follows directly from the definition and the assumptions involving $K$.

%\begin{align*}
%|Z_\eps(t)|_{\cC^\alpha}&\lesssim  \eps^{\alpha -1}\;.
%\end{align*}
%The estimate \eqref{eq:Z-bound-easy2} follows then from \eqref{eq:Z-bound0} (which we prove below) using the crude bound
%\begin{align*}
%     |Z(t) - Z_\eps(t)|_{\cC^\alpha} 
%      \leq
%     |Z(t)|_{\cC^\alpha}    + |Z_\eps(t)|_{\cC^\alpha}   
%     \lesssim t^{-\frac{\beta}{2}} + \eps^{\alpha-1}\;.
%\end{align*}
%

Before estimating the other quantities, we  observe that ${Y_\eps}(t,\cdot) := G\big( \ue(t,\cdot) \big)$ is a rough path controlled by $(X_\eps,\XX_\eps)$ with rough path derivative
\begin{equs}[e:422]
{Y_\eps}'(t,y) = DG(\ue(t,y) \big) u_\eps'(t,y)\;,
\end{equs}
and remainder 
\begin{equs}
 \mathcal{R}_{Y_\eps}(t;x,y) & =    DG\big(\ue(t,x)\big)  \, \Rue(t;x,y) \label{e:423} \\
 &+ \int_0^1 \Big[ DG\big( \lambda \ue(t,y) + (1-\lambda) \ue(t,x)   \big)  - DG\big(\ue(t,x) \big) \Big]\, 
\\& \qquad \times\big( \ue(t,y) - \ue(t,x)  \big) \, d\lambda\;.
\end{equs}
Recalling the boundedness assumptions from the beginning of this section and in particular \eqref{e:lamRembar}, we obtain for $t > \eps^2$,
\begin{equs}\label{eq:Y-eps}
|{Y_\eps}(t)|_{\cC^\alpha} \ls 1\;,
 \qquad |{Y_\eps}'(t)|_{\cC^\alpha} \ls 1\;, 
\qquad |\mathcal{R}_{Y_\eps}(t)|_{2\alpha} \ls (t-\eps^2)^{-\frac{\alpha}{2}}\;. 
 \end{equs}
The analogous statements
\begin{equs}\label{eq:Y-eps0}
|{Y}(t)|_{\cC^\alpha} \ls 1\;,
 \qquad |{Y}'(t)|_{\cC^\alpha} \ls 1\;, 
\qquad |\mathcal{R}_{Y}(t)|_{2\alpha} \ls t^{-\frac{\alpha}{2}}\;
 \end{equs}
for $0 < t < T$ hold as well. 
Moreover, we infer form \cite[Lemma 5.5]{HW10} that 
\begin{equation}\begin{aligned}
\label{eq:Y-eps-diff}
|  Y(t) - {Y_\eps}(t) |_{\cC^\alpha} &\ls |  u(t) - {u_\eps}(t) |_{\cC^\alpha}\;,\\
  |{Y}'(t) - {Y_\eps}'(t) |_{\cC^\alpha}& \ls 
 |{u}(t) - {u_\eps}(t) |_{\cC^\alpha}  + |{u}'(t) - {u_\eps}'(t) |_{\cC^\alpha} 
 \;, \\
 |  \mathcal{R}_Y(t) - \mathcal{R}_{Y_\eps}(t) |_{2\alpha} &\ls 
 |{u}(t) - {u_\eps}(t) |_{\cC^\alpha}  + |  \Ru(t) - \Rue(t) |_{2\alpha}\;. 
\end{aligned}\end{equation}

Let us now turn to the proofs of the estimates \eqref{eq:Z-bound0} and \eqref{eq:Z-bound2}.
The estimate \eqref{eq:Z-bound0} is a direct consequence of Lemma \ref{lem:Gub-int}, combined with the assumptions involving $K$ and the bounds \eqref{eq:Y-eps}.

In order to prove \eqref{eq:Z-bound2} we set
\begin{align*}
Q_\eps(t;x,y) &= 
\iint_x^y G(u_\eps(t,z)) \, d u_\eps(t,z)  
  - G(u_\eps(t,x)) \, \big( u_\eps(t,y) - u_\eps(t,x)\big) 
 \\& \qquad - DG(u_\eps(t,x))u_\eps'(t,x) \XX_\eps (t;x,y) u_\eps'(t,x)\;.
\end{align*} 
%and define $Q(t;x,y)$ accordingly. \comment{Never use this!}
Applying \eqref{eq:IB2} and using \eqref{eq:Y-eps} and the assumption that \[\max\{
\| X_\eps \|_{\cC^{{\alpha}}_T}, \, 
 \|\XX_\eps \|_{ \bB_T^{2{\alpha}}}, \, 
 \| u_\eps \|_{\cC^{{\alpha}}_T},  \,
 \| u_\eps' \|_{\cC^{{\alpha}}_{[\eps^2,T]}} \}
\leq K, \] we infer that
\begin{align}\label{eq:Q-bound}
 | Q_\eps(t) |_{{3\alpha}}
  \lesssim (t-\eps^2)^{-\alpha/2}\;.
\end{align}
%Applying Lemma \ref{lem:Gub-int} we obtain in view of \eqref{eq:Q-bound} that 
%\begin{align*}
% |Z_\eps(t)|_{C^\alpha}
%  \lesssim (t-\eps^2)^{-\beta/2}\;,
%\end{align*}
%which proves \eqref{eq:Z-bound1}. The proof of \eqref{eq:Z-bound0} is analogous.
For fixed $t \in (\eps^2,T]$, we may now write
\begin{align*}
 Z(t,x) - Z_\eps(t,x)
 &  =  \bigg(\iint_{-\pi}^x G(u(t,y)) \, d_y u(t,y)
     -\iint_{-\pi}^x G(u_\eps(t,y)) \, d_y u_\eps(t,y) \bigg)
 \\&     \quad+ \int_\R 
 \bigg(
  \iint_{-\pi}^{-\pi+\eps z} \frac{\eps z - \pi - y}{\eps} G(u_\eps(t,y))\,du_\eps(t,y)
 \\&\qquad\quad         + \iint_x^{x+\eps z} \frac{y -\eps z - x}{\eps}
     G(u_\eps(t,y))\,du_\eps(t,y)
      \bigg) \, \mu(dz) 
  \\&   \quad +
  \int_\R
\int_{-\pi}^x \frac{Q_\eps(t;y,y+ \eps z)}{\eps} \,dy \, \mu(dz)
  \\& =: (T_1 + S_1 + S_2)(t,x)\;.
\end{align*}
Here we used a Fubini-type Theorem for rough integrals (\cite[Lemma 2.10]{HW10}) to arrive at the expression for $S_2$. 

In order to bound $T_1$ we shall apply \cite[p. 102]{Gu04}, which provides a bound for the difference between two rough integrals. A slightly weaker result is provided in \cite[Lemma 2.9]{HW10}, but we  cannot apply this result directly here, as we need to be careful not to obtain products of terms which scale like $(t-\eps^2)^{-\tilde\alpha/2}$. 
Taking \eqref{eq:Y-eps} and \eqref{eq:Y-eps-diff} into account, we infer that
\begin{align*}
 |T_1(t)|_{\cC^{\tilde\alpha}} \lesssim (t-\eps^2)^{-\tilde\alpha/2} \CE_\eps\;.
\end{align*}
Moreover, it follows from \eqref{eq:Q-bound} that
\begin{align*}
 |S_3(t)|_{C^{1}}
  \lesssim 
  \int_{\R }  |z|^{3 \alpha}  \mu (dz) \, \eps^{3 {\alpha} -1} 
  \|Q_\eps\|_{\cC_T^{3{\alpha}}} \,  
  \lesssim 
      (t-\eps^2)^{-\alpha/2} \eps^{3 {\alpha} -1} \;.
\end{align*} 
%
%The bound for $T_2$ can be obtained by a calculation based on Lemma \ref{lem:Gub-int}, which shows that
%\begin{align*}
%  |T_2(t)|_{C^{\ta}}
%  \lesssim 
% \eps^{3 {\alpha}-1} (t-\eps^2)^{-\alpha/2}.
%\end{align*}
%
%--------

In order to bound $S_1$, we note that the second term $S_{12}$ of $S_1$ can be written as
\begin{align}\label{eq:T22}
S_{12}(t,x) 
 = \int_{\R}  \bigg[ \iint_x^{x+\eps z} Y_\eps^{t,z}(t,y)  \dd u_\eps(t,y) \bigg] \mu(dz)\;,
\end{align}
where 
\begin{align*}
Y_{\eps,z,x}(t,y) := \frac{y-\eps z - x}{\eps} G(u_\eps(t,y))\;.
\end{align*}
In view of the a priori bounds on $u_\eps$ and $u_\eps'$, it follows from \cite[Lemma 2.2]{Ha10} that $Y_{\eps,z}(t,\cdot)$ is controlled by $X_\eps(t,\cdot)$ with rough path derivative
\begin{align*}
 Y_{\eps,z,x}'(t,y) := \frac{y-\eps z - x}{\eps} DG(u_\eps(t,y)) u_\eps'(t,y) \;.
\end{align*}
Moreover, since $y \in [x,x+\eps z]$, the same result implies that
\begin{equation}\begin{aligned}
\label{eq:Y-bounds}
 | Y_{\eps,z,x}(t,\cdot)|_{\cC^0} +  | Y_{\eps,z,x}'(t,\cdot)|_{\cC^0}  &  \lesssim |z|\;,\\
 | Y_{\eps,z,x}'(t,\cdot)|_{\cC^\alpha}  &  \lesssim |z| + \eps^{-1}\;,\\
 | \mathcal{R}_{Y_{\eps,z,x}}(t;\cdot,\cdot)|_{2 \alpha}  & 
  \lesssim   |z| + \eps^{-1}(t-\eps^2)^{-\alpha/2}\;.
\end{aligned}\end{equation}%
Lemma \ref{lem:Gub-int} allows us to write the rough integral in \eqref{eq:T22} as
\begin{align*}
 &\iint_{x}^{x+\eps z} Y_{\eps,z,x}(t,y) \dd u_\eps(t,y) 
   = Y_{\eps,z,x}(t,x)   \big( u_\eps(t,x + \eps z)-u_\eps(t,x)\big)
  \\& \qquad\qquad\qquad +  Y_{\eps,z,x}'(t,x) \XX_\eps (t;x,x + \eps z) u_\eps'(t,x) +Q_{\eps}(t;x,x+\eps z)\;, 
\end{align*}
where
\begin{align*}
|Q_{\eps}(t;\cdot,\cdot)|_{3 \alpha} \ls (|z| + \eps^{-1}) (t-\eps^2)^{-\alpha/2}\;.
\end{align*}
Taking the a priori bounds on $ |u_\eps|_{\cC^\alpha}$, $ |u_\eps'|_{\cC^0}$ and $| \XX_\eps(t;\cdot,\cdot)|_{2\alpha}$ into account, it thus follows that 
\begin{align*}
&\bigg| \iint_{\cdot}^{\cdot+\eps z} Y_{\eps,z,x}(t,y) \dd u_\eps(t,y)
\bigg|_{\cC^0}
\\&\qquad\lesssim | Y_{\eps,z,x}(t,\cdot)|_{\cC^0}| \eps z|^\alpha 
 + | Y_{\eps,z,x}'(t,\cdot)|_{\cC^0}
 |\eps z|^{2\alpha}
+ |Q_{\eps}(t;\cdot,\cdot)|_{3 \alpha}  |\eps z|^{3\alpha}
\\&\qquad \lesssim |z| \big( |\eps z|^\alpha 
   + |\eps z|^{2\alpha}\big)
   + (|z| + \eps^{-1}) (t-\eps^2)^{-\alpha/2} |\eps z|^{3 \alpha}\;.
\end{align*}
Taking into account that the $(1+3 \alpha)$-moment of $|\mu|$ is finite (this is the only place where we all moments up to order $\frac{5}{2}$) and using that $\eps^\alpha \leq \eps^{3\alpha-1}$, we infer that
\begin{align*}
|S_{12}(t,\cdot)|_{\cC^0} 
\lesssim \eps^{3 \alpha - 1}(t-\eps^2)^{-\alpha/2}\;.
\end{align*}
Since an analogous argument yields the same estimate for the first term $S_{11} := S_1 - S_{12}$, we infer that
\begin{align*}
|S_{1}(t,\cdot)|_{\cC^0} 
\lesssim \eps^{3 \alpha - 1}(t-\eps^2)^{-\alpha/2}\;.
\end{align*}
Hence, setting $T_2:= S_1 + S_2$ we arrive at the desired conclusion.
\end{proof}

\begin{proof}[Proof of Proposition \ref{prop:Xi-bounds}]
%For $0 < t \leq \eps^2$ we use the estimates~\eqref{eq:Z-bound0} and~\eqref{eq:Z-bound-easy2}. 
%\begin{align*}
%|Z(t)|_{\cC^\alpha}\lesssim  t^{-\beta/2}\;,\qquad
%|\partial_x Z_\eps(t)|_{0} \lesssim  \eps^{\alpha-1}\;,
%\end{align*}
%The first of these bounds has been proved in Lemma \ref{lem:Z-bounds} and the second bound follows easily using the assumptions involving the constant $K$.
%We shall use the crude bound
%\begin{align*}
%& |\Xi(t) - \Xi_\eps(t)|_{\cC^\gamma}
%\\ & \leq 
%    |\Xi(t) |_{\cC^\gamma}  + | \Xi_\eps(t)|_{\cC^\gamma}
%\\& \leq
% \int_0^t |S(t-s)|_{\cC^\alpha \to \cC^{1+\gamma}}
%    | Z(s)|_{\cC^\alpha}\; ds
%+
% \int_0^t |S_\eps(t-s)|_{\cC^0 \to \cC^\gamma}
%    |\partial_x Z_\eps(s)|_{\cC^0}\; ds
%\\& \lesssim
% \int_0^t (t-s)^{-\frac12(1+\gamma-\alpha)}
%     s^{-\beta/2}\; ds
%+
% \int_0^t (t-s)^{-\frac{\gamma}{2}-\kappa}
%    \eps^{\alpha-1}\; ds
%\\&\lesssim t^{\frac12(1+\alpha-\gamma-\beta)} 
%	+ \eps^{\alpha-1} t^{1-\frac{\gamma}{2} - \kappa}
%%\\& \lesssim \eps^{1+\alpha-\gamma-\beta}
%%            + \eps^{1+\alpha-\gamma-2\kappa}\;,
%\\& \lesssim \eps^{1+\alpha-\gamma-\beta}\;,
%\end{align*}
%where we used Lemma \ref{le:newOp} and the fact that $t \leq 2\eps^2$.
%
We start by splitting the integral from $0$ to $t$ into three parts
\begin{align*}
 \Xi(t) - \Xi_\eps(t)
 & \leq \int_0^{\eps^2\wedge t}
    \partial_x (S(t-s) Z(s) - S_\eps(t-s) Z_\eps(s) )\; ds  
 \\&\qquad     +  \int_{\eps^2 \wedge t}^t
         \partial_x (S(t-s) - S_\eps(t-s)) Z(s) \; ds
 \\&\qquad  +  \int_{\eps^2 \wedge t}^t
         \partial_x S_\eps(t-s) ( Z(s) - Z_\eps(s) ) \; ds
 \\& =: I_1 + I_2 + I_3\;.         
\end{align*}
For the first term we use the estimates~\eqref{eq:Z-bound0} and~\eqref{eq:Z-bound-easy2} to obtain for any $\kappa>0$ small enough
\begin{align*}
 |I_1|_{\cC^\gamma}
& \leq 
 \int_0^{\eps^2 \wedge t} |S(t-s)|_{\cC^\alpha \to \cC^{1+\gamma}}
    | Z(s)|_{\cC^\alpha}\; ds
\\& \qquad+
 \int_0^{\eps^2 \wedge t} |S_\eps(t-s)|_{\cN \to \cC^\gamma}
    |\partial_x Z_\eps(s)|_{\cN}\; ds
\\& \lesssim
 \int_0^{\eps^2 \wedge t} (\eps^2 \wedge t -s)^{-\frac12(1+\gamma-\alpha)}
     s^{-\alpha/2}\; ds
\\& \qquad+
 \int_0^{\eps^2 \wedge t} (\eps^2 \wedge t-s)^{-\frac{\gamma}{2}-\kappa}
    \eps^{\alpha-1}\; ds
%\\& \lesssim
% \int_0^{\eps^2} \eps^{-1-\gamma+\alpha}
%     s^{-\alpha/2}\; ds
%+
% \int_0^{\eps^2} \eps^{-\gamma-2\kappa}
%    \eps^{\alpha-1}\; ds
\\& \lesssim \eps^{1-\gamma}
            + \eps^{1+\alpha-\gamma-2\kappa}
\\& \lesssim \eps^{1-\gamma}\;.
\end{align*}
The terms $I_2$ and $I_3$ are equal to $0$ if $t \leq \eps^2$. So from now on we assume $t >\eps^2$. To bound the second part we use Corollary \ref{cor:S-Seps} and Lemma \ref{lem:Z-bounds} to obtain for any $0 \leq \lambda<1  -\gamma  $ and for $\kappa >0$ small enough
\begin{align*}
|I_2|_{\cC^\gamma}
 &\lesssim \int_{\eps^2}^t
         | S(t-s) - S_\eps(t-s)|_{\cC^\alpha \to \cC^{1+\gamma}}
          | Z(s) |_{\cC^\alpha} \; ds
 \\&\lesssim
 \int_{\eps^2}^t
        (t-s)^{-\frac12(1+\gamma -\alpha +\lambda +\kappa)}
        \eps^{\lambda}
          s^{-\frac{\alpha}{2}}\; ds
 \lesssim
        t^{\frac12(1 -\gamma -\lambda -\kappa)}
        \eps^{\lambda}\;.
\end{align*}
%Actually, in this term we could get an even better rate in $\eps$ by allowing for a blowup like $t^{-\frac{\alpha}{2}}$. But the rate we obtain from $I_1$ is not this good, so that we drop this slight improvement.
%
%
The third part can be estimated using Lemma \ref{le:newOp} and Lemma \ref{lem:Z-bounds} by 
\begin{align*}
|I_3|_{\cC^\gamma}
 &\lesssim 
  \int_{\eps^2}^t
     | S_\eps(t-s) |_{\cC^{\tilde\alpha} \to \cC^{1+\gamma}} 
       | T_1 |_{\cC^{\tilde\alpha}} \; ds
       + \int_{\eps^2}^t
     | S_\eps(t-s) |_{\cC^{0} \to \cC^{1+\gamma}} 
       | T_2 |_{\cC^{0}} \; ds
\\ &\lesssim 
  \int_{\eps^2}^t
 \Big(     (t-s)^{-\frac12(1+\gamma - \tilde\alpha + \kappa)} 
  (s-\eps^2)^{-\tilde\alpha/2}  
\CE_\eps\\
& \qquad + (t-s)^{-\frac12(1+\gamma  + \kappa)} 
   (s-\eps^2)^{-\alpha/2}  \eps^{3 \alpha -1} 
 \Big)
  \; ds
%\\ &\lesssim 
%  \big( 
%\CE_\eps
%+ 
%  \eps^{3 \alpha -1} 
% \big)
% (t-\eps^2)^{\frac12(1-\gamma-\tilde\alpha+\tilde\alpha-\kappa)}
\\ &\lesssim 
\CE_\eps  \, (t-\eps^2)^{\frac12(1-\gamma-\kappa)}
+ 
  \eps^{3 \alpha -1} 
 \,
 (t-\eps^2)^{\frac12(1-\gamma-\alpha-\kappa)}\;.
\end{align*}
Note that within the range of parameters that we consider the integral in the second line always converges (for $\kappa >0$ small enough), but it may happen that the last exponent of $(t-\eps)^2$ in the last line is negative. 
\end{proof}

\begin{proof}[Proof of Proposition \ref{prop:Xi-time}]
Let $ \eps^2 \leq s < t$. We need to bound the expression $J(t) - J(s)$ where
\begin{align*}
 J(t) := 
  \int_0^t \partial_x\big(\delta_\eps \big( S(t-r)  Z(r) \big)  \big) \; dr. 
\end{align*}
Here and below we will use the notation $\delta_\eps \big( S Z) := S_\eps Z_\eps - S Z$, $\delta_\eps S:= S_\eps-S$  and so forth. We write  $J(t) -J(s) = I_1(s,t) +I_2(s,t) +I_3(s,t)$, where
\begin{align*}
I_1(s,t):=& \int_0^{\eps^2} \partial_x \delta_\eps  \Big[  \big( S(t-r)  -S(s-r) \big) Z(r)   \Big] \; \; dr , \\
I_2(s,t):= &\int_{\eps^2}^{s} \partial_x \delta_\eps  \Big[  \big( S(t-r)  -S(s-r) \big) Z(r)   \Big] \; \; dr , \\
I_3(s,t) :=& \int_s^{t} \partial_x \delta_\eps  \big[   S(t-r)   Z(r)   \big] \; \; dr.
\end{align*}
In order to bound $I_1$ we use the bounds~\eqref{eq:Z-bound0} and ~\eqref{eq:Z-bound-easy2} and brutally bound the $\eps$-difference by the sum. Then using Lemma~\ref{lem:timecontS} we get
\begin{align*}
 |I_1(s,t)|_{\cN}
&  \leq \int_0^{\eps^2} | S(t-s) - \Id |_{\cC^{1+\gamma}\to \cC^1} 
 |S(s-r)|_{\cC^\alpha \to \cC^{1+\gamma}}
    |Z(r)|_{\cC^\alpha} \;dr
\\&  \qquad  + \int_0^{\eps^2} 
| S_\eps(t-r-\tfrac{s-r}{2}) - S_\eps(s-r-\tfrac{s-r}{2}) |_{\cC^{1+\gamma+\kappa}\to \cC^1} 
 \\& \qquad\qquad \times|S_\eps(\tfrac{s-r}{2})|_{\cC^1 \to \cC^{1+\gamma+\kappa}}
    |Z_\eps(r)|_{\cC^1} \;dr
\\& \lesssim (t-s)^{\frac{\gamma}{2}} \eps^{1- \gamma }\;.
\end{align*}
Actually, it is obvious that the first integral is bounded by the right hand side. For the second integral one even gets a better scaling $\eps^{1+ \alpha - \gamma - \kappa}$ for $\kappa>0$ arbitrarily small. 

For the term $I_2$ it is useful to split it up once more. We write
\begin{align*}
I_2(s,t) &= \int_{\eps^2}^{s} \partial_x  \Big[  \delta_\eps  \big( S(t-r)  -S(s-r) \big)  \Big] Z(r)    \; dr \\
& \qquad + \int_{\eps^2}^{s} \partial_x   \big( S_\eps(t-r)  -S_\eps(s-r) \big)  \delta_\eps Z(r)   \; dr =:  I_{2,1}(s,t)+I_{2,2}(s,t).
\end{align*}
To bound $I_{2,1}$ we use the factorisation 
\begin{equs}
\delta_\eps  \big( S(t-r)  -S(s-r)\big)  &=   \big( S_\eps(t-s+\tfrac{s-r}{2}) - S_\eps(\tfrac{s-r}{2}) \big) \,   \delta_\eps S(\tfrac{s-r}{2})  \\
&\qquad + \delta_\eps  \big( S(t-s+\tfrac{s-r}{2}) - S(\tfrac{s-r}{2}) \big) \,    S(\tfrac{s-r}{2}) .
\end{equs}
Then  Lemma \ref{lem:timecontS} as well as Corollary \ref{cor:S-Seps}  yield for any $\kappa>0$ small enough
\begin{align*}
 |I_{2,1}(s,t)|_{\cN}
% &  \leq \int_{\eps^2}^{s} | S(\tfrac{s-r}{2}) - S_\eps(\tfrac{s-r}{2}) |_{\cC^{\alpha} \to \cC^{1+\gamma+\kappa}}
%\\&\qquad\qquad\times | S(t-s+\tfrac{s-r}{2}) - S(\tfrac{s-r}{2}) |_{\cC^{1+\gamma+\kappa} \to \cC^{1} }
%    |Z(r)|_{\cC^\alpha} \;dr
%\\& \qquad + 
%\int_{\eps^2}^{s} | S_\eps(\tfrac{s-r}{2}) |_{\cC^{\alpha} \to \cC^{1 + \gamma + \lambda +\kappa}}
%\\&\qquad\qquad\times | \delta_\eps \big( S(t-s+\tfrac{s-r}{2}) - S(\tfrac{s-r}{2})\big)  |_{\cC^{1 + \gamma +\lambda +\kappa} \to \cC^{1} }
%    |Z(r)|_{\cC^\alpha} \;dr
%%\\& \qquad + 
%%\int_{\eps^2}^{s} | S(s-r-\tfrac{s-r}{2}) |_{\cC^{\alpha} \to \cC^{1+\gamma+\kappa}}
%%\\&\qquad\qquad\times | S(\tfrac{s-r}{2}) - S_\eps(\tfrac{s-r}{2}) |_{\cC^{1+\gamma+\kappa} \to \cC^1 }
%%    |Z(r)|_{\cC^\alpha} \;dr
  \lesssim (t-s)^{\frac{\gamma}{2}} \eps^{1 - \gamma -\kappa}\;.
\end{align*}

For $I_{2,2}$ we use the decomposition $Z-Z_\eps = T_1 +T_2$. For the term involving $T_1$ we get using Lemma \ref{lem:timecontS} for any $\kappa >0$ small enough
\begin{align*}
  \int_{\eps^2}^{s} 
  | &S_\eps(t-s + \tfrac{s-r}{2}) - S_\eps(\tfrac{s-r}{2}) |_{\cC^{1+ \gamma + \kappa} \to \cC^{1}}
 |S_\eps(\tfrac{s-r}{2})|_{\cC^{\tilde\alpha} \to \cC^{1 + \gamma + \kappa}}
    |T_1(r)|_{\cC^{\tilde\alpha}} \;dr
\\    &  \lesssim 
    (t-s)^{\frac{\gamma}{2}}
    \int_{\eps^2}^{s} 
    (s-r)^{-\frac12(1+ 2 \gamma -\tilde\alpha +\kappa)}
   \;  (r-\eps^2)^{-\frac{\tilde\alpha}{2}}\CE_\eps 
  \; dr
\\& \lesssim (t-s)^{\frac{\gamma}{2}}
     t^{\frac12(1 -2 \gamma   -2\kappa)} \CE_\eps 
\;.
\end{align*}
Note that we have used the brutal bound $(1 + \eps^{\gamma} \,(s-r)^{-\gamma/2} ) \ls (s-r)^{-\gamma/2}$.  For the term involving $T_2$ we cannot be quite so brutal and we write
\begin{equs}
\Big| \int_{\eps^2}^s &\partial_x \big( S_\eps(t-r) -S_\eps(s-r) \big) \,T_2 (r)\, dr \Big|_{\cC^0}\\
&\leq \Big| \int_{\eps^2}^{s-\eps^2 \vee \eps^2} \partial_x  \big( S_\eps(t-r) -S_\eps(s-r) \big) \, T_2 (r) \, dr \Big|_{\cC^0}   \\
& \qquad + \Big| \int_{s-\eps^2 \vee \eps^2}^s \partial_x  \big( S_\eps(t-r) -S_\eps(s-r) \big) \,T_2 (r)\, dr \Big|_{\cC^0} .
\end{equs}
The first integral is bounded by 
\begin{equs}
\int_{\eps^2}^{s-\eps^2 \vee \eps^2} \big| &S_\eps(t-s +\eps^2) -S_\eps(\eps^2)  \big|_{\cC^{1 + 
\gamma + \kappa }  \to \cC^1} \, \\ 
& \qquad \times \big|S_\eps(s-\eps^2 -r) \big|_{\cC^0 \to \cC^{1 + \gamma + \kappa}}   \, \big|T_2 (r)\big|_{\cC^0} dr\\
& \ls (t-s)^{\frac{\gamma}{2}}  \eps^{3 \alpha -1}  \int_{\eps^2}^{(s-\eps^2) \vee \eps^2}   \big( s-\eps^2  -r\big)^{- \frac{1 + \gamma +2\kappa}{2}}(r-\eps^2)^{-\frac{\alpha}{2}}  dr\\
&\ls (t-s)^{\frac{\gamma}{2}}  \eps^{3 \alpha -1} .
\end{equs}
For the second integral we get in the same way 
\begin{equs}
\int_{s-\eps^2 \vee \eps^2}^s \big| &S_\eps\big(t-s +\tfrac{s-r}{2}\big) -S_\eps\big(\tfrac{s-r}{2}\big)  \big|_{\cC^{1 + \gamma + \kappa} \to \cC^1} \,  \\
& \qquad  \times \, \Big|S_\eps\big(\tfrac{s-r}{2} \big) \big|_{\cC^0 \to \cC^{1 + \gamma + \kappa}}   \, \big|T_2 (r)\big|_{\cC^0} dr\\
& \ls (t-s)^{\frac{\gamma}{2}}  \eps^{3 \alpha -1}   \int_{s-\eps^2 \vee \eps^2}^s   \eps^{\gamma}( s-r )^{-\frac{\gamma}{2}} ( s-r )^{-\frac{1 + \gamma+2\kappa}{2}}   (r-\eps^2)^{-\frac{\alpha}{2}}  dr\\
&\ls (t-s)^{\frac{\gamma}{2}}  \eps^{3 \alpha -1} \eps^{1 - \gamma  -\alpha -2\kappa}.
\end{equs}
Summarising these calculations and redefining $\kappa$, we obtain the bound 
\begin{equ}
I_{2,2}(s,t) \leq (t-s)^{\frac{\gamma}{2}} \Big( t^{\frac12(1 -2 \gamma   -\kappa)} \CE_\eps  +  \eps^{3 \alpha -1} \Big).
\end{equ}

For $I_3$ we use the same splitting
\begin{align*}
I_3(s,t) &= \int_{s}^{t} \partial_x  \big[  \delta_\eps  S(t-r)   \big] Z(r)\,dr + \int_{s}^{t} \partial_x  \big( S_\eps(t-r)   \delta_\eps Z(r)   \big) \, dr\\
& =:  I_{3,1}(s,t)+I_{3,2}(s,t).
\end{align*}
The term $I_{3,1}(s,t)$ can be estimated easily using Corollary \ref{cor:S-Seps}:
\begin{align*}
 |I_{3,1}(s,t)|_{\cN}
&  \leq \int_s^t  | S(t-r) - S_\eps(t-r) |_{\cC^\alpha \to \cC^1}
    |Z(r)|_{\cC^\alpha} \;dr
\\& \lesssim (t-s)^{\frac{\gamma}{2}} \eps^{1 - \gamma  - \kappa}\;.
\end{align*}
Finally, for $I_{3,2}$ we get using \eqref{eq:Z-bound2} once more that 
\begin{align*}
 |I_{3,2}(s,t)|_{\cN}
%&  \leq \int_{s}^t |S_\eps(t-r)|_{\cC^{{\tilde\alpha}} \to \cC^1}
%    |Z(r) - Z_\eps(r)|_{\cC^{\tilde\alpha}} \;dr
%\\
    &  \lesssim 
 \int_s^t   (t-r)^{-\frac12(1-{\tilde\alpha} +\kappa)  }
    (r-\eps^2)^{-\frac{\tilde\alpha}{2}}  \CE_\eps 
\\& \qquad +  (t-r)^{-\frac12(1 +\kappa)  }
(r-\eps^2)^{-\frac{\alpha}{2}} \eps^{3 \alpha -1}  \dd r
%\\    &  \lesssim 
% \int_s^t   (t-r)^{-\frac12(1-{\tilde\alpha} +\kappa)  }
%(     (r-s)^{-\frac{\tilde\alpha}{2}} \CE_\eps + 
%    (r-s)^{-\frac{\alpha}{2}} \eps^{3 \alpha -1} ) \dd r
%\\& \lesssim (t-s)^{\frac12(1  - \kappa)}
%    \CE_\eps + 
%   (t-s)^{ \frac12(1 + {\tilde\alpha} - \alpha - \kappa)} 
%    \eps^{3 \alpha -1} \;
\\& \lesssim (t-s)^{\frac{\gamma}{2}}
    \Big( t^{\frac12(1 -\gamma - \kappa) }   \CE_\eps + \eps^{3 {\alpha} -1}\Big) \;.
\end{align*}
Here as above we have absorbed the positive power of $t$ in front of $\eps^{3 {\alpha} -1}$ into the implicit constant (that may depend on $T$) because we do not need it later on. This finishes the argument.
\end{proof}

\section{Bounds on the approximated semigroup}\label{sec:HarmAn}
Throughout the paper, we frequently need bounds on the approximated heat semigroup $S_\eps$. These calculations are collected in this section. 
We first give some auxiliary calculations involving the approximated heat kernels that are needed in the proof of Lemma \ref{le:mistakecorrected}.
\begin{lemma}\label{le:pepscalc}
The following bounds hold:
\begin{enumerate}
\item[(i)] For any $ 0< \gamma <1$ and for any $t \in [0,T]$ we have the bound
\begin{equ}[e:pintbo1]
\int_{-\pi}^\pi   \big| (p_{t}^\eps)'( z) \, \big|  \,  |z|^\gamma \, dz \ls  t^{\frac{-1+\gamma}{2}}.
\end{equ}
If $\gamma =0$ we have
\begin{equ}[e:pintbo2]
\int_{-\pi}^\pi   \big| (p_{t}^\eps)'( z) \, \big|  \,  \, dz \ls t^{ -\frac{1}{2}}|\log t|.
\end{equ}
\item[(ii)] For any $0<\alpha< \frac{1}{2}$ we have for any $0 < t \leq T$ and any $x \in [-\pi,\pi]$ 
\begin{equ}[e:pintbo3]
\int_0^t \int_{-\pi}^\pi \big( p_s^\eps(z) - p_s^\eps(z-x)\big)^2 \big| z\big|^{2 \alpha} \, dz \, ds \ls |x| t^{\alpha}.
\end{equ} 
\end{enumerate}
\end{lemma}

\begin{remark}
All of these bounds scale in the optimal way, except for \eqref{e:pintbo2} where an additional $|\log t|$ appears. When applying this bound in the proof of Lemma \ref{le:mistakecorrected} this small correction does not matter . For \eqref{e:pintbo3} our proof also shows the bound
\begin{equ}[e:pintbo4]
\int_0^t \int_{-\pi}^\pi \big( p_s^\eps(z) - p_s^\eps(z-x)\big)^2 \big| z\big|^{2 \alpha} \, dy \, ds \ls |x|^{1 +2\gamma} ,
\end{equ} 
with a uniform constant for $t \in [0,T]$. We state \eqref{e:pintbo3} in this way because it is convenient in the proof of Lemma \ref{le:mistakecorrected}
\end{remark}
\begin{proof}
(i) We start the calculation by deriving pointwise bounds on 
\begin{equ}
\big( p^\eps_{t-s}\big)' (z)= \frac{1}{\sqrt{2\pi}} \sum_{k  \in \Z} ik e^{-k^2 f(\eps k) t}e^{ikz } .
\end{equ}
For $ |z|^2 \leq t$ we will simply use the brutal bound
\begin{equs}[e:casA]
\big| \big( p^\eps_{t}\big)' (z) \big|  \ls \sum_{k  \in \Z} |k| e^{-k^2 f(\eps k) t} \ls \frac{1}{t},
\end{equs}
which holds uniformly for $t \leq T$. 

Else, for $|z|^2>t $  we perform a summation by part and obtain
\begin{equs}[e:pca1]
\frac{1}{\sqrt{2\pi}}&  \sum_{k  \in \Z} ik e^{-k^2 f(\eps k) t}e^{ikz }  \\
&= \frac{1}{\sqrt{2\pi}} \sum_{k  \in \Z}  i\Big( k e^{-k^2 f(\eps k)t} -(k-1) e^{-(k-1)^2 f(\eps (k-1))t} \Big) g_{k}(z),
\end{equs}
where
\begin{equs}[e:pca2]
 \big|g_{k}(z) \big|  :=  \Big| \sum_{j=0}^k e^{ijz} \Big| =  \Big| \frac{1 - e^{i(k+1)z}}{1 - e^{iz}} \Big| \ls \frac{1}{|z|},
\end{equs}
the last two expressions being valid for $z \neq 0$. In order to bound the sum over the increments of $k e^{-k^2 f(\eps k)t}$  we write
\begin{equs}[e:brut2]
\Big|& \sum_{k  \in \Z}  \Big( k e^{-k^2 f(\eps k)t} -(k-1) e^{-(k-1)^2 f(\eps (k-1))t}  \Big| \\
&\leq \sum_{k  \in \Z}  \Big|  e^{-k^2 f(\eps k)t}   \Big| +  \sum_{k  \in \Z} |k| \,  \Big|  e^{-k^2 f(\eps k)t} - e^{-(k-1)^2 f(\eps (k-1))t}  \Big| .
\end{equs}
The first summand can be bounded easily
\begin{equ}
\sum_{k  \in \Z}  \Big|  e^{-k^2 f(\eps k)t}   \Big|  \ls t^{-\frac{1}{2}},
\end{equ}
where again the implicit constant is uniform for $t \leq T$. 
For the second summand we use the brutal bound
\begin{equs}[e:brut]
\big| t^{\frac12} k \big| & \ls \Big( e^{- k^2  f(\eps k)t/2} + e^{- (k-1)^2  f(\eps(k-1))t/2} \Big)^{-1} .
\end{equs}
Actually,  \eref{e:brut}  is obvious for $k=0$ and for $k=1$ it states that 
\begin{equ}
t^{\frac12} \ls  \Big(   e^{- f(\eps )t/2} + 1   \Big)^{-1} ,
\end{equ}
which is  true for $t \leq T$. For $k\neq 0,1$, we note that 
\begin{equ}
\frac{1}{4} \leq \frac{k^2}{(k-1)^2} \leq 4,
\end{equ}
to bound the right hand side of \eref{e:brut} by 
\begin{equs}
\Big( e^{- k^2  f(\eps k)t\frac{t}{2}} + e^{- (k-1)^2  f(\eps(k-1))\frac{t}{2}} \Big)^{-1}  &\geq \Big( e^{- k^2  c_f \frac{t}{2}} + e^{- \frac14 k^2  c_f \frac{t}{2}} \Big)^{-1}\\
& \geq \frac{1}{2} e^{ \frac18 k^2  c_f t}.
\end{equs}
(Recall the definition of $c_f$ in Assumption \ref{a:f}).  This establishes \eqref{e:brut}.

Plugging \eqref{e:brut} into the second term on the right hand side of \eqref{e:brut2} we get
\begin{equs}[e:pca3]
 \sum_{k  \in \Z} &|k| \,  \Big|  e^{-k^2 f(\eps k)t} - e^{-(k-1)^2 f(\eps (k-1))t}  \Big| \\
 & \leq  t^{-\frac12}\sum_{k  \in \Z} \,  \Big|  e^{-k^2 f(\eps k)t/2} - e^{-(k-1)^2 f(\eps (k-1))t/2}  \Big| \ls t^{ -\frac12},
\end{equs}
where we have made use of the $\BV$ boundedness  from Assumption \ref{a:BV}.

Hence, summarising \eqref{e:pca1}--\eqref{e:pca3} we obtain 
\begin{equs}
\big| \big( p^\eps_{t}\big)' (z) \big| \ls  t^{-\frac12} |z|^{-1}.
\end{equs}
Finally, we can conclude the desired bounds \eqref{e:pintbo1} for $\gamma>0$
\begin{equs}
\int_{-\pi}^\pi   \big| (p_{t}^\eps)'( z) \big| |z|^\gamma \, dz
 & \leq t^{\frac{\gamma}{2}}\int_{|z| \leq t^{\frac{1}{2}}}   \big| (p_{t}^\eps)'( z) \big| \, dz + \int_{|z| > t^{\frac{1}{2}}}   \big| (p_{t}^\eps)'( z) \big| |z|^\gamma \, dz\\
&\ls t^{\frac{\gamma-1}{2} } + 2 \int_{t^{\frac12}}^\pi t^{-\frac12} |z|^{\gamma -1} \, dz \ls t^{\frac{\gamma-1}{2} } ,\label{e:pconc1}
\end{equs}
and similarly for $\gamma=0$ we obtain \eqref{e:pintbo2}:
\begin{equs}
\int_{-\pi}^\pi   \big| (p_{t}^\eps)'( z) \big| \, dz &\leq \int_{|z| \leq t^{\frac{1}{2}}}   \big| (p_{t}^\eps)'( z) \big| \, dz + \int_{|z| > t^{\frac{1}{2}}}   \big| (p_{t}^\eps)'( z) \big| \, dz\\
&\leq t^{-1} t^{\frac12}  + 2 \int_{t^{\frac12}}^\pi t^{ -\frac12} |z|^{-1} \, dz \ls t^{ -\frac12} \big| \log t \big|.\label{e:pconc2}
\end{equs}
(ii) In order to prove \eqref{e:pintbo3} we again have to get pointwise bounds bounds on the integrand. We distinguish between two different cases: the case where $|x| > s^{\frac12}$ and the case where $|x| \leq s^{\frac12}$.

Let us start by the first case, i.e. let us assume that $|x| > s^{\frac12}$. In that case $p_s^\eps(z-x)$ is not a good approximation of $p_s^\eps(z)$ and we bound the difference by the sum. We write
\begin{equs}[e:pepsbo1]
\int_{-\pi}^\pi & \big( p_s^\eps(z) - p_s^\eps(z-x)\big)^2 \big| z\big|^{2 \alpha} \, dz \\
& \ls  \int_{-\pi}^\pi  p_s^\eps(z)^2   \big| z\big|^{2 \alpha} \, dz  + |x|^{2 \alpha}\int_{-\pi}^\pi  p_s^\eps(z)^2  \, dz .
\end{equs}
The derivation of bounds for these two integrals is similar to the calculation for \eqref{e:pintbo1} and \eqref{e:pintbo2}. In the same way as in \eqref{e:casA} we get for $|z| \leq s^{\frac12}$ that
\begin{equs}[e:casAnew]
  p^\eps_{s} (z)^2   \leq  \bigg( \sum_{k  \in \Z}  e^{-k^2 f(\eps k) s} \bigg)^2 \ls s^{-1}.
\end{equs}
Then performing the same summation by part as in \eqref{e:pca1} we get for $|z| > s^{\frac12}$
\begin{equ}[e:pca1new]
  p^\eps_{s} (z)^2  \ls  \bigg( \sum_{k  \in \Z}  \Big(  e^{-k^2 f(\eps k)s} - e^{-(k-1)^2 f(\eps (k-1))s} \Big) |z|^{-1}\bigg)^2
 \ls |z|^{-2}.
\end{equ}
Hence we get as in \eqref{e:pconc1} and \eqref{e:pconc2} that
\begin{equs}[e:pconc1new]
\int_{-\pi}^\pi &   (p_{s}^\eps)'( z)^2  |z|^{2\alpha} \, dz\ls s^{\alpha -\frac{1}{2} } +  \int_{s^{\frac12}}^\pi  |z|^{2\alpha -2} \, dz \ls s^{\alpha -\frac{1}{2} } .
\end{equs}
and 
\begin{equs}[e:pconc2new]
 \int_{-\pi}^\pi &   (p_{s}^\eps)'( z)^2  \, dz \ls  |x|^{2\alpha}  \Big(s^{-\frac{1}{2} } +  \int_{s^{\frac12}}^\pi  |z|^{ -2} \, dz \Big)  \ls |x|^{2\alpha} s^{-\frac{1}{2} } .
\end{equs}
Now let us treat the second case, where $|x|\leq s^{\frac12}$. In that case we write 
\begin{equs}
 p_s^\eps(z) - p_s^\eps(z-x) = \frac{1}{\sqrt{2\pi}} \sum_{k \in \Z^\star} e^{-k^2   f(\eps k) s } e^{ikz} \big(1 - e^{ikx} \big).
 \end{equs}
 As before we use a brutal bound for $|z| \leq s^{\frac12}$ 
\begin{equs}
 \Big( p_s^\eps(z) - p_s^\eps(z-x) \Big)^2 \ls \bigg( \sum_{k \in \Z^\star} e^{-k^2   f(\eps k) s }   |kx| \bigg)^2 \ls |x|^2 s^{-2},
 \end{equs}
 and a summation by part if $|z| >s^{\frac12}$. As above in \eqref{e:pca1} and \eqref{e:pca2} we get
\begin{equs}
\Big|\sum_{k \in \Z^\star} & e^{-k^2   f(\eps k) s } e^{ikz} \big(1 - e^{ikx} \big) \Big|\\
& \ls |z|^{-1} \Big|\sum_{k \in \Z^\star}  e^{-k^2   f(\eps k) s }   \Big(1 - e^{ikx}  \Big) - e^{-(k-1)^2  f(\eps (k-1)) s } \Big(1 - e^{i(k-1)x}    \Big) \Big|\\
& \ls |z|^{-1} \sum_{k \in \Z^\star} \Big|  e^{-k^2   f(\eps k) s } - e^{-(k-1)^2  f(\eps (k-1)) s } \Big| \, \Big|1 - e^{ikx}  \Big|  \\
&\qquad +  |z|^{-1} \sum_{k \in \Z^\star}  e^{-k^2 \,  f(\eps (k-1))\, s } \, \Big| e^{ikx} - e^{i(k-1)x}    \Big|.
\end{equs}
The first sum can be bounded as above in \eqref{e:pca3}
\begin{equs}
 \sum_{k \in \Z^\star} \Big| & e^{-k^2   f(\eps k) s } - e^{-(k-1)^2  f(\eps (k-1)) s } \Big|  \, \Big|1 - e^{ikx}  \Big| \\
&\ls  \sum_{k \in \Z^\star} \Big|  e^{-k^2   f(\eps k) s } - e^{-(k-1)^2  f(\eps (k-1)) s } \Big|  \, |kx| \ls |x| s^{-\frac12}.
\end{equs}
For the second sum we write
\begin{equs}
\sum_{k \in \Z^\star}  e^{-k^2 \,  f(\eps (k-1))\, s } \Big| e^{ikx} - e^{i(k-1)x}    \Big| \ls |x|\sum_{k \in \Z^\star}  e^{-k^2 \,  f(\eps (k-1))\, s } \ls |x| s^{-\frac12}.
\end{equs}
Hence integrating over $z$ yields
\begin{equs}
\int_{-\pi}^\pi & \big( p_s^\eps(z) - p_s^\eps(z-x)\big)^2 \big| z\big|^{2 \alpha} \, dz \\
& \ls s^{\alpha} \int_{|z| \leq s^{\frac12}} |x|^2 s^{-2} dz + \int_{|z| > s^{\frac12}} |z|^{-2+2\alpha}  |x|^2 s^{-1} dz
\ls |x|^2 s^{\alpha - \frac32}.
\end{equs}
Finally, integrating over $s$ we get, splitting the integral over $[0,t]$ into an integral over $[0,|x|^2 \wedge t]$ and an integral over $[|x|^2 \wedge t,t]$
\begin{equs}
\int_0^t  & \int_{-\pi}^\pi \big( p_s^\eps(z) - p_s^\eps(z-x)\big)^2 \big| z\big|^{2 \alpha} \, dy \, ds \\
&\leq \int_0^{|x|^2 \wedge t}  s^{\alpha -\frac{1}{2} } + |x|^{2\alpha} s^{-\frac{1}{2} } \, ds + \int_{|x|^2 \wedge t}^t |x|^2 s^{\alpha - \frac32}\, ds \ls |x| t^\alpha\;,
\end{equs}
thus concluding the proof.
\end{proof}

We will now proceed to prove bounds on the regularisation property of the approximate 
heat semigroup $S_\eps$ on H\"older spaces. These bounds are similar to the well-known optimal 
regularity results for the heat semigroup $S$.  Unfortunately, we cannot apply standard multiplier results 
in H\"older spaces such as the one in \cite{ABB04}, since in our application the conditions in these spaces are typically not 
satisfied uniformly in $\eps$. We circumvent this problem by proving optimal regularity in $L^p$-based Sobolev spaces and 
then use Sobolev embeddings. In this way, we do not obtain the optimal regularity though, but we 
always loose an arbitrarily small exponent $\kappa$.

We first state a simple corollary of the classical Marcinkiewicz multiplier theorem \cite{Mar39}. 
In order to state the result we introduce the following notation. For any sequence $\{ m(k)\}_{k \in \Z}$ we define
\begin{align}
 \| m \|_{\MM}
  := \sup_{k \in \Z} |m(k)|
   + 
   \sup_{l \geq 0} \sum_{k = 2^l}^{2^{l+1}-1} 
 \sum_{\sigma \in \{-1, 1\}}
   | m(\sigma k) - m(\sigma(k+1)) |\;. \label{e:MarkNorm}
\end{align}
The result can now be formulated as follows.

\begin{lemma}\label{lem:Mark}
Let $\{ m(k)\}_{k \in \Z}$ be a real sequence and let $T_m$ be  the associated  Fourier multiplication operator given  by
\begin{align*}
 T_m e^{ik\cdot} = m(k)e^{ik\cdot}\;.
\end{align*}
For any $\gamma \in \R$ we define the sequence $\{m^\gamma\}_{k \in \Z}$ by
\begin{equs}\label{e:ml}
m^\gamma(k) = 
|k|^{-\gamma} m(k)  
\end{equs}
for $k \neq 0$ and $m^\gamma(0)= m(0)$. Then, for any $\bar{\gamma} > ( 0 \wedge \, - \gamma)$ and any  $0 <\kappa < \bar{\gamma}$   we have
\begin{equs}
\big\| T_m \big\|_{\cC^{\bar{\gamma} + \gamma } \to \cC^{\bar{\gamma} - \kappa} } \, \ls \, \| m^\gamma \|_{\mathcal{M} }.
\end{equs}
\end{lemma}
\begin{proof}
For any $1 < p < \infty$  the Marcinkiewicz multiplier theorem \cite{Mar39} asserts that
\begin{equs}
\big\| T_{m^\gamma} \big\|_{L^p \to L^p} \ls \| m^\gamma \|_{\mathcal{M} },
\end{equs}
where $T_{m^\gamma}$ is the Fourier multiplier associated to $m^\gamma$. Hence, it follows immediately from the definition of the Bessel potential spaces $H^{\gamma ,p}$   that 
\begin{equs}
\big\| T_{m} \big\|_{H^{\gamma + \bar{\gamma} - \kappa/2,p} \to H^{  \bar{\gamma} - \kappa/2,p}} \ls \| m^\gamma \|_{\mathcal{M} }.
\end{equs}
(See, e.g. \cite[Section 6.1.2]{Gra09} or \cite{MR1228209} for proofs in the whole space; the extension to the torus is immediate.) Then the desired statement follows from the embedding
\begin{equs}
\cC^{\gamma + \bar{\gamma} } \hookrightarrow  H^{\gamma + \bar{\gamma} - \kappa/2,p}, 
\end{equs}
and the Sobolev embedding
\begin{equs}
H^{ \bar{\gamma} - \kappa/2,p} \hookrightarrow \cC^{\bar{\gamma} - \kappa},
\end{equs}
which holds as soon as $p$ is sufficiently large.
\end{proof}

With this result in hand we are now ready to derive the bounds on $S_\eps$. Throughout the following lemmas we will use the notation
\begin{equ}[e:mm]
\met(k) = e^{- k^2 f(\eps k) t}.
\end{equ} 
In this notation Assumption \ref{a:BV} implies that
\begin{equ}[e:aBV]
\sup_{\eps,t>0} \big\| m_{\eps,t} \big\|_{\BV} < \infty . 
\end{equ}
This is because $m_{\eps,t}(k) = b_{t/\eps^2} (\eps k)$, and the $\BV$-norm is invariant under re\-pa\-ra\-me\-tri\-sa\-tions.

\begin{lemma}\label{le:newOp}
For any $\gamma, \bar{\gamma} \geq 0$  and for any $t>0$ we have
\begin{equ}[e:ApOp1]
 \sup_{\eps \in (0,1)}   \| S_{\eps}(t) \|_{\cC^{\bar{\gamma}} \to \cC^{\bar{\gamma} +\gamma- \kappa}}  \lesssim t^{-\frac{\gamma}{2}}\;.
\end{equ}
\end{lemma}
\begin{proof}
We have for any $k \in \Z^\star$ that
\begin{equs}
|k|^\gamma \met(k) =   |k|^\gamma \,e^{- k^2 f(\eps k) t} \ls t^{-\frac{\gamma}{2}} \sup_{x \in \R}  |x|^\gamma e^{-x^2 c_f} \ls t^{-\frac{\gamma}{2}}.
\end{equs}
To bound the $\BV$ norm of $|k |^\gamma \met$ we write
\begin{equs}
&\sum_{k \in \Z \setminus \{0,1 \}} 
 \Big| \ |k+1|^\gamma \met(k+1) - |k|^\gamma \met (k) \ \Big| \\
& \leq \sum_{k \in \Z \setminus \{0,1 \}} \big| \, |k+1|^\gamma -|k|^\gamma \big|  \met(k) + \sum_{k \in \Z \setminus \{0,1 \}} |k|^\gamma |\met(k+1) -\met(k)|.
\end{equs}
To bound the first term we use the fact that $\big| \, |k+1|^\gamma -|k|^\gamma \big| \ls   |k|^{\gamma -1}$ to obtain
\begin{equs}
\sum_{k \in \Z \setminus \{0,1 \}} \big| \, |k+1|^\gamma -|k|^\gamma \big| \met(k) \ls \sum_{k \in \Z \setminus \{0,1 \}}   |k|^{\gamma-1} e^{ -k^2 t c_f } \ls t^{-\frac{\gamma}{2}}.
\end{equs}
To bound the second term we use the same argument as  above in \eqref{e:brut} to show that
\begin{equs}[e:brut-oncemore]
\big| t^{\frac12} k \big|^{\gamma} & \ls \Big( e^{- k^2  f(\eps k)t/2} + e^{- (k-1)^2  f(\eps(k-1))t/2} \Big)^{-1} .
\end{equs}
This bound then implies that
\begin{equs}
\sum_{k \in \Z \setminus \{0,1 \}} & |k|^\gamma |\met(k+1) -\met(k)| \\
& \ls t^{\frac{\gamma}{2}} \sum_{k \in \Z \setminus \{0,1 \}} |m_{\eps,t/2}(k+1) - m_{\eps,t/2}(k)|   \ls t^{\frac{\gamma}{2}},
\end{equs}
where in the last inequality we have made use of \eqref{e:aBV}.
Then  Lemma \ref{lem:Mark} implies the desired bound \eqref{e:ApOp1}
\end{proof}

\begin{lemma}\label{le:ApOp}
 For  any $\gamma \in [0,1]$, for $\bar \gamma \geq 0$, and for $\kappa>0$  we have
\begin{equ}[e:ApOp2]
\sup_{  t \in [0,T]}  \|  S(t)- S_{\eps}(t)     \|_{\cC^{\bar{\gamma} + \gamma} \to \cC^{\bar{\gamma} - \kappa}} \lesssim\eps^\gamma \;.
\end{equ}
\end{lemma}

\begin{proof}
%The first bound follows immediately from Lemma \ref{lem:Mark} observing that 
%%
%\begin{equs}
%\big\| \met \big\|_{\mathcal{M}} \leq 1 +  2 \big\| \met \big\|_{\BV} \;,
%\end{equs}
%%
%which is uniformly bounded by  \eqref{e:aBV}.

As a shorthand, we use the notations
\begin{equ}
\deltaa \eqdef \met - \mt\;,\qquad 
\deltaag \eqdef \met^\gamma - \mt^\gamma\;,
\end{equ}
where $\met$ is defined in  \eqref{e:mm}  and  $\met^\gamma (k) := |k|^{-\gamma} \met$ and $\mt^\gamma := |k|^{-\gamma} \mt$ as in \eref{e:ml}.
The bound follows from  \ref{lem:Mark} as soon as we have established the estimate
\begin{equs}[e:Mar1]
\big\| \deltaag \big\|_{\mathcal{M}} \ls \eps^\gamma\;.
\end{equs}
Observe that $\deltaag(0) =0$, so that from now on we will only deal with $k \neq 0$.  By symmetry it suffices to consider the positive Fourier modes, hence, to simplify notation,  we will neglect the terms with $\sigma = -1$ in the definition of $\|\cdot\|_{\mathcal{M}}$.

In order to establish \eref{e:Mar1} we start by showing that 
\begin{equs}[e:Mar2]
\sup_{k \in \Z} \big| \deltaag(k) \big| \ls \eps^{\gamma}. 
\end{equs}

Recall that according to Assumption \ref{a:f} $f$ is differentiable on $(-\delta,\delta)$ with bounded derivatives. Therefore, if $0 < |\eps k| \leq \delta$ we can write 
\begin{equs}
\big|  \deltaag(k) \big| &= \frac{1}{|k|^\gamma} \Big[ e^{- tk^2 f(\eps k) } - e^{ -t k^2 } \Big] \label{e:Mar3}\\
&\ls  \frac{1}{|k|^\gamma}  e^{- c_f  t k^2  } t k^2 \big| f(\eps k) -1 \big|
\ls \frac{1}{|k|^\gamma}  |\eps k| \ls \eps^\gamma.
\end{equs}
Here we have made use of the fact that the function $x \mapsto x \exp\big( - c_f x\big)$ is bounded on $[0, \infty)$ as well as of  the boundedness of $f'$ on $(-\delta,\delta)$.

If $|\eps k| \geq \delta$ the bound \eref{e:Mar2} can be established simply by writing
\begin{equs}[e:Mar4]
\big|  \deltaag(k) \big| \ls |k|^{-\gamma} \ls \eps^\gamma.
\end{equs}
The bounds on the $\BV$-norms of the Paley-Littlewood blocks 
\begin{equs}
\sum_{k = 2^l }^{2^{l+1}-1} \big| \deltaag (k) - \deltaag (k+1) \big|
\end{equs}
require more thought.  
Actually, we can always write using the inequality $|f g |_{\BV} \leq |f|_{\cN} |g|_{\BV} + |g|_{\cN} |f|_{\BV}$ 
\begin{equs}
 \sum_{k = 2^l }^{2^{l+1}-1} \big| \deltaag (k) - \deltaag (k+1) \bigr| &\le  \frac{1}{2^{l\gamma}} \sum_{k = 2^l }^{2^{l+1}-1} \big| \deltaa (k) - \deltaa (k+1) \big| \\
&\quad +  \frac{1}{2^{l\gamma}} \sup_{k \in [2^l, 2^{l+1}]}  \big| \deltaa(k) \big|\;.\label{e:Mar5}
\end{equs}
The second summand can be bounded as in  \eref{e:Mar3}  and \eref{e:Mar4}. We get
\begin{equs}[e:Mar6]
\frac{1}{2^{l\gamma}} \sup_{k \in [2^l, 2^{l+1}]} \big| \deltaa(k) \big| \ls \eps^\gamma.
\end{equs}
For the first term on the right-hand side of \eref{e:Mar5} we distinguish between different cases. 

We first consider the case where  $\eps 2^{l+1} \geq\delta$. In this case the   $\met(k)$ for $k \in [2^l, 2^{l+1}]$ are not good approximations to the $\mt(k)$. Hence we bound the difference by the sum 
\begin{equ}
 \big| \deltaa (k) - \deltaa (k+1) \big| 
  \leq    \big| \met(k) - \met(k+1)  \big|  + \big|  \mt (k) - \mt (k+1) \big| \;.\label{e:Mar7}
\end{equ}
Then we get using Assumption \ref{a:BV} on the boundedness of the $\BV$ norm of $\met$
\begin{equs}[e:Mar8]
\frac{1}{2^{l\gamma}} \sum_{k = 2^l }^{2^{l+1}-1} \big| \met(k) - \met(k+1)  \big|  \ls \frac{1}{2^{l\gamma}} \ls \eps^\gamma.
\end{equs}
The second term on the right-hand side of \eref{e:Mar7} can be bounded in the same way.

Secondly, we consider the case  $\eps 2^{l+1} < \delta$.
In order to treat this case, we claim that for any non-negative numbers $g_{ij}$ with $i,j \in \{0,1\}$, we have
\begin{equs}\label{eq:e-id}
| e^{-g_{00}} - e^{-g_{01}} &- e^{-g_{10}} + e^{-g_{11}} |  \leq e^{-m}\Big[  \big|  {g_{00}} - {g_{01}}- {g_{10}} + {g_{11}} \big|   
 \\& + \Big( |g_{00} -g_{01}| + |g_{10} -g_{11}|\Big)
     \Big( |g_{00} -g_{10}| + |g_{01} -g_{11}|\Big) 
   \Big]\;,
\end{equs}
where $m = \min g_{ij}$. To see this,
set
\begin{align*}
g(\lambda,\mu)
 =  (1-\lambda)(1-\mu) g_{00} + (1-\lambda)\mu g_{01}
      + \lambda(1-\mu) g_{10} + \lambda\mu g_{11}
\end{align*}
and note that the left-hand side of \eref{eq:e-id} can be written as 
\begin{align*}
 & \Big|  \int_0^1 \int_0^1 \partial_\lambda \partial_\mu 
     \exp ( - g(\lambda, \mu)) \dd\lambda \dd \mu \Big|
 \\& \leq 
  \Big|  \int_0^1 \int_0^1 
   \Big[  | \partial_\lambda \partial_\mu g(\lambda, \mu)|
    +  |\partial_\lambda g(\lambda, \mu) |
    | \partial_\mu g(\lambda, \mu) |  \Big]
     \exp ( - g(\lambda, \mu)) \dd\lambda \dd \mu \Big|\;.
\end{align*}
The estimate \eref{eq:e-id} follows using the inequalities
\begin{align*}
  \partial_\lambda \partial_\mu g(\lambda, \mu)
  &  =   {g_{00}} - {g_{01}}- {g_{10}} + {g_{11}}\;, \\
    \partial_\lambda g(\lambda, \mu)
  &  \leq  | {g_{00}} - {g_{10}} | + | {g_{01}} - {g_{11}}|\;,\\
    \partial_\mu g(\lambda, \mu)
  &  \leq  | {g_{00}} - {g_{01}} | + | {g_{10}} - {g_{11}}|\;,\\
  g(\lambda, \mu) & \geq m\;.
\end{align*}
Applying this estimate to $g_{0i} =  (k+i)^2 t$ and $g_{1i} =  (k+i)^2 t f(\eps (k+i))$, we infer that 
\begin{equ}\label{e:Mar11}
 \frac{1}{2^{l\gamma}}   \sum_{k = 2^l }^{2^{l+1}-1} \big| \deltaa (k) - \deltaa (k+1) \big|
   \leq  \frac{1}{2^{l\gamma}}  \sum_{k = 2^l }^{2^{l+1}-1} e^{- c_f t k^2}
\big(B_{\eps,t}(k) +  C_{\eps,t}(k)\big)
\end{equ}
where 
\begin{align*}
B_{\eps,t}(k) 
 =  \Big| t k^2 \big(f(\eps k) -1 \big) - t (k+1)^2 \big( f(\eps (k+1)) -1 \big)    \Big|  
 \lesssim t \eps k^2\;,
\end{align*}
and, taking into account that $\eps k \lesssim 1$,
\begin{align*}
  C_{\eps,t}(k) 
  &  = 
\Big[  \Big| t k^2 -  t (k+1)^2 \Big|
      + \Big| t k^2 f(\eps k)-  t (k+1)^2 f(\eps (k+1)) \Big|   \Big] 
\\& \qquad \times    \Big[ \Big| t k^2 \big(f(\eps k) -1 \big) \Big|
      +  t (k+1)^2 \big|f(\eps (k+1)) -1 \big|  \Big] 
\\& \ls  (t k) \cdot (t \eps  k^3)
      = t^2 \eps k^4\;.
\end{align*}
Using these bounds, together with the fact that  $M = \sup_{x \geq 0} \{ x e^{-x},  x^2 e^{-x} \} < \infty$, we infer that
\begin{equ}
 \frac{1}{2^{l\gamma}}  \sum_{k = 2^l }^{2^{l+1}-1} \big| \deltaa (k) - \deltaa (k+1) \big|
  \lesssim
   \frac{M}{2^{l\gamma}}  \sum_{k = 2^l }^{2^{l+1}-1} \eps 
   \lesssim 2^{l(1-\gamma)} \eps 
   \lesssim \eps^{\gamma}\;.
\end{equ}
This finishes the proof of \eref{e:Mar1} and hence of \eref{e:ApOp2}.
\end{proof}

The following result is now an immediate consequence. 

\begin{corollary}\label{cor:S-Seps}
Let $\lambda \in [0,1]$ and $\alpha \leq \gamma  + \lambda$. For $\kappa > 0$ sufficiently small,
\begin{align}\label{eq:S-diff}
 | S(t) - S_\eps(t) |_{\cC^\alpha \to \cC^\gamma}
  \lesssim t^{-\frac12(\gamma - \alpha + \lambda+\kappa)} \eps^{\lambda}\;.
\end{align}
\end{corollary}

\begin{proof}
This follows from the decomposition
\begin{align*}
S(t) - S_\eps(t) =  \big(S(t/2) - S_\eps(t/2)\big)   \big( S( t/2) + S_\eps(t/2) \big) \;,
\end{align*}
and Lemma \ref{le:newOp}  and  Lemma \ref{le:ApOp}.
\end{proof}

The next result concerns the time regularity of solutions to the approximated heat equation. Recall that the approximated heat semigroup $S_\eps$ is not strongly continuous at $0$ and we cannot expect convergence to zero of  $\big|S_\eps(t) - \Id \big|_{\cC^\gamma \to \cN}$ as $t \to 0$. However, the following result states that the approximating semigroup has nice time continuity properties for times $t \geq s$ with a blowup if $s \leq \eps^2$.

\begin{lemma}\label{lem:timecontS}
Let $\bar{\gamma} \geq 0$ and $\gamma \in [0,2]$. Then, for all $t\geq s >0$  we have
\begin{align*}
  |  S_\eps(t) -  S_\eps(s)   |_{\cC^{\bar \gamma +\gamma+\kappa}\to \cC^{\bar \gamma}}
    \lesssim  \big( 1 + \eps^\gamma s^{-\frac{\gamma}{2} } \big) |t-s|^{\frac{\gamma}{2}}  \;.
\end{align*}
\end{lemma}

\begin{proof}
As above we write
\begin{equ}
 m_{\eps,t}^\gamma(k) = k^{-\gamma} 
     \exp(-t k^2 f(\eps k))\;,\qquad
 m_{\eps,t}(k) = \exp(-t k^2 f(\eps k))\;.
\end{equ}
Lemma \ref{lem:Mark} implies the desired result as soon as we have established that 
\begin{equs}
\big| m_{\eps,t}^\gamma - m_{\eps,s}^\gamma \big|_{\MM} \ls  \big( 1 + \eps^\gamma s^{-\frac{\gamma}{2} } \big)(t-s)^{\frac{\gamma}{2}},
\end{equs}
where the norm $| \cdot |_{\MM}$ has been defined in \eref{e:MarkNorm}. By symmetry it suffices to consider the Fourier coefficients with $k > 0$, i.e. the terms with $\sigma = 1$. Throughout the calculations we will write
\begin{equ}
\deltamg(k) \eqdef m_{\eps,t}^\gamma(k) - m_{\eps,s}^\gamma(k)\;,
\end{equ}
and similarly for $\deltam$.

In order to bound the supremum of the $\deltamg$ we write for any $k \in \Z^\star$
\begin{align*}
 | \deltamg(k)|
&  =  k^{-\gamma}  e^{-s k^2 f(\eps k)} 
    \big( 1 - e^{-(t-s) k^2 f(\eps k)}  \big)
\\&  \lesssim   (t-s)^{\frac{\gamma}{2}} e^{- s  k^2 f(\eps k)} 
     f(\eps k)^{{\frac{\gamma}{2}}}
\end{align*}
If $|\eps k| \leq \delta$ Assumption \ref{a:f} implies that $f$ is bounded and hence the whole expression is bounded by $|t-s|^{\frac{\gamma}{2}}$ up to a constant. If $|\eps k| > \delta$ we can write 
\begin{equs}\label{e:Le6666}
 (t-s)^{{\frac{\gamma}{2}}} e^{-s  k^2 f(\eps k)} 
     f(\eps k)^{\frac{\gamma}{2}}
      &\ls (t-s)^{{\frac{\gamma}{2}}} e^{-\frac{s}{\eps^2}  \delta^2 f(\eps k)} 
     f(\eps k)^{\frac{\gamma}{2}}\\
     &\ls  (t-s)^{\frac{\gamma}{2}}  \eps^\gamma s^{-\frac{\gamma}{2}}  \sup_{x \in \R} e^{-  \delta^2 |x| } |x|^{\frac{\gamma}{2}} \ls  (t-s)^{\frac{\gamma}{2}} . 
\end{equs}

It remains to bound the Paley-Littlewood blocks. We start with the case $0 < \eps 2^{l+1} \leq \delta$. On $[-\delta,\delta]$ the function $f$ is $\CC^1$ by assumption. In this case we will  show that
\begin{equ}[e:Le661]
\sum_{k = 2^l}^{2^{l+1}-1} 
   \big| \deltamg (k) - \deltamg(k+1) \big|\lesssim |t-s|^{\gamma / 2}\;.
\end{equ}
We start by writing
\begin{equs}
\sum_{k = 2^l}^{2^{l+1}-1 } 
  & \big| \deltamg (k) - \deltamg(k+1) \big|  \leq \int_{2^l}^{2^{l+1}} \big|  \partial_x \deltamg(x) \big| \, dx\\
   & \leq  \int_{2^l}^{2^{l+1}} \Big| \big[  \eps (t- s) x^{2- \gamma} f'(\eps x)  + 2 x^{1 - \gamma} (t-s)  f(\eps x)  \big] \,  e^{ -t x^2 f( \eps x) }  \Big| \, dx \qquad  \\
&\quad + \int_{2^l}^{2^{l+1}}  \Big|  { \eps  s x^{3} f'(\eps x) + 2 s x^2 f(\eps x) + \gamma \over x^{1+\gamma}}\Big| \Big(e^{  - sx^2f(\eps x)  } -  e^{  - tx^2f(\eps x  )} \Big)  \, dx \\
   & \lesssim |t-s| \int_{2^l}^{2^{l+1}} x^{1-\gamma} e^{ -t x^2 f( \eps x) }  \, dx  \label{e:Lem662}\\
&\quad + \int_{2^l}^{2^{l+1}}  {s x^2 + 1 \over x^{1+\gamma}} \Big(e^{  - sx^2f(\eps x)  } -  e^{  - tx^2f(\eps x  )} \Big)  \, dx\;.
\end{equs}
For the first term, we use the boundedness of $f$ on $[-\delta,\delta]$ as well as the lower bound $f \geq c_f$,
so that
\begin{equs}
|t-s| \int_{2^l}^{2^{l+1}} \Big| x^{1 - \gamma} e^{ -t x^2 f( \eps x) }  \Big| \, dx  &\le  |t-s|^{\frac{\gamma}{2} }\int_{|t-s|^{1/2} 2^l}^{|t-s|^{1/2}2^{l+1}}  z^{1 - \gamma}   \,  e^{ - z^2 c_f }    \, dz \\
& \ls |t-s|^{\frac{\gamma}{2} }\;,
\end{equs}
as required, where we used the fact that $|t-s|\le t$. We break the second term in two components. For the first one, we have 
\begin{equs}
\int_{2^l}^{2^{l+1}}  {s x^2 \over x^{1+\gamma}} &\Big(e^{  - sx^2f(\eps x)  } -  e^{  - tx^2f(\eps x  )} \Big)  \, dx \\
& \ls \int_{2^l}^{2^{l+1}}    {s x^{1-\gamma}}  \,  e^{  -   sx^2 c_f } \big| 1 -  e^{  - (t-s) \,x^2 f(\eps x  )} \big|  \, dx\\
& \ls \int_{2^l}^{2^{l+1}}   s x^{1-\gamma}  \,  e^{  -  sx^2 c_f }  \big|  (t-s) \,x^2 f(\eps x  ) \big|^{\frac{\gamma}{2}}  \, dx\\
 &\ls  (t-s)^{  \frac{\gamma}{2}} \int_{0}^{\infty} 
 		z \,  e^{  - 2 z^2 c_f  }  \, dz  \ls (t-s)^{  \frac{\gamma}{2}}.
\end{equs}
For the remaining term, we obtain
\begin{equs}
\int_{2^l}^{2^{l+1}}  {1 \over x^{1+\gamma}} &\Big(e^{  - sx^2f(\eps x)  } -  e^{  - tx^2f(\eps x  )} \Big)  \, dx \\
& \ls \int_{2^l}^{2^{l+1}}    {1 \over x^{1+\gamma}}   \big| 1 -  e^{  - (t-s) \,x^2 f(\eps x  )} \big|  \, dx
 \ls \int_{2^l}^{2^{l+1}}    {1 \wedge |t-s|x^2 \over x^{1+\gamma}}  \, dx\\
 &\ls  (t-s)^{  \frac{\gamma}{2}} \int_{0}^\infty 
 		{1 \wedge z^2 \over z^{1+\gamma}}   \, dz  \ls (t-s)^{  \frac{\gamma}{2}}.
\end{equs}

Let us now treat the case $\delta \leq \eps 2^{l+1} $.  In this case we will establish  
\begin{equ}[e:Le6611]
\sum_{k = 2^l}^{2^{l+1}-1} 
   \big| \deltamg (k) - \deltamg(k+1) \big|\lesssim \eps^\gamma s^{-\frac{\gamma}{2}} |t-s|^{\gamma / 2}\;.
\end{equ}

Since $\deltamg(x) = x^{-\gamma} \deltam(x)$, we obtain
\begin{equs}
 \sum_{k = 2^l}^{2^{l+1}-1} 
   \big| \deltamg (k) - \deltamg(k+1) \big|  
&\leq  \frac{1}{2^{l\gamma}} \sum_{k = 2^l}^{2^{l+1}-1} 
   \big|\deltam(k) - \deltam(k+1) \big|  
\\&\quad    +  \frac{1}{2^{l\gamma}} \sup_{k \in [2^l, 2^{l+1}]}  \big| \deltam(k) \big|\;. \label{eq:pl}
\end{equs}
The second term in this expression can be bounded by
\begin{align*}
 \frac{1}{2^{l\gamma}} \sup_{k \in [2^l, 2^{l+1}]}  \big|\deltam(k) \big|
&  \lesssim 
 \frac{1}{2^{l\gamma}}
  \sup_{k \in [2^l, 2^{l+1}]}     \exp(-s k^2 f(\eps k)) 
 | t - s |^{\frac{\gamma}{2}} k^\gamma f(\eps k)^{\frac{\gamma}{2}}
\\&  \lesssim 
 | t - s |^{\frac{\gamma}{2}}  \, \sup_{k \in [2^l, 2^{l+1}]}  e^{-s k^2 f(\eps k)} 
 f(\eps k)^{\frac{\gamma}{2}}\;.
\end{align*}
As above in \eqref{e:Le6666} using the fact that $k \geq \frac12 \delta$ this expression can be bounded by $ \eps^\gamma s^{-\frac{\gamma}{2}} | t - s |^{\frac{\gamma}{2}}\;$ up to a constant.

It remains to bound the terms
\begin{equ}[e:BVNorm1]
 \frac{1}{2^{l \gamma}} \,\sum_{k = 2^l}^{2^{l+1}-1} 
    \big| \deltam (k) - \deltam(k+1) \big| \le  \frac{1}{2^{l \gamma}} |\deltam |_\BV\;.
\end{equ}
For this, it turns out to be sufficient to show that for any $\lambda_1 < \lambda_2$ we have the bound
\begin{equ}[e:suff1]
\big|G_{\lambda_1,\lambda_2} \big|_{\BV}  \ls \frac{\lambda_2 - \lambda_1}{\lambda_1}\;,
\end{equ}
where
\begin{equ}
G_{\lambda_1,\lambda_2}(x) :=  \exp\big(-\lambda_1 x^2 f(x) \big) - \exp\big(- \lambda_2  x^2 f(x) \big) \;.
\end{equ}
Assuming that we have established \eref{e:suff1}, we can rewrite $\deltam$ as
\begin{equ}
    \deltam(x) =  \,G_{\frac{s}{\eps^2},\frac{t}{\eps^2}}(\eps x)\;.
\end{equ}
Then \eref{e:suff1} implies the bound
\begin{equ}
|\deltam(x) |_\BV \lesssim {|t-s|\over s}\;.
\end{equ}
On the other hand, Assumption \ref{a:BV} immediately  implies that $|\deltam(x) |_\BV \lesssim 1$,
so that $|\deltam(x) |_\BV \lesssim |t-s|^{\frac{\gamma}{2}} s^{-\frac{\gamma}{2}}$.
Plugging this back into \eref{e:BVNorm1} and using the fact that $\eps 2^l  \geq \delta/2$, 
we immediately obtain the required bound.

It remains to show \eref{e:suff1}.For any $\hat{x} ,c>0$ we set $A_{c}(\hat{x}) = e^{-\hat x} - e^{-(1+c \,)\hat{x}}$. then for $\hat{x} < \hat{y}$ we observe the bound
\begin{equs}
|A_{c}(\hat x) - A_{c}(\hat y)|& =  \int_{\hat x}^{\hat y} e^{-z} - (1+c) e^{-(1+c \,)z} \,dz  \ls  c \int_{\hat x}^{\hat y}     e^{- \frac{z}{2}}  \,dz \\
& = c  \big| e^{-\frac{\hat x}{2}} - e^{-\frac{\hat y}{2}} \big| \;.
\end{equs}
Applying the observation for $\hat x := \lambda_1 x^2 f(x)$ and for $c :=   \frac{\lambda_2 - \lambda_1}{\lambda_1}$
we see that the $\BV$-norm of $G_{\lambda_1,\lambda_2}$ is bounded by $ \frac{\lambda_2 - \lambda_1}{\lambda_1}$ times the $\BV$-norm of
$\exp\bigl(-\frac{\lambda_1}{2} x^2 f(x)\bigr)$.
This on the other hand is bounded uniformly in $\lambda_1$ by Assumption \ref{a:BV}, so that we have established \eref{e:suff1}.
\end{proof}

\appendix
\section{Rough integrals} \label{AppA}

In this appendix we briefly summarise the definition and the properties of rough integrals we use. We refer the reader to \cite{LQ02,Gu04,LCL07,HW10,FV10} for a more complete account of rough path theory.

As above for $X \in \cN$ we will always use the notation $\delta X(x,y) := X(y) - X(x)$. For $R \in \bN$ we will write $\delta R(x,y,z) := R(x,z) - R(x,y) - R(y,z)$. See \cite{Gu04,GT10} for a discussion of the algebraic properties of these operators.

We want to define integrals of the type 
\begin{equ}[e:RoIn]
\int_x^y Y(z) \,  \otimes \,  d Z(z) 
\end{equ}
for functions $Y,Z \in \cC^\alpha$ for some $\alpha \in (0,1)$. 
If  $\alpha >  \frac{1}{2}$ such integrals can be defined as limits of Riemann-sums of type
\begin{equ}[e:RiemSum1]
\sum_{i} Y(x_i)  \otimes \delta Z(x_i, x_{i+1}).
\end{equ}
This yields the Young integral.

%%%%%%%%%%%%%%%%%%%%%%%%%%%%%%%%%%%%%%%%%%%%%%%
%%% Definition Rough path
%%%%%%%%%%%%%%%%%%%%%%%%%%%%%%%%%%%%%%%%%%%%%%%

If $\alpha \le  \frac{1}{2}$ Riemann sums of type \eref{e:RiemSum1} will in general fail to converge. The idea is then to define a better approximation with the help of additional data. To this end we introduce the following definitions. 

\begin{definition}\label{def:RP}
An $\alpha$-rough path  consists of two functions $X \in \cC^\alpha(\R^n)$ and $\mathbf{X}\in \bB^{2\alpha}\big( \R^n  \otimes  \R^n \big)$ satisfying the relation
\begin{gather}
\mathbf{X}(x,z) - \mathbf{X}(x,y) - \mathbf{X}(y,z)  \,=\,  \delta X(x,y)  \otimes \delta X(y,z) \label{eq:cond-it-int}
\end{gather}
for all $x,y,z$. An $\alpha$-rough path $(X,\XX)$ is called \emph{geometric} if in addition for every $x,y$ the symmetric part $\XX^+(x,y) =  \frac{1}{2}(\XX(x,y) +\XX(x,y)^T$ of $\XX(x,y)$ satisfies
\begin{equs}
\XX^+(x,y) =  \frac{1}{2} \delta X(x,y) \otimes \delta X(x,y) .
\end{equs}

\end{definition}

%%%%%%%%%%%%%%%%%%%%%%%%%%%%%%%%%%%%%%%%%%%
%% Gubinelli Rough paths
%%%%%%%%%%%%%%%%%%%%%%%%%%%%%%%%%%%%%%%%%%%

Following  \cite{Gu04} we also define:

\begin{definition}
Let $X$ be in $\cC^\alpha$  A pair $(Y,Y')$ with  $Y \in \cC^\alpha$  and $Y'\in \cC^{\alpha}(\mathcal{L}(\R^n) )$ is said to be \emph{controlled} by $X$ if for all $x,y $
\begin{gather}
\delta Y(x,y)= Y'(x) \, \delta X(x,y)  + \mathcal{R}_Y(x,y),\label{eq:contr-rp}
\end{gather}
with a remainder $\mathcal{R}_Y \in \bB^{2\alpha}$. 
\end{definition}

Note that \eref{eq:contr-rp} is a linear condition. So for a given $X$  the space of paths that are controlled by $X$ is a vector space.

\begin{remark}
In general the decomposition \eref{eq:contr-rp} need not be unique, but  in all of the situations we will encounter  there is a natural choice of $Y'$, which will be called the \emph{rough path derivative} of $Y$.
\end{remark}

%%%%%%%%%%%%%%%%%%%
%% Rough integral
%%%%%%%%%%%%%%%%%%%

%
If $Y,Z$ are  controlled by $X$ and there is a choice of $\XX$ turning $(X,\XX)$ into a rough path, we construct the rough integral integral $\iint Y {d} Z$ as the limit of the second order approximations
\begin{equation}\label{eq:Riem-sum2}
\sum_{i} Y(x_i) \otimes \big(Z(x_{i+1})- Z(x_{i})  \big) + Y'(x_i) \, \mathbf{X}(x_i,x_{i+1}) \, Z'(x_i)^T.
\end{equation}
If  $\alpha >\frac{1}{3}$, it turns out that these approximations converge:

\def\Gubinelli{\cite[Thm~1 and Cor.~2]{Gu04}}
\begin{lemma}[\Gubinelli]\label{lem:Gub-int}
Let $\alpha > \frac 13$. Suppose $(X, \XX)$ is an $\alpha$ rough path and $Y,Z $ are controlled by $X$. Then the Riemann-sums defined in \eref{eq:Riem-sum2} converge as the mesh of the partition goes to zero. We call the limit \emph{rough integral} and denote it by  $\iint Y(x) \otimes \, {d} Z(x)$. 

The mapping $(Y,Z) \mapsto \iint Y \otimes  d Z$ is bilinear and we have the following bound:
\begin{gather}\label{eq:IB1}
\iint_x^y Y(z) \otimes d Z(z)  \, = \, Y(x) \otimes  \delta Z(x,y) + Y'(x )\XX (x,y)Z'(x)^T + Q(x,y), 
\end{gather}
where the remainder satisfies
\begin{align}
\big| Q \big|_{3 \alpha} \, \ls \,&     |\mathcal{R}_Y|_{2 \alpha} |Z|_\alpha + |Y'|_{\cC^0}  |X|_{\alpha}  |\mathcal{R}_Z|_{2 \alpha}  \notag\\
& \qquad +  |\mathbf{X}|_{2 \alpha} \Big(  |Y'|_\alpha|Z'|_{\cN} +  |Y'|_{\cN}|Z'|_\alpha  \Big)  + |X|_\alpha^2 |Y'|_{\cC^0} |Z|_\alpha  . \label{eq:IB2}    
\end{align}
\end{lemma}

The rough integral also possesses continuity properties with respect to different rough paths. We refer the reader to \cite{Gu04} for more details. The reason for using the notation $\iint$ instead of $\int$ is to keep a reminder of the
fact that this is really an abuse of notation since $\iint Y\,dZ$ also depends on the choices of $Y'$, $Z'$, and $\XX$.

\section{Regularity results}\label{AppB}

We first quote a variation on a classical regularity statement due to Garsia, Rodemich and Rumsey \cite{GRR70}. For $R \in \bN$ we will use the notation
\begin{align*}
|\delta R |_{[x,y]} :=& \sup_{x \leq z_1 < z_2< z_3 \leq y} 
	 \big| \delta R(z_1,z_2,z_3) \big|\;,
\end{align*}
%	 =&  \sup_{x \leq z_1 < z_2< z_3 \leq y} 
%	 \big| R(z_1,z_3) - R( z_1,z_2) - R( z_2,z_3) \big| .
where $\delta R(z_1,z_2,z_3) := R(z_1,z_3) - R( z_1,z_2) - R( z_2,z_3)$. The following result is a special case of \cite[Lemma 4]{Gu04} applied to the functions $\psi(u)=u^p$ and $p(x)=x^{\alpha+2/p}$.

\begin{lemma}[] \label{lem:GRR}
Let $\alpha \geq 0$ and $p \geq 1$. For $R \in \bN$ we have
\begin{align*}
 |R|_{\alpha} \ls
 \bigg(   \int_{[-\pi,\pi]^2}  \frac{|R(x,y)|^p}{|x-y|^{\alpha p+2}}  \, d x \, d y \bigg)^{1/p}
   + \sup_{x<y}\frac{|\delta R |_{[x,y]}}{|x-y|^\alpha}\;.
\end{align*}
\end{lemma}

In the special case where $R(x,y) = f(y) - f(x)$ for some function $f$, the second term vanishes and one recovers a version of the Sobolev embedding theorem.

We will now proceed to extend this statement to derive bounds on functions that depend on several variables. We  will use the abbreviated notation
\begin{equ}
\| F \|_{\cC^{\alpha_1}_T(\cC^\alpha_2)} :=  \| F \|_{\cC^{\alpha_1}([0,T], \cC^{\alpha_2}[-\pi, \pi] )}, 
\end{equ}
and similarly for  $\| R \|_{\cC^{\alpha_1}_T(\bB^\alpha_2)}$. %

\begin{lemma}\label{lem:AnIsKol}
Let $\alpha_1, \alpha_2>0$, let $\gamma_1, \gamma_2 \in [0,1]$ and let $p \geq 1$ be such that
\begin{equ}[e:condal]
\alpha_1  \, < \, \gamma_1 \lambda_1 -  \frac{1}{p} \,, \quad   \alpha_2  \, < \, \gamma_2 \lambda_2 -  \frac{1}{p}  
\end{equ}
for some $\lambda_1, \lambda_2 >0$, $\lambda_3 \geq 0$ with $\lambda_1 + \lambda_2 + \lambda_3 = 1$.

\begin{enumerate}
\item Let $F$ be a random function in $C([0,T], C [-\pi, \pi])$ satisfying
\minilab{e:F-conditions}
\begin{equs} \label{e:TBF}
\sup_{x \in [-\pi, \pi] } \E   \big| F(t,x) - F(s,x)  \big|^p 
  & \leq U_1^p |t-s|^{\gamma_1 p}\;, \\
\label{e:SBF}
\sup_{ t \in [0,T]}  \E \big| F(t,x) -F(t,y)  \big|^p 
  & \leq  U_2^p |x-y |^{\gamma_2 p}\;, \\
\label{e:DBF}
  \sup_{\substack{x \in [-\pi, \pi] \\ t \in [0,T]}} 
  \E   \big| F(t,x)  \big|^p  
  & \leq  U_3^p\;.
\end{equs}
Then we have 
\begin{equs}[e:IB1]
\E  \big\| F  & \big\|_{\cC^{\alpha_1}_T (\cC^{ \alpha_2})}^p    \, \ls \,      
 \Big((  U_1^{\lambda_1 } + U_3^{\lambda_1}  ) 
   (U_2^{\lambda_2 } + U_3^{\lambda_2}) 
   U_3^{\lambda_3}
 \Big)^p.
\end{equs}
\item Similarly, let $R$ be a random function in $C([0,T], \bN[-\pi, \pi])$ satisfying
\minilab{e:R-conditions}
\begin{equs} 
\label{e:TBR}
\sup_{x ,y \in [-\pi, \pi] } \E   \big| R(t,x,y) - R(s;x,y)  \big|^p 
	& \leq U_1^p |t-s|^{\gamma_1 p}\;,\\
\label{e:SBR1}
\sup_{ t \in [0,T]}  \E \big| R(t;x,y)  \big|^p
    & \leq  U_2^p  |x-y|^{\gamma_2 p}\;, \\
\label{e:SBR2} \notag
\sup_{ t \in [0,T]} \E  \big| \delta R(t) \big|_{[x,y]}^p       
    & \leq  U_2^p  |x-y|^{\gamma_2 p}\;, \\
\label{e:DBR}
\sup_{\substack{x,y \in [-\pi, \pi] \\ t \in [0,T]}}
	 \E \big|  R(t;x,y)\big|^p
	  &  \leq  U_3^p\;.
\end{equs}
Then we have 
\begin{equs}[e:IB2]
\E   \big\| R \big\|_{\cC^{\alpha_1}_T (\bB^{\alpha_2}) }^p   \,  \ls     
 \Big((  U_1^{\lambda_1 } + U_3^{\lambda_1}  ) 
   (U_2^{\lambda_2 } + U_3^{\lambda_2}) 
   U_3^{\lambda_3}
 \Big)^p.
\end{equs}
\end{enumerate}
\end{lemma}

\begin{proof}
Let us start by proving \eref{e:IB1}. To this end for fixed $0 \leq s <t \leq T$ we introduce the notation 
\begin{equ}
\dFs( x) = F(t,x) - F(s,x).  
\end{equ}
We have to bound the quantity 
\begin{equ}[e:APB1]
\E   \big\|   F  \big\|_{\cC^{\alpha_1}_T (\cC^{ \alpha_2})}^p   
 \ls    \E \bigg(  \sup_{0 \leq s< t \leq T}   \frac{\big|  \dFs    \big|_{\cC^{ \alpha_2}}}{|t-s|^{\alpha_1} }  \bigg) ^p    +    \E    \big|  F(0,  \cdot)    \big|_{\cC^{ \alpha_2}}^p   .\qquad 
\end{equ}
To this end for fixed $s,t$ we can write
\begin{equ}[e:APB2]
\E    \big|  \dFs    \big|_{\cC^{ \alpha_2}}^p    \ls  \E \bigg(   \sup_{x \neq y \in [-\pi ,\pi ]}   \frac{ \big| F_{s, t}(y) - F_{s, t}(x) \big| }{|x-y|^{ \alpha_2}} \bigg)^p      +  \E   \big|  \dFs (0) \big|^p . \qquad
\end{equ}
 For the first term we get using the Garsia-Rodemich-Rumsey Lemma \ref{lem:GRR},
\begin{equs}[e:APB3]
\E  \bigg(     \sup_{x,y \in [-\pi ,\pi ]} &  \frac{\big| F_{s, t}(y) - F_{s, t}(x) \big| }{|x-y|^{\alpha_2}}    \bigg)^p  \\ 
 \ls \,&    \,  \E \bigg[   \int_{[-\pi, \pi]^2 }    \frac{|\dFs(x) -\dFs(y)|^p  }{|x-y|^{\alpha_2 p+  2}} \, dx \, dy \bigg] \\
 =  \,&      \int_{[-\pi, \pi]^2 }   \frac{1}{|x-y|^{\alpha_2 p+  2 }}   \, \E   \Big(  \dFs(x) -\dFs(y)   \Big)^p \, dx \, dy . 
\end{equs}
Using H\"older inequality the expectation in the last integral can be bounded by
\begin{equs}[e:APB4]
  \E   \big|  \dFs(x) -\dFs(y)  \big|^p 
   &\ls    \sup_{ \substack{ z \in \{x,y\}}} \! \! 
   \big(\E  \big|   F(t,z) - F (s,z)  \big|^p \big)^{\lambda_1}
    \\ 
      &\qquad\times  \! \sup_{ \substack{ r \in \{s,t\} } } \!
        \big( \E \,  \big|   F(r,x) - F(r,y)  \big|^p \big)^{\lambda_2}\\
        &\qquad \times     \!\!\!   
         \sup_{ \substack{ z \in \{ x,y \}  \\ r \in \{s,t\}}}  \!  \!  \big(  \E  \big|   F(r,z)  \big|^p \big)^{\lambda_3}\qquad \\ 
   & \leq  \, \big(  U_1^p \,  |t-s|^{p \gamma_1}  \big)^{\lambda_1} \, \big(  U_2^p \,  |x-y|^{p \gamma_2} \big)^{\lambda_2} \,  \big( U_3^p \big)^{ \lambda_3} .
\end{equs}
 Here in the last step we have made use of the bounds  \eref{e:TBF}, \eref{e:SBF}, and \eref{e:DBF}. Similarly, according to the assumption \eref{e:TBF} and \eref{e:DBF} we get for the second term on the right-hand side of \eref{e:APB2} 
\begin{equ}[e:APB5]
 \E  \big|  \dFs (0) \big|^p  \, \ls   \, \big(  U_1^p\,  |t-s|^{p \gamma_1} \big)^{\lambda_1} \big(  U_3^p \big)^{\lambda_2 + \lambda_3} . 
\end{equ}
Therefore,  we get
\begin{equs}[e:APB6]
\E   \big|  \dFs     \big|_{\cC^{ \alpha_2}}^p  \,  \ls   \,     \big(   &U_1^{\lambda_1} \,  U_3^{ \lambda_3}  |t-s|^{ \gamma_1 \lambda_1}      \big)^p  \\
  &  \times \Big( U_3^{p\lambda_2 } +  U_2^{p \lambda_2}   \,  \int_{[-\pi, \pi]^2 }   \frac{ |x-y|^{p \gamma_2 \lambda_2}  }{|x-y|^{\alpha_2 p+  2 }}   \,   \, dx \, dy      \Big).  \qquad   
\end{equs}
The integral appearing in \eref{e:APB6} is finite if and only if  $\alpha_2$ satisfies the condition given in \eref{e:condal}. So in that case we get
\begin{equ}[e:APB8]
\E     \big|  \dFs    \big|_{\cC^{ \gamma_2}}^p   \,  \ls    \,  \big(   U_1^{\lambda_1} \,  U_3^{ \lambda_3}  |t-s|^{ \gamma_1 \lambda_1}     \big)^p     \big( U_2^{ \lambda_2}  + U_3^{\lambda_2}      \big)^p  .
\end{equ}
Then, to get uniform bounds in $s,t$ we apply the Garsia-Rodemich-Rumsey Lemma to the first term  on the right-hand side of \eref{e:APB1}
\begin{equs}
\E \bigg(   \sup_{0 \leq s , t \leq T}   \frac{\big|  \dFs    \big|_{\cC^{ \alpha_2}}}{|t-s|^{\alpha_1} }   \bigg)^p      &  \ls  \int_{[0,T]^2}  \frac{\E   \big|  \dFs    \big|_{\cC^{ \alpha_2}}^p }{|t-s|^{\alpha_1  p +  2 }} \, ds \, dt  \label{e:APB9}\\
&  \ls \,   \big(   U_1^{\lambda_1} \,  U_3^{ \lambda_3}   \big)^p    \big( U_2^{\lambda_2} + U_3^{ \lambda_2}    \big)^p    \int_{[0,T]^2}  \frac{   |t-s|^{ \gamma_1 \lambda_1 p}}{|t-s|^{\alpha_1  p +  2 }}    \, ds \, dt. 
\end{equs}
The integral appearing on the right-hand side of \eref{e:APB9} is finite if and only if $\alpha_1$ satisfies the condition \eref{e:condal}. Then we get 
\begin{equ}
\E \Big(   \sup_{0 \leq s , t \leq T}   \frac{\big|  \dFs    \big|_{\cC^{ \alpha_2}}}{|t-s|^{\alpha_1} }   \Big)^p   \, \ls \,  \big(   U_1^{\lambda_1} \,  U_3^{ \lambda_3} \big)^p      \big( U_2^{\lambda_2}  +     U_3^{ \lambda_2}    \big)^p . 
\end{equ}
Finally, to conclude it only remains to bound the term $ \E  \big| F(0,  \cdot)    \big|_{\cC^{ \alpha_2}}^p  $ on the right-hand side of \eref{e:APB1}. This can be done by observing that 
\begin{equs}
\E  \big| F(0, \cdot) \big|_{\cC^{\alpha_2}}^p   &\ls  \E  \big| F(0,0) \big|^p  \\ 
&\quad +   \int_{[-\pi, \pi]^2}   \frac{1}{|x-y|^{\alpha_2 p + 2 }} \,  \E   \big|  F(0,x) - F(0,y) \big|^p  \, dx \, dy 
\\ &\ls  U_3^p  +  U_3^{(\lambda_1 + \lambda_3)p  } \,  U_2^{\lambda_2 p} 
\ls U_3^{(\lambda_1 + \lambda_3)p  } \, (  U_2^{\lambda_2}+  U_3^{\lambda_2})^p .
\end{equs}
This finishes the proof of \eref{e:IB1}. 

The proof of \eref{e:IB2} is very similar and we only sketch it. As above we will use the notation 
\begin{equ}
\dRs(x,y)  \, = \, R(t;x,y) - R(s;x,y). 
\end{equ}
Similarly to \eref{e:APB1} we need to derive a bound on 
\begin{equ}[e:APB11]
\E  \big\|   R   \big\|_{\cC^{\alpha_1}_T (\bB^{ \alpha_2}) }^p   
 \ls   \E \Big(  \sup_{0 \leq s< t \leq T}   \frac{1}{|t-s|^{\alpha_1} }  \big|  \dRs    \big|_{ \alpha_2}\Big) ^p    +  \,    \E     \big|  R(0; \cdot,\cdot)    \big|_{ \alpha_2}^p  \, . 
\end{equ}
For fixed $s , t$ we get using Gubinelli's version of the Garsia-Rodemich-Rumsey inequality \ref{lem:GRR}
\begin{equs}[e:APB10]
\E   \big|  \dRs     \big|_{ \alpha_2}^p     \, \ls \,      \int_{[-\pi, \pi]^2 } &   \frac{\E\big[|\dRs(x,y)|^p  + | \delta \dRs |_{[x,y]}^p\big]}{|x-y|^{\alpha_2 p+  2}} 
\, dx \, dy  \;.
\end{equs}
The difference with respect to the case of $F$ is the appearance of the extra term $ |\delta R|_{[x.y]}$. On the other side there is no lower order term in the $\bB^{ \alpha_2}$ norm. Then using H\"older inequality and the bounds \eref{e:TBR}, \eref{e:SBR1}, \eref{e:SBR2} and \eref{e:DBR}, and then the integrability condition \eref{e:condal} in the same way as in \eref{e:APB3} and \eref{e:APB4}, the expectation in the right-hand side of \eref{e:APB10} can be bounded by 
\begin{equ}
 \big(  U_1^p \,  |t-s|^{p \gamma_1}  \big)^{\lambda_1} \, \big(  U_2^p \,  |x-y|^{p \gamma_2} \big)^{\lambda_2} \, \big(  U_3^p\,\big)^{\lambda_3}. 
\end{equ}
Plugging this back in \eref{e:APB11} we get as in \eref{e:APB9},
\begin{equs}
\E \bigg( &  \sup_{0 \leq s< t \leq T}   \frac{\big|  \dRs    \big|_{ \alpha_2}}{|t-s|^{\alpha_1} }  \bigg) ^p     \\
& \ls   
  \big(   U_1^{\lambda_1} \,  U_3^{ \lambda_3}  \big)^p     \cdot \big( U_2^{\lambda_2}  +     U_3^{ \lambda_2}    \big)^p    \int_{[0,T]^2}  \frac{|t-s|^{ \gamma_1 \lambda_1 p} }{|t-s|^{\alpha_1  p +  2 }}    \, ds \, dt \\
& \ls  \big(   U_1^{\lambda_1} \,  U_3^{ \lambda_3}   \big)^p     \cdot \big( U_2^{\lambda_2}  +  U_3^{ \lambda_2}    \big)^p. 
\end{equs}
Then applying Lemma \ref{lem:GRR} once more we can bound the second term appearing on the right-hand side of \eref{e:APB11} by 
\begin{equ}
 \E   \big|   R(0; \cdot,\cdot)    \big|_{ \alpha_2}^p   \,  \ls  \,  \big( U_2^{\lambda_2 } \,  U_3^{\lambda_1 + \lambda_3}  \, \big)^p.
 \end{equ}
This finishes the proof of \eref{e:IB2}.   
\end{proof}

In a similar spirit is the following Banach space-valued version of Kolmogorov's continuity criterion, which is slightly more convenient in some cases.

\begin{lemma}\label{lem:Kolmogorov}
Let $(\phi(t))_{t \in [0,T]}$ be a Banach space-valued random field having the property that for any $q \in (2,\infty)$ there exists a constant $K_q > 0$ such that
\begin{equation}\begin{aligned}\label{eq:equiv-mom}
  \big(\E \| \phi(t) \|^q \big)^{\frac1q} & \leq K_q 
  \big(\E \| \phi(t) \|^2 \big)^{\frac12}\;\;,\\
  \big(\E \| \phi(s) - \phi(t) \|^q \big)^{\frac1q} & \leq K_q 
  \big(\E \| \phi(s) - \phi(t) \|^2 \big)^{\frac12}\;,
\end{aligned}\end{equation}
for all $s, t \in [0,T]$. Furthermore, suppose that the estimate
\begin{align*}
 \E \| \phi(s) - \phi(t) \|^2 & \leq K_0 | s - t |^\delta
\end{align*}  
holds for some $K_0, \delta > 0$ and all $s,t \in [0,T]$. Then, for every $p  > 0$ there exists $C > 0$ such that
\begin{align*}
 \E \sup_{t \in [0,T]} \| \phi(t) \|^p
    \leq C \big( K_0 + \E\|\phi(0)\|^2 \big)^{\frac{p}{2}}\;.
\end{align*}
\end{lemma}

The conditions \eqref{eq:equiv-mom} are satisfied for random fields taking values in a fixed Gaussian chaos, as follows from the following well-known result, which is a consequence of the hypercontractivity of the Ornstein-Uhlenbeck semigroup due to Nelson \cite{Nel}. In order to formulate the result, we introduce some notation.

Let $(\Omega, \mathcal{F},\P)$ be a probability space, let $H$ be a separable Hilbert space, and let $\iota : H \to L^2(\Omega)$ be an isometry with the property that $\iota(h)$ is centered Gaussian for all $h \in H$. For a separable Banach space $E$, the $E$-valued Gaussian chaos of order $m \geq 0$ is defined as
\begin{align*}
 \mathcal{H}_m(E) := \overline{\mathrm{lin} \{ H_m( \iota(h)) \otimes x \ : \
  \| h \| = 1, \ x \in E \} }\;,
\end{align*}
where $H_m$ is the Hermite polynomial of degree $m$ and the closure is taken in $L^2(\Omega, \P ; E)$. We set $\mathcal{H}_{\leq m}(E) := \bigcup_{\ell = 0}^m \mathcal{H}_l(E)$. More information on  Banach space-valued Gaussian chaos can be found in \cite{FV10,KwaWoy02,Ma11}.
 
\begin{lemma}\label{lem:Nelson}
Let $1 \leq p < \infty$ and $m \geq 0$. For all $F \in  \mathcal{H}_{\leq m}(E)$ we have
\begin{align*}
 \| F \|_{L^p(\Omega, \P ; E)} \eqsim 
 \| F \|_{L^2(\Omega, \P ; E)}\;.
\end{align*}
\end{lemma}

\bibliographystyle{./Martin}
\bibliography{./burgers}

 \end{document}